\documentclass{article}

\usepackage{arxiv}

\usepackage[utf8]{inputenc} % allow utf-8 input
\usepackage[T1]{fontenc}    % use 8-bit T1 fonts
\usepackage[english]{babel}
\usepackage{hyphenat}
\usepackage{lmodern}
\usepackage{dsfont}
\usepackage{bold-extra}
\usepackage{mathtools}
\usepackage{mathrsfs}
\mathtoolsset{showonlyrefs=true,showmanualtags=true} % equations would be numbered which are referred to in the text %
\usepackage{latexsym}
\usepackage[%
    pdftex,
    bookmarks=true,
    pdfstartview={FitH}
]{hyperref}                 % hyperlinks
\usepackage{url}            % simple URL typesetting
\usepackage{booktabs}       % professional-quality tables
\usepackage{amsmath,amsfonts,amssymb,amsthm,amstext}       % AMS packages
\usepackage[nice]{nicefrac} % compact symbols for 1/2, etc.
\usepackage{microtype}      % microtypography
\usepackage{lipsum}		    % Can be removed after putting your text content
\usepackage{mathscinet}
\usepackage[inline]{enumitem}
\usepackage{graphicx}
\usepackage{float}

\usepackage[labelfont=bf]{caption}
\usepackage{subcaption}
\usepackage{color}
\usepackage[dvipsnames]{xcolor}
\usepackage{tikz}
\usepackage[nottoc]{tocbibind}
\usepackage[square,sort,comma,numbers,sectionbib]{natbib}
\usepackage{doi}
\usepackage{indentfirst}
\usepackage{orcidlink}       % ORCiD linking package
\numberwithin{equation}{section}

\theoremstyle{plain}
\newtheorem{theorem}{Theorem}[section]
\newtheorem*{theorem*}{Theorem}

\newtheorem{lemma}[theorem]{Lemma}
\newtheorem*{lemma*}{Lemma}
\newtheorem{remark}[theorem]{Remark}
\newtheorem*{remark*}{Remark}

\newtheorem*{definition*}{Definition}
\newtheorem{test}{Test}
\newtheorem*{test*}{Test}

\makeatletter
\renewenvironment{proof}[1][\proofname]{%
	\par\pushQED{\qed}\normalfont%
	\topsep6\p@\@plus6\p@\relax
	\trivlist\item[\hskip\labelsep\bfseries#1\@addpunct{.}]%
	\ignorespaces
}{%
	\popQED\endtrivlist\@endpefalse
}
\makeatother

\long\def\c{\mathrm{c}}
\long\def\M{\mathrm{M}}
\long\def\d{\mathrm{d}}
\long\def\e{\mathrm{e}}
\long\def\bigO{\mathcal{O}}
\long\def\L{\mathcal{L}}
\long\def\A{\mathcal{A}}
\long\def\I{\mathcal{I}}
\long\def\ie{\textit{i.e.}}
\long\def\eg{\textit{e.g.}}
\long\def\cf{\textit{cf.}}
\long\def\viz{\textit{viz.}}

\long\def\apriori{\textit{a priori}}

\def\articletitle{On Convergence of a Three-Layer Semi-Discrete Scheme for the Nonlinear Dynamic String Equation of Kirchhoff-Type with Time-Dependent Coefficients}
\def\articleauthorR{Jemal~Rogava}
\def\articleauthorV{Zurab~Vashakidze}

\title{\articletitle}

\date{} 					% Or removing it

\author{
        {\articleauthorR}\,\orcidlink{0000-0001-9460-4283}\\
	Faculty of Exact and Natural Sciences,\\
	Ivane Javakhishvili Tbilisi State University (TSU),\\
	Ilia Vekua Institute of Applied Mathematics (VIAM),\\
	2 University St., Tbilisi 0186, Georgia\\
	\href{mailto:jemal.rogava@tsu.ge}{\texttt{jemal.rogava@tsu.ge}}\\
	\And
	{\articleauthorV}\,\orcidlink{0000-0001-8736-6213}\\
	Institute of Mathematics, School of Science and Technology,\\
	The University of Georgia (UG),\\
	77a M. Kostava St., Tbilisi 0171, Georgia\\[6pt]
	Department of Numerical Mathematics and Modelling,\\
	Ilia Vekua Institute of Applied Mathematics (VIAM),\\
	Ivane Javakhishvili Tbilisi State University (TSU),\\
	2 University St., Tbilisi 0186, Georgia\\
	\href{mailto:zurab.vashakidze@gmail.com}{\texttt{zurab.vashakidze@gmail.com}}, \href{mailto:z.vashakidze@ug.edu.ge}{\texttt{z.vashakidze@ug.edu.ge}}
}

\renewcommand{\shorttitle}{Convergence of a Semi-Discrete Scheme for a Nonlinear Dynamic String Equation}

\hypersetup{
    pdftitle={\articletitle},
    pdfsubject={Numerical Analysis (math.NA); Analysis of PDEs (math.AP)},
    pdfauthor={\articleauthorR,~\articleauthorV},
    pdfkeywords={Non-linear Kirchhoff string equation, Cauchy problem, Three-layer semi-discrete scheme},
    hidelinks,
    colorlinks=true,
    linkcolor=blue,
    citecolor=red,
    filecolor=cyan,
    urlcolor=teal,
}
\allowdisplaybreaks

\begin{document}
\maketitle
\vfill
\begin{abstract}\label{abstract}
    This paper considers the Cauchy problem for the nonlinear dynamic string equation of Kirchhoff-type with time-varying coefficients. The objective of this work is to develop a time domain discretization algorithm capable of approximating a solution to this initial-boundary value problem. To this end, a symmetric three-layer semi-discrete scheme is employed with respect to the temporal variable, wherein the value of a nonlinear term is evaluated at the middle node point. This approach enables the numerical solutions per temporal step to be obtained by inverting the linear operators, yielding a system of second-order linear ordinary differential equations. Local convergence of the proposed scheme is established, and it achieves quadratic convergence regarding the step size of the discretization of time on the local temporal interval. We have conducted several numerical experiments using the proposed algorithm for various test problems to validate its performance. It can be said that the obtained numerical results are in accordance with the theoretical findings.
\end{abstract}
\vfill
% keywords
\keywords{{Non-linear Kirchhoff string equation} \and {Cauchy problem} \and {Three-layer semi-discrete scheme}.}

% MSC 2010
\msc{{65M06} \and {65N12} \and {65N22} \and {65Q30}.}

%\newpage
\section{Introduction}\label{sec:intro}
\subsection{Formulation of the Problem}\label{subsec:statement}
Consider the Cauchy problem for the nonlinear dynamic Kirchhoff string equation
\begin{subequations}
    \begin{gather}
        \frac{{\partial}^{2}u\left( x,t \right)}{\partial{t}^{2}} - \left( \alpha\left( t \right) + \beta\left( t \right)\int\limits_{0}^{\ell}{\left[ \frac{\partial u\left( x,t \right)}{\partial x} \right]}^{2}\d x \right)\frac{{\partial}^{2}u\left( x,t \right)}{\partial{x}^{2}} = f\left( x,t \right)\,,\quad \left( x,t \right)\in \left( 0,\ell \right)\times \left( 0,T \right]\,,\label{eq:main_eqt} \\
        u\left( x,0 \right) = {\psi}_{0}\left( x \right)\,,\quad {u}_{t}^{\prime}\left( x,0 \right) = {\psi}_{1}\left( x \right)\,,\label{eq:initial_conds} \\
        u\left( 0,t \right) = 0\,,\quad u\left( \ell,t \right) = 0\,.\label{eq:boundary_conds}
    \end{gather}
\end{subequations}
Suppose that $\alpha\left( t \right) \geq {\c}_{0} > 0$ and $\beta\left( t \right) \geq {\c}_{1} > 0$, where $\alpha\left( t \right)$ and $\beta\left( t \right)$ are continuously differentiable functions. Further, let $f\left( x,t \right)$ be a continuous function, $u\left( x,t \right)$ an unknown function, and ${\psi}_{0}\left( x \right)$ and ${\psi}_{1}\left( x \right)$ continuous functions. It is also assumed that compatibility conditions are provided for the function ${\psi}_{0}\left( x \right)$, {\ie} ${\psi}_{0}\left( 0 \right) = {\psi}_{0}\left( \ell \right) = 0$.

\subsection{Historical Background, Physical Relevance, and Applications of the Problem}\label{subsec:his_bgd}

The nonlinear dynamic Kirchhoff string equation is a hyperbolic partial differential equation that accurately models the behaviour of an elastic string of length $\ell$,
\begin{equation}\label{eq:intro_kirchhoff_equation}
    {u}_{tt}\left( x,t \right) = \left( {\alpha} + {\beta}\int\limits_{0}^{\ell}{u}_{x}^{2}\left( x,t \right)\,\d{x} \right){u}_{xx}\left( x,t \right)\,,\quad {\alpha} > 0\,,\quad {\beta} > 0\,.\tag{K}
\end{equation}
This equation \eqref{eq:intro_kirchhoff_equation} was first derived by the prominent German physicist Gustav Kirchhoff in the mid-nineteenth century (see \cite{kirchhoff1876vorlesungen}). It serves as a prototypical example of a nonlinear wave equation, formulated as an extension of d'Alembert's wave equation for free vibrations of elastic strings. It considers factors such as tension, as well as external forces like gravity and air resistance. The solution of the equation provides the displacement of the string at any given point in time, and its application extends to a multitude of physical systems.

Here, in the context of the equation \eqref{eq:main_eqt}, $u\left( x,t \right)$ represents the displacement of the string from its equilibrium position at a specific point $x$ and $t$. The external force density is denoted by $f\left( x,t \right)$, while the tension in the string is expressed as $\alpha\left( t \right) + \beta\left( t \right)\int_{0}^{\ell}{u}_{x}^{2}\left( x,t \right)\,\d{x}$. This tension includes material parameters $\alpha\left( t \right)$ and $\beta\left( t \right)$, which may exhibit time dependence in our model. Notably, the key feature of the nonlinear Kirchhoff equation is that the tension is not constant and it varies nonlinearly with the spatial gradient of $u\left( x,t \right)$, reflecting variations in the tension due to geometric bending. The quantity $\alpha\left( t \right)$ represents the base tension of the string unaffected by displacement, while $\beta\left( t \right)$ scales the nonlinear contribution to the tension.

Under the assumption of time-independent material parameters, the range of applications spans from the classical representation of oscillations of thin beams, as originally formulated by Kirchhoff \cite{kirchhoff1876vorlesungen}, to contemporary developments in the realm of nanomechanical oscillators \cite{Yao2022}. Here, micro- or nanomechanical oscillators serve as building blocks for quantum transducers \cite{Lauk_2020}, for ultra-precise sensing devices \cite{Westerveld2021}, or as central elements in state-of-the-art quantum optics experiments with important examples being optomechanical quantum teleportation \cite{Fiaschi2021} and laser cooling \cite{Chan2011}.

Allowing the material parameters $\alpha$ and $\beta$ to become time-dependent enriches the dynamics and paves the way for numerous important applications. If $\alpha$ and $\beta$ are oscillating, they can describe vibrating systems where the vibration has a time-dependent influence on the material properties. Such oscillations might arise in classical applications when strings or cables are exposed to oscillating tensions from the operational vibrations of some machinery. In the case of nanomechanical devices, they can arise due to thermal vibrations or periodic excitations of the oscillator \cite{Samanta2023}. Furthermore, they can appear in quantum transducers, when the driving forces are strong enough to alter the material properties of an oscillator \cite{Sun2006}. Time-varying material constants also allow for the description of environmental changes such as temperature, pressure, or humidity. They could furthermore be used to model beams made out of smart materials whose properties can be modified in response to external events, {\eg} in order to better withstand earthquakes \cite{Song2006}.

\subsection{Survey of Analytical and Numerical Frameworks of the Problem}\label{subsec:analyt_and_num_frmwork}

Despite its antiquity, the Kirchhoff dynamic string equation remains a valuable tool in the study of wave motion and is extensively used in the fields of physics, engineering, and related sciences. Numerous distinguished scientists have contributed to the study of the equation and its various modifications. However, obtaining an exact solution to the equation remains a challenging problem, and numerical methods are commonly employed to approximate the solution. The equation has been extensively investigated using both numerical and analytical methods.

The conditions governing the existence of solutions for equation \eqref{eq:intro_kirchhoff_equation} were first examined by S. Bernstein in 1940 \cite{bernstein1940}. Bernstein's work focused on the scenario where the initial data are $2\pi$-periodic analytic functions on $\mathbb{R}$. Poho\v{z}aev extended Bernstein's results in 1975 \cite{pohozaev1975} to encompass the multidimensional case regarding spatial variables. In 1978, Lions presented an abstract framework for the Kirchhoff equation \eqref{eq:intro_kirchhoff_equation} in a publication that significantly contributed to its popularity \cite{lions1978}. Since then, many researchers have studied the solvability issues of the classical and modified Kirchhoff static and dynamic equations in general function spaces and developed a qualitative approach. These scholars include Alves, Corr\^{e}a, and Ma \cite{alves2005}; Arosio and Panizzi \cite{arosio1996}; Berselli and Manfrin \cite{bereselli2000}; D'Ancona and Spagnolo \cite{dancona1993,dancona1994,dancona1995}; Greenberg and Hu \cite{greenberg198081}; Kajitani \cite{kajitani2009}; Liu and Rincon \cite{liu2003}; Ma \cite{MA2005e1967}; Manfrin \cite{manfrin2002}; Matos \cite{matos1991}; Medeiros \cite{medeiros1979}; Nishihara \cite{nishihara1984}; and Rzymowski \cite{rzymowski2002}.

Over the years, many numerical methods have been developed to address the boundary and initial-boundary value problems arising from both the classical and modified Kirchhoff's string equations. These methods encompass a variety of techniques, including but not limited to finite difference, spectral, and finite element methods. Moreover, several of these methods have been combined with other numerical schemes, resulting in even more sophisticated approaches. Christie and Sanz-Serna \cite{christie1984} applied the finite element method to discretize the Kirchhoff equation \eqref{eq:intro_kirchhoff_equation}. They utilized predictor-corrector methods based on Crank-Nicolson-type schemes for time integration. Gudi \cite{gudi2012} investigated a finite element method for a second-order nonlinear elliptic boundary value problem of the stationary version of the generalized Kirchhoff equation. Here, the finite element system is replaced with an equivalent sparse system for which the Jacobian of the Newton-Raphson iterative method is sparse. Liu and Rincon \cite{liu2003} solved the Kirchhoff equation for nonlinear elastic strings with moving boundaries using the explicit Lax-Friedrichs difference scheme. Peradze \cite{peradze2005} developed an algorithm for the numerical solution of the dynamic Kirchhoff string equation, which utilizes the Galerkin method, the modified Crank-Nicolson difference scheme, and a Picard-type iteration process. In \cite{peradze2009}, Peradze proposed a numerical algorithm based on the projection method and the finite difference method for approximating the Kirchhoff wave equation with respect to spatial and time variables. In \cite{ren2017}, Q. Ren and H. Tian developed a numerical scheme for obtaining an approximate solution to a nonlocal stationary analogue of the Kirchhoff equation using the Legendre-Galerkin spectral method, which reduces the problem to a nonlinear finite-dimensional system, followed by an iterative solution. In all of the works considered, numerical algorithms are designed to combine the reduction of the original problem to a finite-dimensional system of nonlinear equations with an iterative process for finding an approximate solution to the obtained system.

The works authored by J. Rogava and M. Tsiklauri \cite{RogavaTsiklauri_LocConvg2012,RogavaTsiklauri_EvolEqt2014} are primarily concerned with the development and analysis of a symmetric three-layer semi-discrete scheme for solving the Cauchy problem associated with the abstract generalization of the dynamic Kirchhoff equation. In the proposed scheme, the value of a nonlinear term is evaluated at the middle node point, resulting in the transformation of the original problem into a linear one for each temporal layer. In \cite{vashakidze2020,vashakidze2022}, the same approach is extended through the combination of the Legendre-Galerkin spectral method for numerically solving the spatial one-dimensional nonlinear dynamic Kirchhoff string equation. This choice leads to the reduction of the stated nonlinear problem to a system of linear equations per temporal layer, with the corresponding coefficient matrix being sparse and featuring nonzero elements exclusively on the main diagonal, the second diagonal below it, and the second diagonal above it. In the article \cite{RogTsikVash2023}, a semi-discrete scheme is considered for an abstract analogue of the J. Ball beam integro-differential equation, regarding the general nonlinear case. The stability and convergence of this scheme are studied. The book authored by T. Jangveladze, Z. Kiguradze, and B. Neta \cite{JangKigNeta2016} addresses the design of algorithms for finding numerical solutions and conducting investigations regarding initial-boundary value problems for specific classes of integro-differential equations.

\subsection{Outline of the Problem}\label{subsec:out_prob}

The present study addresses the initial-boundary value problem \eqref{eq:main_eqt}-\eqref{eq:boundary_conds} pertaining to the nonlinear dynamic Kirchhoff-type string equation with time-dependent coefficients. Our objective is to obtain an approximate solution for this problem by employing a symmetric three-layer semi-discrete scheme with respect to the temporal variable. Notably, in this scheme, the nonlinear term is taken at the middle node point ({\cf} \cite{RogavaTsiklauri_LocConvg2012,RogavaTsiklauri_EvolEqt2014,vashakidze2020,vashakidze2022}). By applying this approach, the stated hyperbolic nonlinear partial differential equation can be reduced to a system of linear ordinary differential equations of second order. We complete an investigation into the convergence of the proposed semi-discrete scheme for approximating the solution to the problem. Our research demonstrates that the considered scheme attains quadratic convergence regarding the step size of the discretization of time on the local temporal interval.

From the standpoint of numerical implementation, a fourth-order accuracy finite-difference scheme is developed regarding the spatial variable, resulting in obtaining a tridiagonal system of linear equations at each time step. Furthermore, the coefficient matrices of these systems reveal strict diagonal dominance. The condition numbers of the coefficient matrices of these tridiagonal systems are estimated. Based on this estimate, a relation between the lengths of spatial and temporal grids is established, ensuring that the condition numbers remain within the practically accepted range. To validate the efficacy of the scheme, we subjected it to four carefully chosen test problems. In the first three tests, exact solutions are known, exhibiting oscillatory behaviour. Numerical experiments show that the order of convergence of the scheme aligns with the theoretical results. The proposed algorithm accurately describes the behaviour of the oscillating solutions. Notably, In the fourth test, we consider a homogeneous version of the Kirchhoff string equation \eqref{eq:main_eqt} with fixed coefficients and an unknown analytical solution. However, we propose an approach that allows us to find an experimental solution to this test problem. All results and numerical computations are detailed in the \hyperref[sec:numresult]{final section}. Furthermore, the source code for the proposed algorithm, implemented in GNU Octave, is available in the \href{https://github.com/zv1991/Semi-Discrete-Scheme-for-Kirchhoff-Eqn}{GitHub repository} and on \href{https://doi.org/10.5281/zenodo.10137813}{Zenodo} (see \cite{rogava_2023_10137922}).

\section{Description of the Time Domain Discretization Algorithm}\label{subsec:descr_algor}

Let the temporal interval $\left[ 0,T \right]$ be divided into equal sub-intervals ${\left[ {t}_{i - 1},{t}_{i} \right]}_{i = 1}^{n}$ using a uniform grid with step size $\tau$, \ie
\begin{equation*}
    0 = {t}_{0} < {t}_{1} < \cdots < {t}_{n - 1} < {t}_{n} = T\,,\quad {t}_{k} = {k}{\tau}\,,\quad {k} = 0, 1, \ldots, n\,,\quad {\tau} = \frac{T}{n}\,.
\end{equation*}
The equation given in \eqref{eq:main_eqt} is transformed into the following form at discrete time points $t = {t}_{k}$ ($k = 1, 2, \ldots, {n - 1}$):
\begin{equation}\label{eq:discrete_main_eqt}
    \frac{{\Delta}^{2}u\left( x,{t}_{k - 1} \right)}{{\tau}^{2}} - \frac{1}{2}q\left( {t}_{k} \right)\left( \frac{\d ^{2}u\left( x,{t}_{k + 1} \right)}{\d {x}^{2}} + \frac{\d ^{2}u\left( x,{t}_{k - 1} \right)}{\d {x}^{2}} \right) = f\left( x, {t}_{k} \right) + {R}_{1,k}\left( x,\tau \right) + {R}_{2,k}\left( x,\tau \right)\,,
\end{equation}
where
\begin{align*}
    {\Delta}u\left( x,{t}_{k} \right) &= u\left( x,{t}_{k + 1} \right) - u\left( x,{t}_{k} \right)\,,\\
    q\left( t \right) &= \alpha\left( t \right) + \beta\left( t \right)\int\limits_{0}^{\ell}{\left[ \frac{\partial u\left( x,t \right)}{\partial x} \right]}^{2}\d x\,,\\
    {R}_{1,k}\left( x,\tau \right) &= \frac{{\Delta}^{2}u\left( x,{t}_{k - 1} \right)}{{\tau}^{2}} - \frac{{\partial}^{2}u\left( x,{t}_{k} \right)}{\partial{t}^{2}}\,,\\
    {R}_{2,k}\left( x,\tau \right) &= -\frac{1}{2}q\left( {t}_{k} \right)\frac{\d ^{2}}{\d {x}^{2}}{\Delta}^{2}u\left( x,{t}_{k - 1} \right)\,.
\end{align*}
It is shown that ${R}_{1,k}\left( x,\tau \right) = \bigO\left( {\tau}^{2} \right)$ and ${R}_{2,k}\left( x,\tau \right) = \bigO\left( {\tau}^{2} \right)$ (see the \hyperref[theorem:theorem1]{\bf Theorem \ref*{theorem:theorem1}}). By neglecting the remainder terms ${R}_{1,k}\left( x,\tau \right)$ and ${R}_{2,k}\left( x,\tau \right)$ in equation \eqref{eq:discrete_main_eqt}, the following symmetric three-layer semi-discrete scheme is obtained
\begin{equation}\label{eq:semidiscrete_scheme}
    \frac{{\Delta}^{2}{u}_{k - 1}\left( x \right)}{{\tau}^{2}} - \frac{1}{2}{q}_{k}\left( \frac{\d ^{2}{u}_{k + 1}\left( x \right)}{\d {x}^{2}} + \frac{\d ^{2}{u}_{k - 1}\left( x \right)}{\d {x}^{2}} \right) = {f}_{k}\left( x \right)\,,\quad k = 1, 2, \ldots, {n - 1}\,,
\end{equation}
where ${f}_{k}\left( x \right) = f\left( x, {t}_{k} \right)$ and
\begin{equation}\label{eq:numres_qintegr}
    {q}_{k} = {\alpha}_{k} + {\beta}_{k}\int\limits_{0}^{\ell}{\left( \frac{\d {u}_{k}\left( x \right)}{\d x} \right)}^{2}\d x\,,\quad {\alpha}_{k} = \alpha\left( {t}_{k} \right)~\text{and}~{\beta}_{k} = \beta\left( {t}_{k} \right)\,.
\end{equation}
In the following, the solution to the obtained differential-difference equations \eqref{eq:semidiscrete_scheme} are denoted by ${u}_{k}\left( x \right)$ and are considered as approximate solutions to the problem \eqref{eq:main_eqt}-\eqref{eq:boundary_conds} at the time instants $t = {t}_{k}$, thus $u\left( x,{t}_{k} \right) \approx {u}_{k}\left( x \right)$.

We now express equation \eqref{eq:semidiscrete_scheme} in an alternative form
\begin{equation}\label{eq:discrt_operator_eqt}
    \left( 2{\I} - {\tau}^{2}{q}_{k}\frac{\d ^{2}}{\d {x}^{2}} \right){u}_{k + 1}\left( x \right) = {g}_{k}\left( x \right)\,,\quad k = 1, 2, \ldots, {n - 1}\,,
\end{equation}
where
\begin{equation*}
    {g}_{k}\left( x \right) = {2}\left( {\tau}^{2}{f}_{k}\left( x \right) + {2}{u}_{k}\left( x \right) \right) - \left( 2{\I} - {\tau}^{2}{q}_{k}\frac{\d ^{2}}{\d {x}^{2}} \right){u}_{k - 1}\left( x \right)\,.
\end{equation*}
For each discrete time layer, the boundary conditions are rewritten as follows:
\begin{equation}\label{eq:operator_diff_eq_bound_cond}
    {u}_{k + 1}\left( 0 \right) = 0\,,\quad {u}_{k + 1}\left( \ell \right) = 0\,.
\end{equation}
The values of the unknown functions at the zeroth and first layers are determined by the initial conditions given in \eqref{eq:initial_conds} and the equation \eqref{eq:main_eqt},
\begin{align}
    {u}_{0}\left( x \right) &= {\psi}_{0}\left( x \right)\,,\label{eq:numres_zerothlayr} \\
    {u}_{1}\left( x \right) &= {\psi}_{0}\left( x \right) + {\tau}{\psi}_{1}\left( x \right) + \frac{{\tau}^{2}}{2}{\psi}_{2}\left( x \right)\,,\quad {\psi}_{2}\left( x \right) = {f}_{0}\left( x \right) + {q}_{0}\frac{\d ^{2}{\psi}_{0}\left( x \right)}{\d {x}^{2}}\,.\label{eq:semidscrete_scheme_first_layer}
\end{align}

Let us define the notation for a second-order differential operator
\begin{equation}\label{eq:second_order_diff_operator}
    {\L}_{0} = -\frac{\d ^{2}}{\d {x}^{2}}\,,\quad D\left( {\L}_{0} \right) = \left\{ u\left( x \right) \in {C}^{2}\left( \left[ 0,\ell \right] \right) \mid u\left( 0 \right) = u\left( \ell \right) = 0 \right\}\,.
\end{equation}
With the use of the notation \eqref{eq:second_order_diff_operator}, equation \eqref{eq:semidiscrete_scheme} can be expressed as follows
\begin{equation}\label{eq:semidiscrete_scheme_operator_form}
    \frac{{\Delta}^{2}{u}_{k - 1}\left( x \right)}{{\tau}^{2}} + \frac{1}{2}{q}_{k}\left( {\L}_{0}{u}_{k + 1}\left( x \right) + {\L}_{0}{u}_{k - 1}\left( x \right) \right) = {f}_{k}\left( x \right)\,,\quad k = 1, 2, \ldots, {n - 1}\,,
\end{equation}
it can be easily deduced that through the integration by parts ${q}_{k}$ can be expressed in the following form using the notation \eqref{eq:second_order_diff_operator}, \ie
\begin{equation*}
    {q}_{k} = {\alpha}_{k} + {\beta}_{k}\left( {\L}_{0}{u}_{k}, {u}_{k} \right)\,,\quad {\alpha}_{k} = \alpha\left( {t}_{k} \right)~\text{and}~{\beta}_{k} = \beta\left( {t}_{k} \right)\,.
\end{equation*}
Throughout the text and in the subsequent discussions, $\left( {\cdot},{\cdot} \right)$ represents the inner product in ${L}_{2}\left( 0,\ell \right)$ and the associated norm is denoted by $\left\| {\cdot} \right\|$.

The extension of the operator ${\L}_{0}$ to a self-adjoint one is denoted by ${\L}$ (${\L}_{0} \subset {\L}$). By using this operator ${\L}$, the equation \eqref{eq:semidiscrete_scheme_operator_form} can be rewritten in the following manner, {\ie}
\begin{equation}\label{eq:semidiscrete_scheme_extension_operator_form}
    \frac{{\Delta}^{2}{u}_{k - 1}\left( x \right)}{{\tau}^{2}} + \frac{1}{2}{{\tilde q}_{k}}\left( {\L}{u}_{k + 1}\left( x \right) + {\L}{u}_{k - 1}\left( x \right) \right) = {f}_{k}\left( x \right)\,,\quad k = 1, 2, \ldots, {n - 1}\,,
\end{equation}
where
\begin{equation*}
    {\tilde q}_{k} = {\alpha}_{k} + {\beta}_{k}{\left\| {\L}^{\nicefrac{1}{2}} {u}_{k} \right\|}^{2}\,,\quad {\alpha}_{k} = \alpha\left( {t}_{k} \right)~\text{and}~{\beta}_{k} = \beta\left( {t}_{k} \right)\,.
\end{equation*}

\section{Error Estimation of the Approximate Solution Resulting from the Time Domain Discretization}\label{sec:error_estimate}
\subsection{Auxiliary Lemmata}\label{subsec:auxiliary_lemmata}
Prior to evaluating the error of the approximate solution of problem \eqref{eq:main_eqt}-\eqref{eq:boundary_conds}, it is necessary to consider a simple but important lemma that is crucial to fulfilling our objective. This lemma serves as a discrete analogue of the well-known {Gr\"{o}nwall}-type inequality. Numerous authors have considered various variations of this lemma ({\cf} \cite{BookDiffElaydi2005,emmrich1999}), and we present one such version.
\begin{lemma}[\textbf{Discrete {Gr\"{o}nwall}-type inequality}]\label{lemma:gronwall-inequality}
    Let $\left\{ {\varepsilon}_{k} \right\}$, $\left\{ {a}_{k} \right\}$ and $\left\{ {h}_{k} \right\}$ be sequences of nonnegative real numbers that satisfy the following inequality
    \begin{equation}\label{eq:aux_lemma_gronwall_main_inqt}
        {\varepsilon}_{k + 1} \leq \sum\limits_{i = 1}^{k}{{a}_{i}{\varepsilon}_{i}} + \sum\limits_{i = 0}^{k}{{h}_{i}}\,,
    \end{equation}
    then
    \begin{equation}\label{eq:aux_lemma_gronwall_inqt_to_proof}
        {\varepsilon}_{k + 1} \leq \sum\limits_{i = 1}^{k}{{\alpha}_{i,k}{h}_{i - 1}} + {h}_{k}\,,\quad {\alpha}_{i,k} = {\left( 1 + {a}_{i} \right)}{\cdots}{\left( 1 + {a}_{k} \right)}\,,\quad {i}\leq{k}\,.
    \end{equation}
\end{lemma}
\begin{proof}
    In order to prove this lemma, mathematical induction is applied. Before proceeding to the proof, it is advisable to note that ${\alpha}_{i,k} = {\alpha}_{i, i}\cdots{\alpha}_{k,k} = {\alpha}_{i,j}{\alpha}_{j + 1,k}$, ${1} \leq {i} \leq {j} \leq {k - 1}$. The validity of the statement \eqref{eq:aux_lemma_gronwall_inqt_to_proof} for ${k} = {0}$ is self-evident. This follows from \eqref{eq:aux_lemma_gronwall_main_inqt}, which provides evidence for the assertion that ${\varepsilon}_{1} \leq {h}_{0}$.
    
    Using \eqref{eq:aux_lemma_gronwall_main_inqt} for ${k} = {1}$, the following inequality can be derived through the application of the previously established inequality
    \begin{equation*}
        {\varepsilon}_{2} \leq {a}_{1}{\varepsilon}_{1} + {h}_{0} + {h}_{1} \leq \left( 1 + {a}_{1} \right){h}_{0} + {h}_{1} = {\alpha}_{1,1}{h}_{0} + {h}_{1}\,.
    \end{equation*}
    % In accordance with the preceding cases, it is possible to derive the following conclusion for ${k} = {2}$,
    % \begin{align*}
    %     {\varepsilon}_{3} &\leq {a}_{1}{\varepsilon}_{1} + {a}_{2}{\varepsilon}_{2} + {h}_{0} + {h}_{1} + {h}_{2} \leq {a}_{1}{h}_{0} + {a}_{2}\left( {\alpha}_{1,1}{h}_{0} + {h}_{1} \right) + {h}_{0} + {h}_{1} + {h}_{2} \\
    %     &= {a}_{2}{\alpha}_{1,1}{h}_{0} + {a}_{1}{h}_{0} + {a}_{2}{h}_{1} + {h}_{0} + {h}_{1} + {h}_{2} \\
    %     &= {a}_{2}{\alpha}_{1,1}{h}_{0} + \left( 1 + {a}_{1} \right){h}_{0} + \left( 1 + {a}_{2} \right){h}_{1} + {h}_{2} \\
    %     &= {a}_{2}{\alpha}_{1,1}{h}_{0} + {\alpha}_{1,1}{h}_{0} + {\alpha}_{2,2}{h}_{1} + {h}_{2} = \left( {\alpha}_{1,1} + {a}_{2}{\alpha}_{1,1} \right){h}_{0} + {\alpha}_{2,2}{h}_{1} + {h}_{2} \\
    %     &= {\alpha}_{1,2}{h}_{0} + {\alpha}_{2,2}{h}_{1} + {h}_{2}\,.
    % \end{align*}
    Let us consider the assumption that the inequality \eqref{eq:aux_lemma_gronwall_inqt_to_proof} holds true for any ${k} \leq {m - 1}$ (${m} \geq {1}$). Our objective is to demonstrate that \eqref{eq:aux_lemma_gronwall_inqt_to_proof} is also valid for ${k} = {m}$. As per \eqref{eq:aux_lemma_gronwall_main_inqt}, it follows that
    \begin{align*}
        {\varepsilon}_{m + 1} &\leq \sum\limits_{i = 1}^{m}{{a}_{i}{\varepsilon}_{i}} + \sum\limits_{i = 0}^{m}{{h}_{i}} = {a}_{1}{\varepsilon}_{1} + \sum\limits_{i = 2}^{m}{{a}_{i}{\varepsilon}_{i}} + \sum\limits_{i = 1}^{m}{{h}_{i - 1}} + {h}_{m} \\
        &\leq {a}_{1}{h}_{0} + \sum\limits_{i = 2}^{m}{{a}_{i}\left( \sum\limits_{j = 1}^{i - 1}{{\alpha}_{j,i - 1}{h}_{j - 1}} + {h}_{i - 1} \right)} + \sum\limits_{i = 1}^{m}{{h}_{i - 1}} + {h}_{m} \\
        &= \sum\limits_{i = 2}^{m}{\sum\limits_{j = 1}^{i - 1}{{a}_{i}{\alpha}_{j,i - 1}{h}_{j - 1}}} + \sum\limits_{i = 1}^{m}{{\alpha}_{i,i}{h}_{i - 1}} + {h}_{m} \\
        &= \sum\limits_{j = 1}^{m - 1}{\sum\limits_{i = j + 1}^{m}{{a}_{i}{\alpha}_{j,i - 1}{h}_{j - 1}}} + \sum\limits_{j = 1}^{m - 1}{{\alpha}_{j,j}{h}_{j - 1}} + {\alpha}_{m,m}{h}_{m - 1} + {h}_{m} \\
        &= \sum\limits_{j = 1}^{m - 1}{\left( {\alpha}_{j,j} + \sum\limits_{i = j + 1}^{m}{{a}_{i}{\alpha}_{j,i - 1}} \right){h}_{j - 1}} + {\alpha}_{m,m}{h}_{m - 1} + {h}_{m} \\
        &= \sum\limits_{j = 1}^{m - 1}{\left( 1 + {a}_{m} \right){\alpha}_{j,m - 1}{h}_{j - 1}} + {\alpha}_{m,m}{h}_{m - 1} + {h}_{m} \\
        &= \sum\limits_{j = 1}^{m}{{\alpha}_{j,m}{h}_{j - 1}} + {h}_{m}\,.
    \end{align*}
    % \begin{align*}
    %     \sum\limits_{i = 2}^{m}{\sum\limits_{j = 1}^{i - 1}}{{x}_{i,j}} = &{x}_{2,1} \\
    %     + &{x}_{3,1} + {x}_{3,2} \\
    %     + &{x}_{4,1} + {x}_{4,2} + {x}_{4,3} \\
    %     + &{x}_{5,1} + {x}_{5,2} + {x}_{5,3} + {x}_{5,4} \\
    %     + &\ldots \\
    %     + &{x}_{m,1} + {x}_{m,2} + {x}_{m,3} + {x}_{m,4} + {x}_{m,4} + \ldots + {x}_{m,m - 1} \\
    %     = &\sum\limits_{j = 1}^{m - 1}{\sum\limits_{i = j + 1}^{m}}{{x}_{i,j}}\,.
    % \end{align*}
    Ultimately, the following conclusion can be drawn
    \begin{equation*}
        {\varepsilon}_{m + 1} \leq \sum\limits_{i = 1}^{m}{{\alpha}_{i,m}{h}_{i - 1}} + {h}_{m}\,.
    \end{equation*}
    Hence, through the application of mathematical induction, it can be established that the inequality \eqref{eq:aux_lemma_gronwall_inqt_to_proof} holds for all values of ${k} \geq {0}$.
\end{proof}
\begin{remark}\label{rmk:remark1}
    The application of the well-established inequality between the arithmetic and geometric means (the \uppercase{am}-\uppercase{gm} inequality) allows us to arrive at the following conclusion
    \begin{align*}
        {\alpha}_{i,k} &= \prod\limits_{j = i}^{k}{\left( 1 + {a}_{j} \right)} \leq {\left( {1} + \frac{1}{k - i + 1}\sum\limits_{j = i}^{k}{{a}_{j}} \right)}^{k - i + 1} \leq \exp{\left( \sum\limits_{j = i}^{k}{{a}_{j}} \right)} \leq {\e}^{{\nu}_{k}}\,,\quad {\nu}_{k} = \sum\limits_{j = 1}^{k}{{a}_{j}}\,.
    \end{align*}
    It is worth mentioning that if the series $\displaystyle \sum\limits_{j = 1}^{\infty}{{a}_{j}}$ converges, then the following estimate holds true
    \begin{equation*}
        {\alpha}_{i,k} \leq {\e}^{{\nu}}\,,\quad {\nu} = \sum\limits_{j = 1}^{\infty}{{a}_{j}}\,.
    \end{equation*}
\end{remark}

To maintain the self-contained nature of this paper, we state the following lemma along with a reference to its proof (see \cite[\textbf{Lemma 3.2}]{RogavaTsiklauri_LocConvg2012}).
% We formulate the following lemma (a detailed proof of this lemma can be found in \cite[\textbf{Lemma 3.2}]{RogavaTsiklauri_LocConvg2012}) while maintaining this paper self-contained.
\begin{lemma}[see the \textbf{Lemma 3.2} in \cite{RogavaTsiklauri_LocConvg2012}]\label{lemma:rogava-tsiklauri}
    Let the sequences of nonnegative numbers, ${\left\{ {\alpha}_{k} \right\}}_{k = 0}^{n}$ and ${\left\{ {c}_{k} \right\}}_{k = 0}^{n}$, be such that they fulfill the following inequality:
    \begin{equation*}
        {\alpha}_{k + 1} \leq {\alpha}_{k}\left( 1 + {\tau}{\alpha}_{k}^{s} \right) + {\tau}{c}_{k}\,,
    \end{equation*}
    where ${s} > 0$ and ${\tau} > 0$.

    Therefore, the estimate is valid
    \begin{equation*}
        {\alpha}_{k} \leq \frac{\alpha}{{\left( 1 - {s}{\alpha}^{s}{t}_{k}{a}_{k} \right)}^{\nicefrac{1}{s}}}\,,\quad {t}_{k} = {k}{\tau} < \frac{1}{{s}{\alpha}^{s}{a}_{k}}\,,\quad {\alpha} = \max\left( 1,{\alpha}_{0} \right)\,,\quad {a}_{k} = 1 + \max\limits_{{0} \leq {i} \leq {k}}{\left( {c}_{i} \right)}\,.
    \end{equation*}
\end{lemma}

\subsection{Main Lemmata}\label{subsec:main_lemmata}
It is noteworthy that, throughout the text, the following letters ${\c}$ and ${\M}$ enumerated with lower indices represent positive constants.
\begin{lemma}\label{lemma:lemma1}
    The sequences of functions $\left( {u}_{k}\left( x \right) - {u}_{k - 1}\left( x \right) \right) {/} {\tau}$ and ${\L}^{\nicefrac{1}{2}}{u}_{k}\left( x \right)$ are uniformly bounded with respect to the ${L}_{2}$-norm, {\ie}, there exist constants ${\M}_{1}$ and ${\M}_{2}$, which are independent of $\tau$, such that the following inequalities hold:
    \begin{equation*}
        \left\| \frac{{u}_{k} - {u}_{k - 1}}{\tau} \right\| \leq {\M}_{1}\,,\quad \left\| {\L}^{\nicefrac{1}{2}}{u}_{k} \right\| \leq {\M}_{2}\,,\quad {k} = 1, 2, \ldots, n\,.
    \end{equation*}
\end{lemma}
\begin{proof}
    By taking the inner product of both sides of equation \eqref{eq:semidiscrete_scheme_extension_operator_form} with ${u}_{k + 1} - {u}_{k - 1} = \Delta{u}_{k} + \Delta{u}_{k - 1}$ and applying integration by parts, we can obtain
    \begin{equation}\label{eq:lemma1_inner_product_eqt}
        {\left\| \frac{\Delta{u}_{k}}{\tau} \right\|}^{2} + \frac{1}{2}{{\tilde q}_{k}}{\left\| {\L}^{\nicefrac{1}{2}}{u}_{k + 1} \right\|}^{2} = {\left\| \frac{\Delta{u}_{k - 1}}{\tau} \right\|}^{2} + \frac{1}{2}{{\tilde q}_{k}}{\left\| {\L}^{\nicefrac{1}{2}}{u}_{k - 1} \right\|}^{2} + \left( {f}_{k},\Delta{u}_{k} \right) + \left( {f}_{k},\Delta{u}_{k - 1} \right)\,,
    \end{equation}
    recall that,
    \begin{gather*}
        \Delta{u}_{k} = {u}_{k + 1} - {u}_{k}\,,\\
        {\tilde q}_{k} = {\alpha}_{k} + {\beta}_{k}{\left\| {\L}^{\nicefrac{1}{2}} {u}_{k} \right\|}^{2}\,,\quad {\alpha}_{k} = \alpha\left( {t}_{k} \right)~\text{and}~{\beta}_{k} = \beta\left( {t}_{k} \right)\,.
    \end{gather*}
    Let us denote
    \begin{equation*}
        {\mu}_{k} = {\left\| \frac{\Delta{u}_{k - 1}}{\tau} \right\|}^{2}\,,\quad {\gamma}_{k} = {\left\| {\L}^{\nicefrac{1}{2}}{u}_{k} \right\|}^{2}\,,
    \end{equation*}
    and
    \begin{equation*}
        {\tilde q}_{k} = {\alpha}_{k} + {\beta}_{k}{\gamma}_{k}\,.
    \end{equation*}
    Using these notations, the equality \eqref{eq:lemma1_inner_product_eqt} should be written in that way
    \begin{equation*}
        {\mu}_{k + 1} + \frac{1}{2}\left( {\alpha}_{k} + {\beta}_{k}{\gamma}_{k} \right){\gamma}_{k + 1} = {\mu}_{k} + \frac{1}{2}\left( {\alpha}_{k} + {\beta}_{k}{\gamma}_{k} \right){\gamma}_{k - 1} + \left( {f}_{k},\Delta{u}_{k} \right) + \left( {f}_{k},\Delta{u}_{k - 1} \right)\,.
    \end{equation*}
    By utilizing the Cauchy-Schwarz inequality on the right-hand side of the aforementioned equation, it can be deduced that
    \begin{equation*}
        {\mu}_{k + 1} + \frac{1}{2}\left( {\alpha}_{k} + {\beta}_{k}{\gamma}_{k} \right){\gamma}_{k + 1} \leq {\mu}_{k} + \frac{1}{2}\left( {\alpha}_{k} + {\beta}_{k}{\gamma}_{k} \right){\gamma}_{k - 1} + {\tau}\left( \sqrt{{\mu}_{k + 1}} + \sqrt{{\mu}_{k}} \right)\left\| {f}_{k} \right\|\,.
    \end{equation*}
    The following notations are introduced:
    \begin{align*}
        {\lambda}_{k} &= {\mu}_{k} + \frac{1}{2}\left( {\alpha}_{k - 1} + {\beta}_{k - 1}{\gamma}_{k - 1} \right){\gamma}_{k}\,,\\
        {\varepsilon}_{k} &= \left( {\xi}_{k} - {\xi}_{k + 1} \right) + {\eta}_{k} + {\tau}\left( \sqrt{{\mu}_{k}} + \sqrt{{\mu}_{k + 1}} \right)\left\| {f}_{k} \right\|\,,\\
        {\xi}_{k} &= \frac{1}{2}{\alpha}_{k}{\gamma}_{k - 1}\,,\\
        {\eta}_{k} &= \frac{1}{2}\left( {\alpha}_{k + 1} - {\alpha}_{k - 1} \right){\gamma}_{k} + \frac{1}{2}\left( {\beta}_{k} - {\beta}_{k - 1} \right){\gamma}_{k - 1}{\gamma}_{k}\,.
    \end{align*}
    We shall rewrite the inequality mentioned above by using the introduced notations, indeed we derive
    \begin{equation}\label{eq:lemma1_main_inequality}
        {\lambda}_{k + 1} \leq {\lambda}_{k} + {\varepsilon}_{k}\,.
    \end{equation}
    
    It is assumed that ${\alpha}\left( t \right)$ and ${\beta}\left( t \right)$ are continuous and continuously differentiable functions over the interval $t \in \left[ 0,T \right]$, with ${\alpha}\left( t \right) \geq {\c}_{0} > 0$ and ${\beta}\left( t \right) \geq {\c}_{1} > 0$. Based on these conditions, the following estimations hold:
    \begin{equation}\label{eq:lemma1_alpha_ineqt}
        \left| {\alpha}_{k + 1} - {\alpha}_{k - 1} \right| \leq \int\limits_{{t}_{k - 1}}^{{t}_{k + 1}}{\left| {\alpha}^{\prime}\left( t \right) \right|}{\d{t}} \leq 2{{\c}_{2}}{\tau}\,,\quad {\c}_{2} = \max\limits_{0 \leq t \leq T}{\left| {\alpha}^{\prime}\left( t \right) \right|}\,,
    \end{equation}
    analogously,
    \begin{equation}\label{eq:lemma1_beta_ineqt}
        \left| {\beta}_{k} - {\beta}_{k - 1} \right| \leq {{\c}_{3}}{\tau}\,,\quad {\c}_{3} = \max\limits_{0 \leq t \leq T}{\left| {\beta}^{\prime}\left( t \right) \right|}\,.
    \end{equation}
    By considering \eqref{eq:lemma1_alpha_ineqt} and \eqref{eq:lemma1_beta_ineqt}, it is possible to reach the following conclusion
    \begin{align}\label{eq:lemma1_eta_ineqt}
        \left| {\eta}_{k} \right| &\leq \frac{1}{2}\left| {\alpha}_{k + 1} - {\alpha}_{k - 1} \right|{\gamma}_{k} + \frac{1}{2}\left| {\beta}_{k} - {\beta}_{k - 1} \right|{\gamma}_{k - 1}{\gamma}_{k}\nonumber \\
        &\leq {\tau}\left( {\c}_{2} + \frac{1}{2}{\c}_{3}{\gamma}_{k - 1} \right){\gamma}_{k} \leq \max\left( {\c}_{2},\frac{1}{2}{\c}_{3} \right){\tau}\left( 1 + {\gamma}_{k - 1} \right){\gamma}_{k}\nonumber \\
        &= \max\left( {\c}_{2},\frac{1}{2}{\c}_{3} \right){\tau}\left( \frac{{\alpha}_{k - 1}}{{\alpha}_{k - 1}} + \frac{{\beta}_{k - 1}}{{\beta}_{k - 1}}{\gamma}_{k - 1} \right){\gamma}_{k}\nonumber \\
        &\leq \max\left( {\c}_{2},\frac{1}{2}{\c}_{3} \right){\tau}\left( \frac{{\alpha}_{k - 1}}{{\c}_{0}} + \frac{{\beta}_{k - 1}}{{\c}_{1}}{\gamma}_{k - 1} \right){\gamma}_{k}\nonumber \\
        &\leq {2}\max\left( {\c}_{2},\frac{1}{2}{\c}_{3} \right)\max\left( \frac{1}{{\c}_{0}},\frac{1}{{\c}_{1}} \right){\tau}\left[ \frac{1}{2}\left( {\alpha}_{k - 1} + {\beta}_{k - 1}{\gamma}_{k - 1} \right){\gamma}_{k} \right]\nonumber \\
        &= {\c}_{4}{\tau} \left[ \frac{1}{2}\left( {\alpha}_{k - 1} + {\beta}_{k - 1}{\gamma}_{k - 1} \right){\gamma}_{k} \right] \leq {\c}_{4}{\tau}\left[ {\mu}_{k} + \frac{1}{2}\left( {\alpha}_{k - 1} + {\beta}_{k - 1}{\gamma}_{k - 1} \right){\gamma}_{k} \right] = {\c}_{4}{\tau}{\lambda}_{k}\,,
    \end{align}
    where
    \begin{equation*}
        {\c}_{4} = {2}\max\left( {\c}_{2},\frac{1}{2}{\c}_{3} \right)\max\left( \frac{1}{{\c}_{0}},\frac{1}{{\c}_{1}} \right)\,.
    \end{equation*}
    In accordance with equation \eqref{eq:lemma1_eta_ineqt}, it can be deduced that ${\varepsilon}_{k} \leq {\tilde \varepsilon}_{k}$, where
    \begin{equation*}
        {\tilde \varepsilon}_{k} = \left( {\xi}_{k} - {\xi}_{k + 1} \right) + {\c}_{4}{\tau}{\lambda}_{k} + {\tau}\left( \sqrt{{\mu}_{k}} + \sqrt{{\mu}_{k + 1}} \right)\left\| {f}_{k} \right\|\,.
    \end{equation*}
    As a result, we arrive at the following inequality
    \begin{equation}\label{eq:lemma1_lambda_varepsilon_inqt}
        {\lambda}_{k + 1} \leq {\lambda}_{k} + {\varepsilon}_{k} \leq {\lambda}_{k} + {\tilde \varepsilon}_{k}\,.
    \end{equation}
    Through the application of a telescoping series cancellation technique on inequality \eqref{eq:lemma1_lambda_varepsilon_inqt}, it follows that
    \begin{align*}
        {\lambda}_{k + 1} \leq {\lambda}_{1} + \sum\limits_{i = 1}^{k}{\tilde \varepsilon}_{i} &= {\lambda}_{1} + \sum\limits_{i = 1}^{k}\left( {\xi}_{i} - {\xi}_{i + 1} \right) + {\c}_{4}{\tau}\sum\limits_{i = 1}^{k}{\lambda}_{i} + {\tau}\sum\limits_{i = 1}^{k}\left( \sqrt{{\mu}_{i}} + \sqrt{{\mu}_{i + 1}} \right)\left\| {f}_{i} \right\| \\
        &= {\lambda}_{1} + \left( {\xi}_{1} - {\xi}_{k + 1} \right) + {\c}_{4}{\tau}\sum\limits_{i = 1}^{k}{\lambda}_{i} + {\tau}\sum\limits_{i = 1}^{k}\left( \sqrt{{\mu}_{i}} + \sqrt{{\mu}_{i + 1}} \right)\left\| {f}_{i} \right\|\,.
    \end{align*}
    Subsequently, by rearranging the terms, it is obtained that
    \begin{equation*}
        {\lambda}_{k + 1} + {\xi}_{k + 1} \leq {\lambda}_{1} + {\xi}_{1} + {\c}_{4}{\tau}\sum\limits_{i = 1}^{k}{\lambda}_{i} + {\tau}\sum\limits_{i = 1}^{k}\left( \sqrt{{\mu}_{i}} + \sqrt{{\mu}_{i + 1}} \right)\left\| {f}_{i} \right\|\,.
    \end{equation*}
    Let us denote ${\delta}_{k} = \sqrt{{\lambda}_{k} + {\xi}_{k}}$, it is clear that this yields
    \begin{equation*}
        {\delta}_{k + 1}^{2} \leq {\delta}_{1}^{2} + {\c}_{4}{\tau}\sum\limits_{i = 1}^{k}{\delta}_{i}^{2} + {\tau}\sum\limits_{i = 1}^{k}\left( {\delta}_{i} + {\delta}_{i + 1} \right)\left\| {f}_{i} \right\|\,.
    \end{equation*}
    Assuming that ${\delta}_{j} = \max\limits_{1 \leq i \leq k + 1}{\delta}_{i}$. it follows (as demonstrated in \cite{RogavaTsiklauri_EvolEqt2014}) that
    \begin{align*}
        {\delta}_{j}^{2} &\leq {\delta}_{1}^{2} + {\c}_{4}{\tau}\sum\limits_{i = 1}^{j - 1}{\delta}_{i}^{2} + {\tau}\sum\limits_{i = 1}^{j - 1}\left( {\delta}_{i} + {\delta}_{i + 1} \right)\left\| {f}_{i} \right\|\,.
    \end{align*}
    Dividing both sides of the above inequality by ${\delta}_{j}$ gives us
    \begin{align*}
        {\delta}_{j} &\leq \frac{{\delta}_{1}}{{\delta}_{j}}{\delta}_{1} + {\c}_{4}{\tau}\sum\limits_{i = 1}^{j - 1}\frac{{\delta}_{i}}{{\delta}_{j}}{\delta}_{i} + {\tau}\sum\limits_{i = 1}^{j - 1}\left( \frac{{\delta}_{i}}{{\delta}_{j}} + \frac{{\delta}_{i + 1}}{{\delta}_{j}} \right)\left\| {f}_{i} \right\| \\
        &\leq {\delta}_{1} + {\c}_{4}{\tau}\sum\limits_{i = 1}^{j - 1}{\delta}_{i} + {2}{\tau}\sum\limits_{i = 1}^{j - 1}\left\| {f}_{i} \right\| \\
        &\leq {\delta}_{1} + {\c}_{4}{\tau}\sum\limits_{i = 1}^{k}{\delta}_{i} + {2}{\tau}\sum\limits_{i = 1}^{k}\left\| {f}_{i} \right\|\,.
    \end{align*}
    Due to the fact that ${\delta}_{k + 1} \leq {\delta}_{j}$, it can be concluded that
    \begin{equation*}
        {\delta}_{k + 1} \leq {\delta}_{1} + {\c}_{4}{\tau}\sum\limits_{i = 1}^{k}{\delta}_{i} + {2}{\tau}\sum\limits_{i = 1}^{k}\left\| {f}_{i} \right\|\,.
    \end{equation*}
    Thus, on account of the application of \hyperref[lemma:gronwall-inequality]{\bf Lemma \ref*{lemma:gronwall-inequality} (Discrete Gr\"{o}nwall-type inequality)} together with \hyperref[rmk:remark1]{\bf Remark \ref*{rmk:remark1}}, one can establish that
    \begin{equation}\label{eq:lemma1_gronwall_inequality}
        {\delta}_{k + 1} \leq {\e}^{{\c}_{4}{t}_{k}}\left( {\delta}_{1} + {2}{\tau}\sum\limits_{i = 1}^{k}{\left\| {f}_{i} \right\|} \right)\,.
    \end{equation}
    Considering the estimate
    \begin{equation*}
        \sum\limits_{i = 1}^{k}\left\| {f}_{i} \right\| \leq {k}\max\limits_{1 \leq i \leq k}{\left\| {f}_{i} \right\|}\,,
    \end{equation*}
    we observe that the inequality \eqref{eq:lemma1_gronwall_inequality} should be expressed as
    \begin{equation*}
        {\delta}_{k + 1} \leq {\e}^{{\c}_{4}{t}_{k}}\left( {\delta}_{1} + {2}{t}_{k}\max\limits_{1 \leq i \leq k}{\left\| {f}_{i} \right\|} \right)\,.
    \end{equation*}
    It then follows that ${\mu}_{k}$ and ${\gamma}_{k}$ are uniformly bounded.
\end{proof}
\begin{lemma}\label{lemma:lemma2}
    The sequences of functions ${\L}^{\nicefrac{1}{2}}\left( {u}_{k}\left( x \right) - {u}_{k - 1}\left( x \right) \right) / {\tau}$ and ${\L}{u}_{k}\left( x \right)$ are locally uniformly bounded with respect to the ${L}_{2}$-norm, {\ie}, there exists $\overline{T} > 0$ such that
    \begin{equation*}
        \left\| {\L}^{\nicefrac{1}{2}}\frac{{u}_{k} - {u}_{k - 1}}{\tau} \right\| \leq {\M}_{3}\,,\quad \left\| {\L}{u}_{k} \right\| \leq {\M}_{4}\,,\quad {k} = 1, 2, \ldots, \left[ \frac{\overline{T}}{\tau} \right]\,.
    \end{equation*}
    Here, ${\M}_{3}$ and ${\M}_{4}$ are positive constants that depend on the value of $\overline{T}$.
\end{lemma}
\begin{proof}
    Let us consider taking the inner product of both sides of the equation \eqref{eq:semidiscrete_scheme_extension_operator_form} with ${\L}\left(  {u}_{k + 1} - {u}_{k - 1} \right) = {\L}\left( \Delta{u}_{k} \right) + {\L}\left( \Delta{u}_{k - 1} \right)$. Upon integrating by parts, we obtain the following identity:
    \begin{align}\label{eq:lemma2_inner_product_eqt}
        {\left\| \frac{1}{\tau}{\L}^{\nicefrac{1}{2}}\left( \Delta{u}_{k} \right) \right\|}^{2} + \frac{1}{2}{{\tilde q}_{k}}{\left\| {\L}{u}_{k + 1} \right\|}^{2} &= {\left\| \frac{1}{\tau}{\L}^{\nicefrac{1}{2}}\left( \Delta{u}_{k - 1} \right) \right\|}^{2} + \frac{1}{2}{{\tilde q}_{k}}{\left\| {\L}{u}_{k - 1} \right\|}^{2}\nonumber \\
        &+ \left( {\L}^{\nicefrac{1}{2}}{f}_{k},{\L}^{\nicefrac{1}{2}}\left( \Delta{u}_{k} \right) \right) + \left( {\L}^{\nicefrac{1}{2}}{f}_{k},{\L}^{\nicefrac{1}{2}}\left( \Delta{u}_{k - 1} \right) \right)\,,
    \end{align}
    here, we assume that ${f}_{k} \in D\left( {\L}^{\nicefrac{1}{2}} \right)$.

    \noindent By applying the Cauchy-Schwarz inequality, we can derive
    \begin{align}\label{eq:lemma2_f_k_inequality}
        &\left| \left( {\L}^{\nicefrac{1}{2}}{f}_{k},{\L}^{\nicefrac{1}{2}}\left( \Delta{u}_{k} \right) \right) + \left( {\L}^{\nicefrac{1}{2}}{f}_{k},{\L}^{\nicefrac{1}{2}}\left( \Delta{u}_{k - 1} \right) \right) \right| \leq\nonumber \\
        &{\tau}\left\| {\L}^{\nicefrac{1}{2}}{f}_{k} \right\|\left( \left\| \frac{1}{\tau}{\L}^{\nicefrac{1}{2}}\left( \Delta{u}_{k} \right) \right\| + \left\| \frac{1}{\tau}{\L}^{\nicefrac{1}{2}}\left( \Delta{u}_{k - 1} \right) \right\| \right)\,.
    \end{align}
    Let us introduce the denotations
    \begin{equation*}
        {\tilde \mu}_{k} = {\left\| \frac{1}{\tau}{\L}^{\nicefrac{1}{2}}\left( \Delta{u}_{k - 1} \right) \right\|}^{2}\,,\quad {\nu}_{k} = {\left\| {\L}{u}_{k} \right\|}^{2}\,,\quad {\gamma}_{k} = {\left\| {\L}^{\nicefrac{1}{2}}{u}_{k} \right\|}^{2}\,,\quad {\sigma}_{k} = \left\| {\L}^{\nicefrac{1}{2}}{f}_{k} \right\|\,.
    \end{equation*}
    Substituting the notations introduced in the previous step and using the inequality \eqref{eq:lemma2_f_k_inequality}, we can rewrite the equality \eqref{eq:lemma2_inner_product_eqt} as follows
    \begin{equation*}
        {\tilde \mu}_{k + 1} + \frac{1}{2}\left( {\alpha}_{k} + {\beta}_{k}{\gamma}_{k} \right){\nu}_{k + 1} \leq {\tilde \mu}_{k} + \frac{1}{2}\left( {\alpha}_{k} + {\beta}_{k}{\gamma}_{k} \right){\nu}_{k - 1} + {\tau}{\sigma}_{k}\left( \sqrt{{\tilde \mu}_{k}} + \sqrt{{\tilde \mu}_{k + 1}} \right)\,.
    \end{equation*}
    Continuing from the previous step, using the notation $\displaystyle {\tilde \nu}_{k} = \frac{1}{2}\left( {\nu}_{k - 1} + {\nu}_{k} \right)$, we have
    \begin{equation*}
        {\tilde \mu}_{k + 1} + \left( {\alpha}_{k} + {\beta}_{k}{\gamma}_{k} \right){\tilde \nu}_{k + 1} \leq {\tilde \mu}_{k} + \left( {\alpha}_{k} + {\beta}_{k}{\gamma}_{k} \right){\tilde \nu}_{k} + {\tau}{\sigma}_{k}\left( \sqrt{{\tilde \mu}_{k}} + \sqrt{{\tilde \mu}_{k + 1}} \right)\,.
    \end{equation*}
    Expanding the inequality in the previous step, we obtain
    \begin{align}\label{eq:lemma2_lambda_inequality}
        {\tilde \mu}_{k + 1} + \left( {\alpha}_{k} + {\beta}_{k}{\gamma}_{k} \right){\tilde \nu}_{k + 1} &\leq {\tilde \mu}_{k} + \left( {\alpha}_{k - 1} + {\beta}_{k - 1}{\gamma}_{k - 1} \right){\tilde \nu}_{k}\nonumber\\
        &+ \left( {\alpha}_{k} - {\alpha}_{k - 1} \right){\tilde \nu}_{k} + \left( {\beta}_{k}{\gamma}_{k} - {\beta}_{k - 1}{\gamma}_{k - 1} \right){\tilde \nu}_{k} + {\tau}{\sigma}_{k}\left( \sqrt{{\tilde \mu}_{k}} + \sqrt{{\tilde \mu}_{k + 1}} \right)\,.
    \end{align}
    Evaluating the absolute value of the following difference ${\gamma}_{k} - {\gamma}_{k - 1}$ by virtue of the \hyperref[lemma:lemma1]{\bf Lemma \ref*{lemma:lemma1}}
    \begin{equation}\label{eq:lemma2_difference_of_gamma}
        \left| {\gamma}_{k} - {\gamma}_{k - 1} \right| \leq \left( \sqrt{{\gamma}_{k - 1}} + \sqrt{{\gamma}_{k}} \right){\tau}\sqrt{{\tilde \mu}_{k}} \leq {2}{\M}_{2}{\tau}\sqrt{{\tilde \mu}_{k}}\,.
    \end{equation}
    Similar to the estimation presented in the inequality \eqref{eq:lemma1_alpha_ineqt}, we can derive the following
    \begin{equation}\label{eq:lemma2_alpha_ineqt}
        \left| {\alpha}_{k} - {\alpha}_{k - 1} \right| \leq {\c}_{2}{\tau}\,.
    \end{equation}
    By applying the inequalities \eqref{eq:lemma1_beta_ineqt} and \eqref{eq:lemma2_difference_of_gamma} along with \hyperref[lemma:lemma1]{\bf Lemma \ref*{lemma:lemma1}}, one can obtain the following estimates:
    \begin{align}\label{eq:lemma2_beta_gamma}
        \left| {\beta}_{k}{\gamma}_{k} - {\beta}_{k - 1}{\gamma}_{k - 1} \right| &\leq \left| {\beta}_{k} - {\beta}_{k - 1} \right|{\gamma}_{k} + {\beta}_{k - 1}\left| {\gamma}_{k} - {\gamma}_{k - 1} \right|\nonumber \\
        &\leq {\c}_{3}{\tau}{\gamma}_{k} + {2}{\beta}_{k - 1}{\M}_{2}{\tau}\sqrt{{\tilde \mu}_{k}}\nonumber \\
        &\leq {\c}_{3}{\M}_{2}^{2}{\tau} + {2}{\beta}_{k - 1}{\M}_{2}{\tau}\sqrt{{\tilde \mu}_{k}}\nonumber \\
        &\leq \max\left( {\c}_{3}{\M}_{2}^{2},{2}{\beta}_{k - 1}{\M}_{2} \right){\tau}\left( 1 + \sqrt{{\tilde \mu}_{k}} \right)\nonumber \\
        &={\c}_{5}{\tau}\left( 1 + \sqrt{{\tilde \mu}_{k}} \right)\,.
    \end{align}
    Through the utilization of inequalities \eqref{eq:lemma2_alpha_ineqt} and \eqref{eq:lemma2_beta_gamma}, we are able to reformulate inequality \eqref{eq:lemma2_lambda_inequality} in the following manner
    \begin{align*}
        {\tilde \mu}_{k + 1} + \left( {\alpha}_{k} + {\beta}_{k}{\gamma}_{k} \right){\tilde \nu}_{k + 1} &\leq {\tilde \mu}_{k} + \left( {\alpha}_{k - 1} + {\beta}_{k - 1}{\gamma}_{k - 1} \right){\tilde \nu}_{k} \\
        &+ {\c}_{2}{\tau}{\tilde \nu}_{k} + {\c}_{5}{\tau}\left( 1 + \sqrt{{\tilde \mu}_{k}} \right){\tilde \nu}_{k} + {\tau}{\sigma}_{k}\left( \sqrt{{\tilde \mu}_{k}} + \sqrt{{\tilde \mu}_{k + 1}} \right) \\
        &= {\tilde \mu}_{k} + \left( {\alpha}_{k - 1} + {\beta}_{k - 1}{\gamma}_{k - 1} \right){\tilde \nu}_{k} + {\c}_{6}{\tau}{\tilde \nu}_{k} \\
        &+ {\c}_{5}{\tau}\sqrt{{\tilde \mu}_{k}}{\tilde \nu}_{k} + {\tau}{\sigma}_{k}\left( \sqrt{{\tilde \mu}_{k}} + \sqrt{{\tilde \mu}_{k + 1}} \right)\,,\quad {\c}_{6} = {\c}_{2} + {\c}_{5}\,.
    \end{align*}
    If we introduce the notation given by $\displaystyle {\tilde \lambda}_{k} = {\tilde \mu}_{k} + \left( {\alpha}_{k - 1} + {\beta}_{k - 1}{\gamma}_{k - 1} \right){\tilde \nu}_{k}$, we shall obtain
    \begin{equation}\label{eq:lemma2_tilde_lambda_eqt}
        {\tilde \lambda}_{k + 1} \leq {\tilde \lambda}_{k} + {\c}_{6}{\tau}{\tilde \nu}_{k} + {\c}_{5}{\tau}\sqrt{{\tilde \mu}_{k}}{\tilde \nu}_{k} + {\tau}{\sigma}_{k}\left( \sqrt{{\tilde \mu}_{k}} + \sqrt{{\tilde \mu}_{k + 1}} \right)\,.
    \end{equation}
    Let us now consider the following set of inequalities:
    \begin{gather*}
        {\tilde \mu}_{k} \leq {\tilde \lambda}_{k}\,,\quad {\tilde \nu}_{k} \leq \frac{1}{{\c}_{0}}{\tilde \lambda}_{k} = {\c}_{7}{\tilde \lambda}_{k}\,,\quad {\sigma}_{k} \leq {\c}_{8} = \max\limits_{0 \leq t \leq {T}}\left\| {\L}^{\nicefrac{1}{2}} f\left( \cdot,{t} \right) \right\|\,, \\
        \sqrt{{\tilde \mu}_{k}} \leq \frac{1}{2}\left( 1 + {\tilde \mu}_{k} \right) \leq \frac{1}{2}\left( 1 + {\tilde \lambda}_{k} \right)\,,
    \end{gather*}
    By utilizing equation \eqref{eq:lemma2_tilde_lambda_eqt}, we are able to draw the conclusion that
    \begin{align*}
        {\tilde \lambda}_{k + 1} &\leq {\tilde \lambda}_{k} + {\c}_{9}{\tau}{\tilde \lambda}_{k} + {\c}_{10}{\tau}\sqrt{{\tilde \lambda}_{k}}{\tilde \lambda}_{k} + {\c}_{8}{\tau}\sqrt{{\tilde \lambda}_{k + 1}} + \frac{1}{2}{\c}_{8}{\tau}\left( 1 + {\tilde \lambda}_{k} \right) \\
        &=\left( 1 + {\c}_{11}{\tau} + {\c}_{10}{\tau}\sqrt{{\tilde \lambda}_{k}} \right){\tilde \lambda}_{k} + {\c}_{8}{\tau}\sqrt{{\tilde \lambda}_{k + 1}} + {\c}_{12}{\tau}\,, \\
        & {\c}_{9} = {\c}_{6}{\c}_{7}\,,\quad {\c}_{10} = {\c}_{5}{\c}_{7}\,,\quad {\c}_{11} = {\c}_{9} + {\c}_{12}\,,\quad {\c}_{12} = \frac{1}{2}{\c}_{8}\,.
    \end{align*}
    Hence, we are able to conclude that
    \begin{equation*}
        {\tilde \lambda}_{k + 1} \leq {\tilde \lambda}_{k}\left( 1 + {\c}_{11}{\tau} \right)\left( 1 + \frac{{\c}_{10}}{1 + {\c}_{11}{\tau}}{\tau}\sqrt{{\tilde \lambda}_{k}} \right) + {\c}_{8}{\tau}\sqrt{{\tilde \lambda}_{k + 1}} + {\c}_{12}{\tau}
    \end{equation*}
    Let us divide both sides of the aforementioned inequality by ${\left( 1 + {\c}_{11}{\tau} \right)}^{k + 1}$; as a result, we obtain
    \begin{align*}
        \frac{{\tilde \lambda}_{k + 1}}{{\left( 1 + {\c}_{11}{\tau} \right)}^{k + 1}} &\leq \frac{{\tilde \lambda}_{k}}{{\left( 1 + {\c}_{11}{\tau} \right)}^{k}}\left( 1 + \frac{{\c}_{10}}{1 + {\c}_{11}{\tau}}{\tau}\sqrt{{\tilde \lambda}_{k}} \right) + \frac{{\c}_{8}\sqrt{{\tilde \lambda}_{k + 1}}}{{\left( 1 + {\c}_{11}{\tau} \right)}^{k + 1}}{\tau} + \frac{{\c}_{12}}{{\left( 1 + {\c}_{11}{\tau} \right)}^{k + 1}}{\tau} \\
        &= \frac{{\tilde \lambda}_{k}}{{\left( 1 + {\c}_{11}{\tau} \right)}^{k}}\left( 1 + \frac{{\c}_{10}\sqrt{{\left( 1 + {\c}_{11}{\tau} \right)}^{k}}}{1 + {\c}_{11}{\tau}}{\tau}\sqrt{\frac{{\tilde \lambda}_{k}}{{\left( 1 + {\c}_{11}{\tau} \right)}^{k}}} \right) \\
        &+\frac{{\c}_{8}}{\sqrt{{\left( 1 + {\c}_{11}{\tau} \right)}^{k + 1}}}{\tau}\sqrt{\frac{{\tilde \lambda}_{k + 1}}{{\left( 1 + {\c}_{11}{\tau} \right)}^{k + 1}}} + \frac{{\c}_{12}}{{\left( 1 + {\c}_{11}{\tau} \right)}^{k + 1}}{\tau}\,.
    \end{align*}
    By utilizing the notation $\displaystyle {\tilde \xi}_{k} = \frac{{\tilde \lambda}_{k}}{{\left( 1 + {\c}_{11}{\tau} \right)}^{k}}$, and considering the following set of plain inequalities
    \begin{gather*}
        \frac{{\c}_{10}\sqrt{{\left( 1 + {\c}_{11}{\tau} \right)}^{k}}}{1 + {\c}_{11}{\tau}} \leq {\c}_{10}\sqrt{{\e}^{{\c}_{11}{T}}} = {\c}_{13}\,, \\
        \frac{{\c}_{8}}{\sqrt{{\left( 1 + {\c}_{11}{\tau} \right)}^{k + 1}}} \leq {\c}_{8}\,,\quad\frac{{\c}_{12}}{{\left( 1 + {\c}_{11}{\tau} \right)}^{k + 1}} \leq {\c}_{12}\,,
    \end{gather*}
    As a result of the aforementioned inequalities, we are able to determine that
    \begin{equation}\label{eq:lemma2_tilde_xi_inqt}
        {\tilde \xi}_{k + 1} \leq {\tilde \xi}_{k}\left( 1 + {\c}_{13}{\tau}\sqrt{{\tilde \xi}_{k}} \right) + {\c}_{8}{\tau}\sqrt{{\tilde \xi}_{k + 1}} + {\c}_{12}{\tau}\,.
    \end{equation}
    In order to simplify inequality \eqref{eq:lemma2_tilde_xi_inqt} and express it in a more straightforward manner, let us introduce new denotations
    \begin{equation*}
        {y}_{k + 1} = \sqrt{{\tilde \xi}_{k + 1}},\quad {w}_{k} = {\tilde \xi}_{k}\left( 1 + {\c}_{13}{\tau}\sqrt{{\tilde \xi}_{k}} \right) + {\c}_{12}{\tau}\,,
    \end{equation*}
    correspondingly, we derive the following quadratic inequality for ${y}_{k + 1}$,
    \begin{equation*}
        {y}_{k + 1}^{2} - {\c}_{8}{\tau}{y}_{k + 1} - {w}_{k} \leq 0\,.
    \end{equation*}
    It should be noted that the discriminant of this quadratic polynomial is given by $\displaystyle {\c}_{8}^{2}{\tau}^{2} + {4}{w}_{k}$, which is non-negative. Therefore, we can write the following inequality
    \begin{equation*}
        {y}_{k + 1} \leq \frac{{\c}_{8}}{2}{\tau} + \sqrt{\frac{{\c}_{8}^{2}}{4}{\tau}^{2} + {w}_{k}}\,.
    \end{equation*}
    By squaring both sides of the aforementioned inequality, we can readily draw a conclusion
    \begin{equation}\label{eq:lemma2_quad_poly_y_k+1}
        {y}_{k + 1}^{2} \leq \frac{{\c}_{8}^{2}}{2}{\tau}^{2} + {\c}_{8}{\tau}\sqrt{\frac{{\c}_{8}^{2}}{4}{\tau}^{2} + {w}_{k}} + {w}_{k}\,.
    \end{equation}
    Let us evaluate the second term in the right-hand side of inequality \eqref{eq:lemma2_quad_poly_y_k+1}. We have
    \begin{align*}
        {\c}_{8}{\tau}\sqrt{\frac{{\c}_{8}^{2}}{4}{\tau}^{2} + {w}_{k}} &= \sqrt{{\c}_{8}{\tau}}\sqrt{\frac{{\c}_{8}^{3}}{4}{\tau}^{3} + {\c}_{8}{w}_{k}{\tau}} \\
        &\leq \frac{1}{2}\left( {\c}_{8}{\tau} + \frac{{\c}_{8}^{3}}{4}{\tau}^{3} + {\c}_{8}{w}_{k}{\tau} \right)\,.
    \end{align*}
    Taking the last inequality into account in \eqref{eq:lemma2_quad_poly_y_k+1}, we eventually conclude that
    \begin{align}\label{eq:lemma2_y_k+1_ineqt}
        {y}_{k + 1}^{2} &\leq \frac{{\c}_{8}}{2}{\tau} + \frac{{\c}_{8}^{2}}{2}{\tau}^{2} + \frac{{\c}_{8}^{3}}{8}{\tau}^{3} + {w}_{k} + \frac{{\c}_{8}}{2}{w}_{k}{\tau}\nonumber \\
        &= \frac{{\c}_{8}}{2}{\tau}\left( 1 + {\c}_{8}{\tau} + \frac{{\c}_{8}^{2}}{4}{\tau}^{2} \right) + \left( 1 + \frac{{\c}_{8}}{2}{\tau} \right){w}_{k}\,.
    \end{align}
    Upon reversing the substitution of $\displaystyle {y}_{k + 1} = \sqrt{{\tilde \xi}_{k + 1}}$ and $\displaystyle {w}_{k} = {\tilde \xi}_{k}\left( 1 + {\c}_{13}{\tau}\sqrt{{\tilde \xi}_{k}} \right) + {\c}_{12}{\tau}$ in inequality \eqref{eq:lemma2_y_k+1_ineqt}, we obtain the following expression
    \begin{align*}
        {\tilde \xi}_{k + 1} \leq \left( 1 + \frac{{\c}_{8}}{2}{\tau} \right){\tilde \xi}_{k}\left( 1 + {\c}_{13}{\tau}\sqrt{{\tilde \xi}_{k}} \right) &+ \left( 1 + \frac{{\c}_{8}}{2}{\tau} \right){\c}_{12}{\tau} \\
        &+ \frac{{\c}_{8}}{2}{\tau}\left( 1 + {\c}_{8}{\tau} + \frac{{\c}_{8}^{2}}{4}{\tau}^{2} \right)\,.
    \end{align*}
    By reintroducing the notations $\displaystyle {\c}_{14} = \frac{{\c}_{8}}{2}$ and $\displaystyle {\c}_{15} = \left( 1 + {\c}_{14}{\tau} \right){\c}_{12} + {\c}_{14}\left( 1 + {2}{\c}_{14}{\tau} + {\c}_{14}^{2}{\tau}^{2} \right)$, we can conclude that
    \begin{equation*}
        {\tilde \xi}_{k + 1} \leq \left( 1 + {\c}_{14}{\tau} \right){\tilde \xi}_{k}\left( 1 + {\c}_{13}{\tau}\sqrt{{\tilde \xi}_{k}} \right) + {\c}_{15}{\tau}\,.
    \end{equation*}
    Upon dividing both sides of the aforementioned inequality by ${\left( 1 + {\c}_{14}{\tau} \right)}^{k + 1}$, we obtain the following expression
    \begin{align}\label{eq:lemma2_frac_xi_k_inequality}
        \frac{{\tilde \xi}_{k + 1}}{{\left( 1 + {\c}_{14}{\tau} \right)}^{k + 1}} &\leq \frac{{\tilde \xi}_{k}}{{\left( 1 + {\c}_{14}{\tau} \right)}^{k}}\left( 1 + {\c}_{13}{\tau}\sqrt{{\tilde \xi}_{k}} \right) + \frac{{\c}_{15}}{{\left( 1 + {\c}_{14}{\tau} \right)}^{k + 1}}{\tau}\nonumber \\
        &= \frac{{\tilde \xi}_{k}}{{\left( 1 + {\c}_{14}{\tau} \right)}^{k}}\left( 1 + {\c}_{13}\sqrt{{\left( 1 + {\c}_{14}{\tau} \right)}^{k}}{\tau}\sqrt{\frac{{\tilde \xi}_{k}}{{\left( 1 + {\c}_{14}{\tau} \right)}^{k}}}\,\, \right) + \frac{{\c}_{15}}{{\left( 1 + {\c}_{14}{\tau} \right)}^{k + 1}}{\tau}\,.
    \end{align}
    The following estimates are applicable
    \begin{gather*}
        {\c}_{13}\sqrt{{\left( 1 + {\c}_{14}{\tau} \right)}^{k}} \leq {\c}_{13}\sqrt{{\e}^{{\c}_{14}{T}}} = {\c}_{16}\,,\quad \frac{{\c}_{15}}{{\left( 1 + {\c}_{14}{\tau} \right)}^{k + 1}} \leq {\c}_{15}\,.
    \end{gather*}
    By introducing the notation $\displaystyle {\zeta}_{k} = \frac{{\tilde \xi}_{k}}{{\left( 1 + {\c}_{14}{\tau} \right)}^{k}}$ and applying it along with the aforementioned inequalities in \eqref{eq:lemma2_frac_xi_k_inequality}, we can deduce the following conclusion
    \begin{equation*}
        {\zeta}_{k + 1} \leq {\zeta}_{k}\left( 1 + {\c}_{16}{\tau}\sqrt{{\zeta}_{k}} \right) + {\c}_{15}{\tau}\,.
    \end{equation*}
    Let us make a certain transformation, specifically by introducing $\displaystyle {\overline{\tau}} = {\c}_{16}{\tau}$, we have
    \begin{equation*}
        {\zeta}_{k + 1} \leq {\zeta}_{k}\left( 1 + {\overline{\tau}}\sqrt{{\zeta}_{k}} \right) + {\c}_{17}{\overline{\tau}}\,,\quad {\c}_{17} = \frac{{\c}_{15}}{{\c}_{16}}\,.
    \end{equation*}
    As a result of \hyperref[lemma:rogava-tsiklauri]{\bf Lemma \ref*{lemma:rogava-tsiklauri}}, it can be deduced that
    \begin{equation}\label{eq:lemma2_result_of_lemma_rogavatsiklauri}
        {\zeta}_{k} \leq \frac{\zeta}{{\left( 1 - \dfrac{\left( 1 + {\c}_{17} \right)\sqrt{\zeta}}{2}{\overline{t}}_{k} \right)}^{2}} \leq \frac{{\tilde \lambda}}{{\left( 1 - \dfrac{{\c}_{16}\left( 1 + {\c}_{17} \right)\sqrt{{\tilde \lambda}}}{2}{{t}_{k}} \right)}^{2}}\,,\quad {k} = 1, 2, \ldots, {m}\,,
    \end{equation}
    where
    \begin{equation*}
        {\zeta} = \max\left( 1,{\zeta}_{1} \right) \leq \max\left( 1,{\tilde \lambda}_{1} \right) = {\tilde \lambda}\,,\quad {\overline{t}}_{k} = {k}{\overline{\tau}} = {\c}_{16}{{t}_{k}} < \frac{2}{\left( 1 + {\c}_{17} \right)\sqrt{{\tilde \lambda}}} \leq \frac{2}{\left( 1 + {\c}_{17} \right)\sqrt{{\zeta}}}\,.
    \end{equation*}
    According to the introduced denotation, we have
    \begin{equation}\label{eq:lemma2_conn_btwn_zeta_k_lambda_k_inqt}
        {\zeta}_{k} = \dfrac{{\tilde \lambda}_{k}}{{\left( 1 + {\c}_{11}{\tau} \right)}^{k}{\left( 1 + {\c}_{14}{\tau} \right)}^{k}} \geq \dfrac{{\tilde \lambda}_{k}}{{\e}^{\left( {\c}_{11} + {\c}_{14} \right){t}_{k}}}\,.
    \end{equation}
    The inequality \eqref{eq:lemma2_result_of_lemma_rogavatsiklauri} with the help of \eqref{eq:lemma2_conn_btwn_zeta_k_lambda_k_inqt} gives us the following estimation
    \begin{equation}\label{eq:lemma2_result_of_lemma_rogavatsiklauri_tilde_lambda}
        {\tilde \lambda}_{k} \leq \frac{{\tilde \lambda}}{{\left( 1 - {\c}_{18}\sqrt{{\tilde \lambda}}{{t}_{k}} \right)}^{2}}{\e}^{{\c}_{19}{t}_{k}}\,,\quad {k} = 1, 2, \ldots, {m}\,,
    \end{equation}
    
    It is worth mentioning that the value of ${m}$ in inequality \eqref{eq:lemma2_result_of_lemma_rogavatsiklauri_tilde_lambda} depends on both the coefficient of ${t}_{k}$ (which appears in the denominator of the fraction) and the number of time interval divisions, ${n}$. The coefficient of ${t}_{k}$ can be explicitly estimated using data from the problem \eqref{eq:main_eqt}-\eqref{eq:boundary_conds}, taking into account the value of $T$ as well. Furthermore, the inequality ${\tilde \lambda} \leq {\widetilde{\M}}$ holds, where ${\widetilde{\M}}$ is a positive constant that depends on $\left\| {\L}^{2}{\psi}_{0} \right\|$, $\left\| {\L}{\psi}_{1} \right\|$, $\left\| {\L}{f}_{0} \right\|$, and the value of $T$.
    % \begin{align*}
    %     {\tilde \lambda}_{1} &= {\tilde \mu}_{1} + \left( {\alpha}_{0} + {\beta}_{0}{\gamma}_{0} \right){\tilde \nu}_{0} \\
    %     &= {\left\| \frac{1}{\tau}{\L}^{\nicefrac{1}{2}}\left( {u}_{1} - {u}_{0} \right) \right\|}^{2} + \frac{1}{2}\left( {\alpha}_{0} + {\beta}_{0}{\left\| {\L}^{\nicefrac{1}{2}}{u}_{0} \right\|}^{2} \right)\left( {\left\| {\L}{u}_{0} \right\|}^{2} + {\left\| {\L}{u}_{1} \right\|}^{2} \right)\,.
    % \end{align*}
    % \begin{equation*}
    %     {\tilde \mu}_{k} = {\left\| \frac{1}{\tau}{\L}^{\nicefrac{1}{2}}\left( \Delta{u}_{k - 1} \right) \right\|}^{2}\,,\quad {\nu}_{k} = {\left\| {\L}{u}_{k} \right\|}^{2}\,,\quad {\gamma}_{k} = {\left\| {\L}^{\nicefrac{1}{2}}{u}_{k} \right\|}^{2}\,,\quad {\tilde \nu}_{k} = \frac{1}{2}\left( {\nu}_{k - 1} + {\nu}_{k} \right)\,.
    % \end{equation*}
    
    From \eqref{eq:lemma2_result_of_lemma_rogavatsiklauri_tilde_lambda}, the following inequality follows:
    \begin{equation}\label{eq:lemma2_final_result_loc_bound}
        {\tilde \lambda}_{k} \leq \dfrac{\widetilde{\M}}{{\left( 1 - \overline{\M}{\,}\overline{T} \right)}^{2}}{\e}^{{\c}_{19}\overline{T}}\,,\quad {k} = 1, 2, \ldots, \left[ \frac{\overline{T}}{{\tau}} \right]\,,
    \end{equation}
    where $\overline{\M} = {\c}_{18}\sqrt{\widetilde{\M}}$, $\overline{T} = \dfrac{q}{\overline{\M}}$, ${0} < {q} < {1}$.

    The inequality given in \eqref{eq:lemma2_final_result_loc_bound} implies that the sequences of functions ${\L}^{\nicefrac{1}{2}}\left( {u}_{k}\left( x \right) - {u}_{k - 1}\left( x \right) \right) / {\tau}$ and ${\L}{u}_{k}\left( x \right)$ are uniformly bounded over the local interval $\left[ {0}, \overline{T} \right]$.
\end{proof}
\begin{remark}\label{rmk:remark-lemma1}
    It is evident that when the sequences of functions ${u}_{k}\left( x \right) \in D\left( {\L}_{0} \right)$, from the second estimate of \hyperref[lemma:lemma1]{\bf Lemma \ref*{lemma:lemma1}} follows that the norm of the derivative of ${u}_{k}$ with respect to $x$ is bounded above by a constant ${\M}_{2}$, that is
    \begin{equation*}
        \left\| \frac{\d {u}_{k}}{\d x} \right\| \leq {\M}_{2}\,.
    \end{equation*}
\end{remark}
\subsection{Error Estimate of the Approximate Solution}\label{subsec:error_estimate}
It should be noted that throughout the text, the index $t$ is occasionally omitted in expressing the derivative of the function ${u}\left( x,t \right)$ with respect to the temporal variable $t$. Specifically, the notation ${u}^{\prime}\left( x,t \right)$ is used to represent the first-order derivative of the function ${u}\left( x,t \right)$ with respect to its second argument, rather than ${u}_{{t}}^{\prime}\left( x,t \right)$ or ${u}_{{t}}\left( x,t \right)$. When mixed partial derivatives of the function ${u}\left( x,t \right)$ with respect to the spatial and the temporal variables are encountered, lower indices ${x}$ and ${t}$ are used to denote the appropriate partial derivatives. For instance, ${u}_{{x}{x}{t}}\left( x,t \right)$ represents the second-order and the first-order partial derivatives of the function ${u}\left( x,t \right)$ with respect to the spatial and the temporal variables, respectively.

Before presenting the theorem on the convergence of the scheme \eqref{eq:semidiscrete_scheme}, we first provide a remark on the smoothness of the solution to problem \eqref{eq:main_eqt}-\eqref{eq:boundary_conds} to ascertain the order of convergence of the proposed symmetric three-layer semi-discrete scheme \eqref{eq:semidiscrete_scheme}. To ensure the well-posedness of the problem, a minimum degree of smoothness of the solution is required, which guarantees convergence but is insufficient to determine the order of convergence. By raising the smoothness of the solution by one degree, the order of convergence becomes equal to one (specifically, in both the previous and current scenarios, where ${u}_{1}\left( x \right) = {\psi}_{0}\left( x \right) + {\tau}{\psi}_{1}\left( x \right)$ is sufficient). Furthermore, if the smoothness degree is risen by two and the initial function is specified using formula \eqref{eq:semidscrete_scheme_first_layer}, the order of convergence is improved by an additional degree, resulting in a total order of two. However, further increasing the smoothness degree would be superfluous, as the approximation order of scheme \eqref{eq:semidiscrete_scheme} does not exceed two.

The subsequent theorem is established regarding the convergence of the scheme \eqref{eq:semidiscrete_scheme}.
\begin{theorem}\label{theorem:theorem1}
    Let the problem \eqref{eq:main_eqt}-\eqref{eq:boundary_conds} be well-posed. Besides, the following conditions are fulfilled:
    \begin{enumerate}[label=(\alph*)]
        \item\label{itm_theorem_a} ${\psi}_{0}\left( x \right) \in D\left( {\L}_{0} \right)$, ${\psi}_{1}\left( x \right) \in {C}^{1}\left( \left[ {0},{\ell} \right] \right)$ and the function ${u}_{{x}}\left( x,t \right)$ has a second-order continuous derivative with respect to the temporal variable.
        \item\label{itm_theorem_b} The solution ${u}\left( x,t \right)$ of the problem \eqref{eq:main_eqt}-\eqref{eq:boundary_conds} is a continuously differentiable function up to and including the third order with respect to the temporal variable, moreover, ${u}^{{\prime}{\prime}{\prime}}\left( x,t \right)$ is a Lipschitz continuous function with respect to the temporal variable.
        \item\label{itm_theorem_c} The function ${u}_{{x}{x}}\left( x,t \right)$ is continuously differentiable with respect to the temporal variable, as well ${u}_{{x}{x}{t}}\left( x,t \right)$ satisfies the Lipschitz condition with respect to the temporal variable.
    \end{enumerate}
    Then there exists $\overline{T}$ $\left( 0 < \overline{T} \leq {T} \right)$ such that for the error of an approximate solution denoted by ${z}_{k}\left( x \right)$ and defined as follows ${z}_{k}\left( x \right) = {u}\left( x,{t}_{k} \right) - {u}_{k}\left( x \right)$ the following estimates hold:
    \begin{equation*}
        \max\limits_{1 \leq k \leq m}\left\| \frac{{\d}{z}_{k}}{{\d}{x}} \right\| \leq {\M}_{5}{\tau}^{2}\,,\quad \max\limits_{0 \leq k \leq m - 1}\left\| \frac{{z}_{k + 1} - {z}_{k}}{\tau} \right\| \leq {\M}_{6}{\tau}^{2}\,,
    \end{equation*}
    where $\displaystyle {m} = \left[ \frac{\overline{T}}{{\tau}} \right]$.
\end{theorem}
\begin{proof}
    If we evaluate the exact representation of equation \eqref{eq:main_eqt} at ${t} = {t}_{k}$, which is expressed as equality \eqref{eq:discrete_main_eqt}, and subtract it from equation \eqref{eq:semidiscrete_scheme}, evidently for ${z}_{k}\left( x \right)$ we obtain
    \begin{equation}\label{eq:theorem_semidiscrete_scheme_for_z_k}
        \frac{{\Delta}^{2}{z}_{k - 1}\left( x \right)}{{\tau}^{2}} - \frac{1}{2}{q}_{k}\left( \frac{{\d}^{2}{z}_{k + 1}\left( x \right)}{{\d}{x}^{2}} + \frac{{\d}^{2}{z}_{k - 1}\left( x \right)}{{\d}{x}^{2}} \right) = {g}_{k}\left( x \right)\,,\quad {k} = 1, 2, \ldots, {n - 1}\,,
    \end{equation}
    here
    \begin{align*}
        {\Delta}{z}_{k}\left( x \right) &= {z}_{k + 1}\left( x \right) - {z}_{k}\left( x \right)\,, \\
        {g}_{k}\left( x \right) &= \frac{1}{2}\left( {q}\left( {t}_{k} \right) - {q}_{k} \right)\left( \frac{{\d}^{2}{u}\left( x,{t}_{k + 1} \right)}{{\d}{x}^{2}} + \frac{{\d}^{2}{u}\left( x,{t}_{k - 1} \right)}{{\d}{x}^{2}} \right) + {R}_{k}\left( x,{\tau} \right)\,, \\
        {R}_{k}\left( x,{\tau} \right) &= {R}_{1,k}\left( x,\tau \right) + {R}_{2,k}\left( x,\tau \right) = \underbrace{\left( \frac{{\Delta}^{2}u\left( x,{t}_{k - 1} \right)}{{\tau}^{2}} - \frac{{\partial}^{2}u\left( x,{t}_{k} \right)}{\partial{t}^{2}} \right)}_{{R}_{1,k}\left( x,\tau \right)} + \underbrace{\left( - \frac{1}{2}{q}\left( {t}_{k} \right)\frac{\d ^{2}}{\d {x}^{2}}{\Delta}^{2}u\left( x,{t}_{k - 1} \right) \right)}_{{R}_{2,k}\left( x,\tau \right)}\,.
    \end{align*}
    We shall proceed to estimate the remainder term ${R}_{k}\left( x,{\tau} \right)$ in equation \eqref{eq:theorem_semidiscrete_scheme_for_z_k}. To do so, we consider the Taylor expansion of the function ${u}\left( x,t \right)$ around the point $t = {t}_{k}$. Specifically, we have
    \begin{equation}\label{eq:theorem_taylor_series_four_terms}
        {u}\left( x,t \right) = {u}\left( x,{t}_{k} \right) + {\left( t - {t}_{k} \right)}{u}^{{\prime}}\left( x,{t}_{k} \right) + \frac{{\left( t - {t}_{k} \right)}^{2}}{2}{u}^{{\prime}{\prime}}\left( x,{t}_{k} \right) + \frac{{\left( t - {t}_{k} \right)}^{3}}{6}{u}^{{\prime}{\prime}{\prime}}\left( x,{t}_{k} \right) + {\widetilde R}_{3}\left( x,t \right)\,.
    \end{equation}
    In this context, the remainder term ${\widetilde R}_{3}\left( x,t \right)$, also known as the Lagrange remainder, is given by
    \begin{equation*}
        {\widetilde R}_{3}\left( x,t \right) = \int\limits_{{t}_{k}}^{t}\int\limits_{{t}_{k}}^{{s}_{1}}\int\limits_{{t}_{k}}^{{s}_{2}}{\left( {u}^{{\prime}{\prime}{\prime}}\left( x,{s}_{3} \right) - {u}^{{\prime}{\prime}{\prime}}\left( x,{t}_{k} \right) \right)}{\d {s}_{3}}{\d {s}_{2}}{\d {s}_{1}}\,.
    \end{equation*}
    By virtue of condition \ref{itm_theorem_b} in \hyperref[theorem:theorem1]{\bf Theorem \ref*{theorem:theorem1}}, it can be simply conclude that
    \begin{equation}\label{eq:theorem_est_lagrange_rem_triple_integ}
        \max\limits_{0 \leq x \leq \ell}\left| {\widetilde R}_{3}\left( x,t \right) \right| \leq \frac{{\c}_{20}}{24}{\left( {t} - {t}_{k} \right)}^{4}\,.
    \end{equation}
    Upon substituting $t = {t}_{k - 1}$ and $t = {t}_{k + 1}$ in equality \eqref{eq:theorem_taylor_series_four_terms}, we obtain the following result
    \begin{equation*}
        \frac{{\Delta}^{2}u\left( x,{t}_{k - 1} \right)}{{\tau}^{2}} - \frac{{\partial}^{2}u\left( x,{t}_{k} \right)}{\partial{t}^{2}} = \frac{1}{{\tau}^{2}}\left( {\widetilde R}_{3}\left( x,{t}_{k - 1} \right) + {\widetilde R}_{3}\left( x,{t}_{k + 1} \right) \right)\,,
    \end{equation*}
    and from here by taking into account inequality \eqref{eq:theorem_est_lagrange_rem_triple_integ}, we obtain the following estimate
    \begin{align}\label{eq:theorem_abs_r1k}
        \max\limits_{0 \leq x \leq \ell}\left| {R}_{1,k}\left( x,\tau \right) \right| &= \max\limits_{0 \leq x \leq \ell}\left| \frac{{\Delta}^{2}u\left( x,{t}_{k - 1} \right)}{{\tau}^{2}} - \frac{{\partial}^{2}u\left( x,{t}_{k} \right)}{\partial{t}^{2}} \right| = \frac{1}{{\tau}^{2}}\max\limits_{0 \leq x \leq \ell}\left| {\widetilde R}_{3}\left( x,{t}_{k - 1} \right) + {\widetilde R}_{3}\left( x,{t}_{k + 1} \right) \right|\nonumber \\
        &\leq \frac{1}{{\tau}^{2}}\left( \max\limits_{0 \leq x \leq \ell}\left| {\widetilde R}_{3}\left( x,{t}_{k - 1} \right) \right| + \max\limits_{0 \leq x \leq \ell}\left| {\widetilde R}_{3}\left( x,{t}_{k + 1} \right) \right| \right) \leq \frac{{\c}_{20}}{12}{\tau}^{2}\,.
    \end{align}
    To obtain an estimation for ${R}_{2,k}\left( x,\tau \right)$, it is crucial to consider the Taylor series for the function ${u}\left( x,t \right)$ around the point $t = {t}_{k}$ while retaining the first two terms, {\ie}
    \begin{equation}\label{eq:theorem_taylor_series_two_terms}
        {u}\left( x,t \right) = {u}\left( x,{t}_{k} \right) + {\left( t - {t}_{k} \right)}{u}^{{\prime}}\left( x,{t}_{k} \right) + {\widetilde R}_{1}\left( x,t \right)\,,
    \end{equation}
    the remainder term ${\widetilde R}_{1}\left( x,t \right)$ can be expressed in integral form as follows
    \begin{equation*}
        {\widetilde R}_{1}\left( x,t \right) = \int\limits_{{t}_{k}}^{{t}}{\left( {u}^{{\prime}}\left( x,s \right) - {u}^{{\prime}}\left( x,{t}_{k} \right) \right)}{\d s}\,.
    \end{equation*}
    By virtue of condition \ref{itm_theorem_c} in \hyperref[theorem:theorem1]{\bf Theorem \ref*{theorem:theorem1}} the following inequality can be readily obtained
    \begin{equation}\label{eq:theorem_est_lagrange_rem_single_integ}
        \max\limits_{0 \leq x \leq \ell}\left| \frac{\d ^{2}}{\d {x}^{2}}{\widetilde R}_{1}\left( x,t \right) \right| \leq \frac{{\c}_{21}}{2}{\left( t - {t}_{k} \right)}^{2}\,.
    \end{equation}
    We substitute $t = {t}_{k - 1}$ and $t = {t}_{k + 1}$ in equality \eqref{eq:theorem_taylor_series_two_terms}, sum and arrange them, which yields
    \begin{equation*}
        {\Delta}^{2}u\left( x,{t}_{k - 1} \right) = {\widetilde R}_{1}\left( x,{t}_{k - 1} \right) + {\widetilde R}_{1}\left( x,{t}_{k + 1} \right)\,.
    \end{equation*}
    Through \eqref{eq:theorem_est_lagrange_rem_single_integ}, the last equality can be further estimated as follows
    \begin{align*}
        \max\limits_{0 \leq x \leq \ell}\left| \frac{\d ^{2}}{\d {x}^{2}}{\Delta}^{2}u\left( x,{t}_{k - 1} \right) \right| \leq \max\limits_{0 \leq x \leq \ell}\left| \frac{\d ^{2}}{\d {x}^{2}}{\widetilde R}_{1}\left( x,{t}_{k - 1} \right) \right| + \max\limits_{0 \leq x \leq \ell}\left| \frac{\d ^{2}}{\d {x}^{2}}{\widetilde R}_{1}\left( x,{t}_{k + 1} \right) \right| \leq {\c}_{21}{\tau}^{2}\,,
    \end{align*}
    from here
    \begin{equation}\label{eq:theorem_abs_r2k}
        \max\limits_{0 \leq x \leq \ell}\left| {R}_{2,k}\left( x,\tau \right) \right| = \frac{1}{2}{q}\left( {t}_{k} \right) \max\limits_{0 \leq x \leq \ell}\left| \frac{\d ^{2}}{\d {x}^{2}}{\Delta}^{2}u\left( x,{t}_{k - 1} \right) \right| \leq \frac{1}{2}{q}\left( {t}_{k} \right){\c}_{21}{\tau}^{2}\,.
    \end{equation}
    Thus, by combining \eqref{eq:theorem_abs_r1k} and \eqref{eq:theorem_abs_r2k}, we arrive at the following result
    \begin{align}\label{eq:theorem_main_remainder_term}
        \left\| {R}_{k}\left( x,{\tau} \right) \right\| &\leq \left\| {R}_{1,k}\left( x,{\tau} \right) \right\| + \left\| {R}_{2,k}\left( x,{\tau} \right) \right\| \nonumber \\
        &\leq \frac{\sqrt{\ell}}{2}\left( \frac{{\c}_{20}}{6} + {\max\limits_{0 \leq t \leq {T}}}{q}\left( t \right){\c}_{21} \right){\tau}^{2} = {\c}_{22}{\tau}^{2}\,.
    \end{align}

    Considering the inner product for both sides of the equality \eqref{eq:theorem_semidiscrete_scheme_for_z_k} with ${z}_{k + 1} - {z}_{k - 1} = {\Delta}{z}_{k} + {\Delta}{z}_{k - 1}$, we have
    \begin{equation}\label{eq:theorem_inner_product_z_k_eqt}
        {\left\| \frac{{\Delta}{z}_{k}}{{\tau}} \right\|}^{2} + \frac{1}{2}{q}_{k}{\left\| \frac{{\d}{z}_{k + 1}}{{\d}{x}} \right\|}^{2} = {\left\| \frac{{\Delta}{z}_{k - 1}}{{\tau}} \right\|}^{2} + \frac{1}{2}{q}_{k}{\left\| \frac{{\d}{z}_{k - 1}}{{\d}{x}} \right\|}^{2} + \left( {g}_{k}, {\Delta}{z}_{k} + {\Delta}{z}_{k - 1} \right)\,.
    \end{equation}
    Introducing the denotations
    \begin{gather*}
        {\overline{\mu}}_{k} = \left\| \frac{{\Delta}{z}_{k - 1}}{{\tau}} \right\|\,,\quad {\overline{\vartheta}}_{k} = \left\| \frac{{\d}{z}_{k}}{{\d}{x}} \right\|\,,\quad {\vartheta}_{k} = \left\| \frac{{\d}{u}_{k}}{{\d}{x}} \right\|\,,\quad {\overline{\delta}}_{k} = \left( {g}_{k}, {\Delta}{z}_{k} + {\Delta}{z}_{k - 1} \right)\,, \\
        {q}_{k} = {\alpha}_{k} + {\beta}_{k}{\vartheta}_{k}^{2}\,.
    \end{gather*}
    Thus, the equality expressed in \eqref{eq:theorem_inner_product_z_k_eqt} should be rewritten in the following manner
    \begin{equation*}
        {\overline{\mu}}_{k + 1}^{2} + \frac{1}{2}\left( {\alpha}_{k} + {\beta}_{k}{\vartheta}_{k}^{2} \right){\overline{\vartheta}}_{k + 1}^{2} = {\overline{\mu}}_{k}^{2} + \frac{1}{2}\left( {\alpha}_{k} + {\beta}_{k}{\vartheta}_{k}^{2} \right){\overline{\vartheta}}_{k - 1}^{2} + {\overline{\delta}}_{k}\,.
    \end{equation*}
    Finally, by using the last equality, we obtain
    \begin{equation}\label{eq:theorem_lambda_eqt}
        {\overline{\lambda}}_{k + 1} = {\overline{\lambda}}_{k} + \left( {\overline{\varepsilon}}_{k} + {\overline{\delta}}_{k} \right)\,,
    \end{equation}
    where
    \begin{align*}
        {\overline{\lambda}}_{k} &= {\overline{\mu}}_{k}^{2} + \frac{1}{2}\left( {\alpha}_{k - 1} + {\beta}_{k - 1}{\vartheta}_{k - 1}^{2} \right){\overline{\vartheta}}_{k}^{2}\,, \\
        {\overline{\varepsilon}}_{k} &= \frac{1}{2}\left( {\alpha}_{k} + {\beta}_{k}{\vartheta}_{k}^{2} \right){\overline{\vartheta}}_{k - 1}^{2} - \frac{1}{2}\left( {\alpha}_{k - 1} + {\beta}_{k - 1}{\vartheta}_{k - 1}^{2} \right){\overline{\vartheta}}_{k}^{2} \\
        &= \frac{1}{2}\left( {\alpha}_{k}{\overline{\vartheta}}_{k - 1}^{2} - {\alpha}_{k - 1}{\overline{\vartheta}}_{k}^{2} \right) + \frac{1}{2}\left( {\beta}_{k}{\vartheta}_{k}^{2}{\overline{\vartheta}}_{k - 1}^{2} - {\beta}_{k - 1}{\vartheta}_{k - 1}^{2}{\overline{\vartheta}}_{k}^{2} \right)\,.
    \end{align*}
    From \eqref{eq:theorem_lambda_eqt}, we have
    \begin{align}\label{eq:theorem_lambda_sum_eqt}
        {\overline{\lambda}}_{k + 1} &= {\overline{\lambda}}_{1} + \sum\limits_{i = 1}^{k}{\left( {\overline{\varepsilon}}_{i} + {\overline{\delta}}_{i} \right)}\nonumber \\
        &= {\overline{\lambda}}_{1} + \frac{1}{2}\sum\limits_{i = 1}^{k}{\left( {\alpha}_{i}{\overline{\vartheta}}_{i - 1}^{2} - {\alpha}_{i - 1}{\overline{\vartheta}}_{i}^{2} \right)} + \frac{1}{2}\sum\limits_{i = 1}^{k}{\left( {\beta}_{i}{\vartheta}_{i}^{2}{\overline{\vartheta}}_{i - 1}^{2} - {\beta}_{i - 1}{\vartheta}_{i - 1}^{2}{\overline{\vartheta}}_{i}^{2} \right)} + \sum\limits_{i = 1}^{k}{{\overline{\delta}}_{i}}\,.
    \end{align}
    The following representations are faithful
    \begin{align}
        &\sum\limits_{i = 1}^{k}{\left( {\alpha}_{i}{\overline{\vartheta}}_{i - 1}^{2} - {\alpha}_{i - 1}{\overline{\vartheta}}_{i}^{2} \right)} = \sum\limits_{i = 1}^{k}{\left( {\alpha}_{i}{\overline{\vartheta}}_{i - 1}^{2} - {\alpha}_{i + 1}{\overline{\vartheta}}_{i}^{2} \right)} + \sum\limits_{i = 1}^{k}{\left( {\alpha}_{i + 1} - {\alpha}_{i - 1} \right){\overline{\vartheta}}_{i}^{2}}\nonumber \\
        &= {\alpha}_{1}{\overline{\vartheta}}_{0}^{2} - {\alpha}_{k + 1}{\overline{\vartheta}}_{k}^{2} + \sum\limits_{i = 1}^{k}{\left( {\alpha}_{i + 1} - {\alpha}_{i - 1} \right){\overline{\vartheta}}_{i}^{2}}\,.\label{eq:theorem_sum_alpha_varrho} \\
        &\sum\limits_{i = 1}^{k}{\left( {\beta}_{i}{\vartheta}_{i}^{2}{\overline{\vartheta}}_{i - 1}^{2} - {\beta}_{i - 1}{\vartheta}_{i - 1}^{2}{\overline{\vartheta}}_{i}^{2} \right)} = {\beta}_{1}{\vartheta}_{1}^{2}{\overline{\vartheta}}_{0}^{2} + \sum\limits_{i = 1}^{k - 1}{\left( {\beta}_{i + 1}{\vartheta}_{i + 1}^{2} - {\beta}_{i - 1}{\vartheta}_{i - 1}^{2} \right){\overline{\vartheta}}_{i}^{2}} - {\beta}_{k - 1}{\vartheta}_{k - 1}^{2}{\overline{\vartheta}}_{k}^{2}\nonumber \\
        &= {\beta}_{1}{\vartheta}_{1}^{2}{\overline{\vartheta}}_{0}^{2} + \sum\limits_{i = 1}^{k - 1}{\left( {\beta}_{i + 1} - {\beta}_{i - 1} \right){\vartheta}_{i + 1}^{2}{\overline{\vartheta}}_{i}^{2}} + \sum\limits_{i = 1}^{k - 1}{{\beta}_{i - 1}\left( {\vartheta}_{i + 1}^{2} - {\vartheta}_{i - 1}^{2} \right){\overline{\vartheta}}_{i}^{2}} - {\beta}_{k - 1}{\vartheta}_{k - 1}^{2}{\overline{\vartheta}}_{k}^{2}\,.\label{eq:theorem_sum_sum_beta_varrho}
    \end{align}
    Through the insertion of equalities \eqref{eq:theorem_sum_alpha_varrho} and \eqref{eq:theorem_sum_sum_beta_varrho} into \eqref{eq:theorem_lambda_sum_eqt}, we derive that
    \begin{align}\label{eq:theorem_expand_sum_lambda_eqt}
        &{\overline{\lambda}}_{k + 1} + \frac{1}{2}\left( {\alpha}_{k + 1} + {\beta}_{k - 1}{\vartheta}_{k - 1}^{2} \right){\overline{\vartheta}}_{k}^{2} = {\overline{\lambda}}_{1} + \frac{1}{2}\left( {\alpha}_{1} + {\beta}_{1}{\vartheta}_{1}^{2} \right){\overline{\vartheta}}_{0}^{2}\nonumber \\
        &+ \frac{1}{2}\sum\limits_{i = 1}^{k}{\left( {\alpha}_{i + 1} - {\alpha}_{i - 1} \right){\overline{\vartheta}}_{i}^{2}} + \frac{1}{2}\sum\limits_{i = 1}^{k - 1}{\left( {\beta}_{i + 1} - {\beta}_{i - 1} \right){\vartheta}_{i + 1}^{2}{\overline{\vartheta}}_{i}^{2}} + \frac{1}{2}\sum\limits_{i = 1}^{k - 1}{{\beta}_{i - 1}\left( {\vartheta}_{i + 1}^{2} - {\vartheta}_{i - 1}^{2} \right){\overline{\vartheta}}_{i}^{2}} + \sum\limits_{i = 1}^{k}{{\overline{\delta}}_{i}}\,.
    \end{align}
    By virtue of equation \eqref{eq:lemma1_alpha_ineqt}, we arrive at the straightforward conclusion that
    \begin{equation}\label{eq:theorem_sum_alpha_ineqt}
        \frac{1}{2}\sum\limits_{i = 1}^{k}{\left( {\alpha}_{i + 1} - {\alpha}_{i - 1} \right){\overline{\vartheta}}_{i}^{2}} \leq {\c}_{2}{\tau}\sum\limits_{i = 1}^{k}{{\overline{\vartheta}}_{i}^{2}}\,,\quad {\c}_{2} = \max\limits_{0 \leq t \leq {T}}{\left| {\alpha}^{\prime}\left( t \right) \right|}\,.
    \end{equation}
    As a consequence of \eqref{eq:lemma1_beta_ineqt} and the \hyperref[lemma:lemma1]{\bf Lemma \ref*{lemma:lemma1}} (see also the \hyperref[rmk:remark-lemma1]{\bf Remark \ref*{rmk:remark-lemma1}}), we find that
    \begin{equation}\label{eq:theorem_sum_beta_ineqt}
        \frac{1}{2}\sum\limits_{i = 1}^{k - 1}{\left( {\beta}_{i + 1} - {\beta}_{i - 1} \right){\vartheta}_{i + 1}^{2}{\overline{\vartheta}}_{i}^{2}} \leq {\c}_{3}{\M}_{2}^{2}{\tau}\sum\limits_{i = 1}^{k - 1}{{\overline{\vartheta}}_{i}^{2}}\,,\quad {\c}_{3} = \max\limits_{0 \leq t \leq {T}}{\left| {\beta}^{\prime}\left( t \right) \right|}\,.
    \end{equation}
    According to \hyperref[lemma:lemma1]{\bf Lemma \ref*{lemma:lemma1}} and the \hyperref[lemma:lemma2]{\bf Lemma \ref*{lemma:lemma2}}, an estimate for the difference $\displaystyle {\vartheta}_{i + 1}^{2} - {\vartheta}_{i - 1}^{2}$ can be derived
    \begin{align}\label{eq:theorem_varrho_square_inqt}
        \left| {\vartheta}_{i + 1}^{2} - {\vartheta}_{i - 1}^{2} \right| &\leq \left| {\vartheta}_{i + 1} - {\vartheta}_{i - 1} \right|\left( {\vartheta}_{i + 1} + {\vartheta}_{i - 1} \right)\nonumber \\
        &\leq \left( \left| {\vartheta}_{i + 1} - {\vartheta}_{i} \right| + \left| {\vartheta}_{i} - {\vartheta}_{i - 1} \right| \right)\left( {\vartheta}_{i + 1} + {\vartheta}_{i - 1} \right)\nonumber \\
        &\leq {\tau}\left( \left\| {\frac{\d}{{\d}{x}}} \left( \frac{\Delta{u}_{i}}{\tau} \right) \right\| + \left\| {\frac{\d}{{\d}{x}}} \left( \frac{\Delta{u}_{i - 1}}{\tau} \right) \right\| \right)\left( \left\| \frac{{\d}{u}_{i + 1}}{{\d}{x}} \right\| + \left\| \frac{{\d}{u}_{i - 1}}{{\d}{x}} \right\| \right)\nonumber \\
        &\leq {4}{\M}_{2}{\M}_{3}{\tau}\,.
    \end{align}
    By virtue of equality \eqref{eq:theorem_varrho_square_inqt}, it can be deduced that
    \begin{equation}\label{eq:theorem_beta_varrho_square_inqt}
        \frac{1}{2}\sum\limits_{i = 1}^{k - 1}{{\beta}_{i - 1}\left( {\vartheta}_{i + 1}^{2} - {\vartheta}_{i - 1}^{2} \right){\overline{\vartheta}}_{i}^{2}} \leq {2}{\M}_{2}{\M}_{3}{\c}_{23}{\tau}\sum\limits_{i = 1}^{k - 1}{{\overline{\vartheta}}_{i}^{2}}\,,\quad {\c}_{23} = \max\limits_{0 \leq t \leq {T}}{\beta\left( t \right)}\,.
    \end{equation}
    By employing the Cauchy-Schwarz inequality, we can find that
    \begin{align*}
        \left| {\overline{\delta}}_{i} \right| &\leq \left\| {g}_{i} \right\|\left( \left\| {\Delta}{z}_{i} \right\| + \left\| {\Delta}{z}_{i - 1} \right\| \right) = {\tau}\left\| {g}_{i} \right\|\left( {\overline{\mu}}_{i + 1} + {\overline{\mu}}_{i} \right) \\
        &\leq {\tau}\left( \frac{1}{2}\left| {q}\left( {t}_{i} \right) - {q}_{i} \right|\left( \left\| \frac{{\d}^{2}{u}\left( \cdot,{t}_{i + 1} \right)}{{\d}{x}^{2}} \right\| + \left\| \frac{{\d}^{2}{u}\left( \cdot,{t}_{i - 1} \right)}{{\d}{x}^{2}} \right\| \right) + \left\| {R}_{i}\left( \cdot,{\tau} \right) \right\| \right)\left( {\overline{\mu}}_{i + 1} + {\overline{\mu}}_{i} \right)\,.
    \end{align*}
    We shall proceed to estimate the given difference while taking into consideration the implications of \hyperref[lemma:lemma1]{\bf Lemma \ref*{lemma:lemma1}}
    \begin{align*}
        \frac{1}{2}\left| {q}\left( {t}_{i} \right) - {q}_{i} \right| &= \frac{1}{2}{\beta}_{i}\left| {\left\| \frac{{\d}u\left( \cdot,{t}_{i} \right)}{{\d}{x}} \right\|}^{2} - {\left\| \frac{{\d}{u}_{i}}{{\d}{x}} \right\|}^{2} \right| = \frac{1}{2}{\beta}_{i}\left| \left\| \frac{{\d}u\left( \cdot,{t}_{i} \right)}{{\d}{x}} \right\| - \left\| \frac{{\d}{u}_{i}}{{\d}{x}} \right\| \right|\left( \left\| \frac{{\d}u\left( \cdot,{t}_{i} \right)}{{\d}{x}} \right\| + \left\| \frac{{\d}{u}_{i}}{{\d}{x}} \right\| \right) \\
        &\leq \frac{1}{2}{\beta}_{i}\left\| \frac{{\d}{z}_{i}}{{\d}{x}} \right\|\left( \left\| \frac{{\d}u\left( \cdot,{t}_{i} \right)}{{\d}{x}} \right\| + \left\| \frac{{\d}{u}_{i}}{{\d}{x}} \right\| \right) \leq \frac{1}{2}{\c}_{23}\left( {\c}_{24} + {\M}_{2} \right){{\overline{\vartheta}}_{i}} = {\c}_{25}{{\overline{\vartheta}}_{i}}\,.
    \end{align*}
    Utilizing the last inequality and combining it with \eqref{eq:theorem_main_remainder_term} for ${\overline{\delta}}_{i}$, we arrive at the following result
    \begin{align}\label{eq:theorem_overline_delta_i}
        \left| {\overline{\delta}}_{i} \right| &\leq {\tau}\left( {\c}_{25}{\c}_{26}{{\overline{\vartheta}}_{i}} + {\c}_{22}{\tau}^{2} \right)\left( {\overline{\mu}}_{i + 1} + {\overline{\mu}}_{i} \right)\nonumber \\
        &\leq \max\left( {\c}_{25}{\c}_{26},{\c}_{22} \right){\tau}\left( {\overline{\vartheta}}_{i} + {\tau}^{2} \right)\left( {\overline{\mu}}_{i + 1} + {\overline{\mu}}_{i} \right)\nonumber \\
        &={\c}_{27}{\tau}\left( {\overline{\vartheta}}_{i} + {\tau}^{2} \right)\left( {\overline{\mu}}_{i + 1} + {\overline{\mu}}_{i} \right)\nonumber \\
        &= {\c}_{27}{\tau}\left[ \left( {\overline{\vartheta}}_{i}{\overline{\mu}}_{i + 1} + {\overline{\vartheta}}_{i}{\overline{\mu}}_{i} \right) + {\tau}^{2}\left( {\overline{\mu}}_{i + 1} + {\overline{\mu}}_{i} \right) \right]\nonumber \\
        &\leq {\c}_{27}{\tau}\left[ \frac{1}{2}\left( {\overline{\vartheta}}_{i}^{2} + {\overline{\mu}}_{i + 1}^{2} \right) + \frac{1}{2}\left( {\overline{\vartheta}}_{i}^{2} + {\overline{\mu}}_{i}^{2} \right) + \frac{1}{2}\left( {\tau}^{4} + {\left( {\overline{\mu}}_{i + 1} + {\overline{\mu}}_{i} \right)}^{2} \right) \right]\nonumber \\
        &= {\c}_{27}{\tau}\left[ {\overline{\vartheta}}_{i}^{2} + \left( {\overline{\mu}}_{i + 1}^{2} + {\overline{\mu}}_{i}^{2} \right) + \frac{1}{2}{\tau}^{4} + {\overline{\mu}}_{i + 1}{\overline{\mu}}_{i} \right]\nonumber \\
        &\leq {\c}_{27}{\tau}\left[ {\overline{\vartheta}}_{i}^{2} + \left( {\overline{\mu}}_{i + 1}^{2} + {\overline{\mu}}_{i}^{2} \right) + \frac{1}{2}{\tau}^{4} + \frac{1}{2}\left( {\overline{\mu}}_{i + 1}^{2} + {\overline{\mu}}_{i}^{2} \right) \right]\nonumber \\
        &= {\c}_{27}{\tau}\left( {\overline{\vartheta}}_{i}^{2} + \frac{3}{2}{\overline{\mu}}_{i + 1}^{2} + \frac{3}{2}{\overline{\mu}}_{i}^{2} + \frac{1}{2}{\tau}^{4} \right) \leq {\c}_{28}{\tau}\left( {\overline{\vartheta}}_{i}^{2} + {\overline{\mu}}_{i}^{2} + {\overline{\mu}}_{i + 1}^{2} \right) + {\c}_{28}{\tau}^{5}\,.
    \end{align}
    In order to obtain the desired estimate for ${\overline{\delta}}_{i}$, it is necessary to use the following inequalities. It is important to note that ${\alpha}\left( t \right) \geq {\c}_{0} > 0$. Therefore, we have:
    \begin{align*}
        {\overline{\lambda}}_{i} &= {\overline{\mu}}_{i}^{2} + \frac{1}{2}\left( {\alpha}_{i - 1} + {\beta}_{i - 1}{\vartheta}_{i - 1}^{2} \right){\overline{\vartheta}}_{i}^{2} \geq {\overline{\mu}}_{i}^{2} + \frac{{\c}_{0}}{2}{\overline{\vartheta}}_{i}^{2} \geq {\c}_{29}\left( {\overline{\mu}}_{i}^{2} + {\overline{\vartheta}}_{i}^{2} \right)\,,\quad {\c}_{29} = \min\left( 1,\frac{{\c}_{0}}{2} \right)\,,
    \end{align*}
    consequently, it can be inferred that
    \begin{equation}\label{eq:theorem_sum_mu_square_and_varrho_square}
        {\overline{\mu}}_{i}^{2} + {\overline{\vartheta}}_{i}^{2} \leq \frac{1}{{\c}_{29}}{\overline{\lambda}}_{i}\,.
    \end{equation}
    It is evident that the following inequality holds
    \begin{equation}\label{eq:theorem_overl_mu_leq_overl_lambda}
        {\overline{\mu}}_{i}^{2} \leq {\overline{\lambda}}_{i} = {\overline{\mu}}_{i}^{2} + \frac{1}{2}\left( {\alpha}_{i - 1} + {\beta}_{i - 1}{\vartheta}_{i - 1}^{2} \right){\overline{\vartheta}}_{i}^{2}\,.
    \end{equation}
    Taking into consideration the estimates \eqref{eq:theorem_sum_mu_square_and_varrho_square} and \eqref{eq:theorem_overl_mu_leq_overl_lambda}, we can obtain from \eqref{eq:theorem_overline_delta_i}:
    \begin{align}\label{eq:theorem_overal_detla_final_ineqt}
        \left| {\overline{\delta}}_{i} \right| &\leq {\c}_{28}{\tau}\left( \left( {\overline{\mu}}_{i}^{2} + {\overline{\vartheta}}_{i}^{2} \right) + {\overline{\mu}}_{i + 1}^{2} \right) + {\c}_{28}{\tau}^{5} \leq {\c}_{28}{\tau}\left( \frac{1}{{\c}_{29}}{\overline{\lambda}}_{i} + {\overline{\lambda}}_{i + 1} \right) + {\c}_{28}{\tau}^{5}\nonumber \\
        &\leq {\c}_{28}\max\left( 1,\frac{1}{{\c}_{29}} \right){\tau}\left( {\overline{\lambda}}_{i} + {\overline{\lambda}}_{i + 1} \right) + {\c}_{28}{\tau}^{5}\nonumber \\
        &= {\c}_{28}\max\left( 1,\frac{2}{{\c}_{0}} \right){\tau}\left( {\overline{\lambda}}_{i} + {\overline{\lambda}}_{i + 1} \right) + {\c}_{28}{\tau}^{5}\nonumber \\
        &= {\c}_{30}{\tau}\left( {\overline{\lambda}}_{i} + {\overline{\lambda}}_{i + 1} \right) + {\c}_{28}{\tau}^{5}\,.
    \end{align}
    Incorporating the inequalities \eqref{eq:theorem_sum_alpha_ineqt}, \eqref{eq:theorem_sum_beta_ineqt}, \eqref{eq:theorem_beta_varrho_square_inqt} and \eqref{eq:theorem_overal_detla_final_ineqt} into \eqref{eq:theorem_expand_sum_lambda_eqt}, we arrive at
    \begin{align*}
        &{\overline{\lambda}}_{k + 1} + \frac{1}{2}\left( {\alpha}_{k + 1} + {\beta}_{k - 1}{\vartheta}_{k - 1}^{2} \right){\overline{\vartheta}}_{k}^{2} \leq {\overline{\lambda}}_{1} + \frac{1}{2}\left( {\alpha}_{1} + {\beta}_{1}{\vartheta}_{1}^{2} \right){\overline{\vartheta}}_{0}^{2} \\
        &+ {\c}_{2}{\tau}\sum\limits_{i = 1}^{k}{{\overline{\vartheta}}_{i}^{2}} + \left( {\c}_{3}{\M}_{2}^{2} + {2}{\M}_{2}{\M}_{3}{\c}_{23} \right){\tau}\sum\limits_{i = 1}^{k - 1}{{\overline{\vartheta}}_{i}^{2}} + {\c}_{30}{\tau}\sum\limits_{i = 1}^{k}{\left( {\overline{\lambda}}_{i} + {\overline{\lambda}}_{i + 1} \right)} + {\c}_{28}{k}{\tau}^{5} \\
        &\leq {\overline{\lambda}}_{1} + \frac{1}{2}\left( {\alpha}_{1} + {\beta}_{1}{\vartheta}_{1}^{2} \right){\overline{\vartheta}}_{0}^{2} + \left( {\c}_{2} + {\c}_{3}{\M}_{2}^{2} + {2}{\M}_{2}{\M}_{3}{\c}_{23} \right){\tau}\sum\limits_{i = 1}^{k}{{\overline{\vartheta}}_{i}^{2}} + {\c}_{30}{\tau}\sum\limits_{i = 1}^{k}{\left( {\overline{\lambda}}_{i} + {\overline{\lambda}}_{i + 1} \right)} + {\c}_{28}{k}{\tau}^{5} \\
        &= {\overline{\lambda}}_{1} + \frac{1}{2}\left( {\alpha}_{1} + {\beta}_{1}{\vartheta}_{1}^{2} \right){\overline{\vartheta}}_{0}^{2} + {\c}_{31}{\tau}\sum\limits_{i = 1}^{k}{{\overline{\vartheta}}_{i}^{2}} + {\c}_{30}{\tau}\sum\limits_{i = 1}^{k}{\left( {\overline{\lambda}}_{i} + {\overline{\lambda}}_{i + 1} \right)} + {\c}_{28}{k}{\tau}^{5}\,.
    \end{align*}
    The following three simple estimates can be derived:
    \begin{gather*}
        {\overline{\vartheta}}_{i}^{2} \leq \frac{2}{{\c}_{0}}{\overline{\lambda}}_{i} = {2}{\c}_{7}{\overline{\lambda}}_{i}\,, \\
        {\c}_{30}{\tau}\sum\limits_{i = 1}^{k}{\left( {\overline{\lambda}}_{i} + {\overline{\lambda}}_{i + 1} \right)} = {\c}_{30}{\tau}\left( {\overline{\lambda}}_{1} + {2}\sum\limits_{i = 2}^{k}{{\overline{\lambda}}_{i}} + {\overline{\lambda}}_{k + 1} \right) \leq {2}{\c}_{30}{\tau}\sum\limits_{i = 1}^{k + 1}{{\overline{\lambda}}_{i}}\,, \\
        {\c}_{28}{k}{\tau}^{5} = {\c}_{28}\left( {k}{\tau} \right){\tau}^{4} \leq {\c}_{28}{\overline{T}}{\tau}^{4}\,.
    \end{gather*}
    Taking into account the above-mentioned inequalities, it can be concluded that
    \begin{align*}
        {\overline{\lambda}}_{k + 1} + \frac{1}{2}\left( {\alpha}_{k + 1} + {\beta}_{k - 1}{\vartheta}_{k - 1}^{2} \right){\overline{\vartheta}}_{k}^{2} &\leq {\overline{\lambda}}_{1} + \frac{1}{2}\left( {\alpha}_{1} + {\beta}_{1}{\vartheta}_{1}^{2} \right){\overline{\vartheta}}_{0}^{2} \\
        &+ {2}\left( {\c}_{7}{\c}_{31} + {\c}_{30} \right){\tau}\sum\limits_{i = 1}^{k}{{\overline{\lambda}}_{i}} + {2}{\c}_{30}{\tau}{\overline{\lambda}}_{k + 1} + {\c}_{28}{\overline{T}}{\tau}^{4} \\
        &= {\c}_{32}{\tau}\sum\limits_{i = 1}^{k}{{\overline{\lambda}}_{i}} + {\c}_{33}{\tau}{\overline{\lambda}}_{k + 1} + {\c}_{28}{\overline{T}}{\tau}^{4}\,.
    \end{align*}
    Therefore, we have
    \begin{equation}\label{eq:theorem_final_overl_lambda_inqt}
        \left( 1 - {\c}_{33}{\tau} \right){\overline{\lambda}}_{k + 1} \leq \overline{\alpha} + {\c}_{32}{\tau}\sum\limits_{i = 1}^{k}{{\overline{\lambda}}_{i}}\,,
    \end{equation}
    where
    \begin{equation*}
        \overline{\alpha} = {\overline{\lambda}}_{1} + \frac{1}{2}\left( {\alpha}_{1} + {\beta}_{1}{\vartheta}_{1}^{2} \right){\overline{\vartheta}}_{0}^{2} + {\c}_{28}{\overline{T}}{\tau}^{4}\,.
    \end{equation*}
    Assuming that ${\c}_{33}{\tau} < 1$, we can derive from \eqref{eq:theorem_final_overl_lambda_inqt} that
    \begin{equation}\label{eq:theorem_the_previous_inqt_gronwall_lemma}
        {\overline{\lambda}}_{k + 1} \leq {\widetilde \alpha} + {\widetilde c}{\tau}\sum\limits_{i = 1}^{k}{{\overline{\lambda}}_{i}}\,,
    \end{equation}
    here
    \begin{equation*}
        {\widetilde \alpha} = \frac{\overline{\alpha}}{1 - {\c}_{33}{\tau}}\,,\quad {\widetilde c} = \frac{{\c}_{32}}{1 - {\c}_{33}{\tau}}\,.
    \end{equation*}
    From \eqref{eq:theorem_the_previous_inqt_gronwall_lemma} by applying \hyperref[lemma:gronwall-inequality]{\bf Lemma \ref*{lemma:gronwall-inequality} (Discrete Gr\"{o}nwall-type inequality)} along with \hyperref[rmk:remark1]{\bf Remark \ref*{rmk:remark1}} we get
    \begin{equation}\label{eq:theorem_gronwall_lemma}
        {\overline{\lambda}}_{k + 1} \leq {\e}^{{\tilde c}{t}_{k}}{\widetilde \alpha} \leq {\e}^{{\tilde c}\overline{T}}{\widetilde \alpha}\,.
    \end{equation}
    To estimate ${\widetilde \alpha}$, we need to first estimate ${\overline{\lambda}}_{1}$. For this purpose, we use conditions \ref{itm_theorem_a} and \ref{itm_theorem_b} of \hyperref[theorem:theorem1]{\bf Theorem \ref*{theorem:theorem1}}, and consider the Taylor expansion of the function ${u}\left( x,{t}_{1} \right)$ around the point ${t} = {0}$ with respect to the temporal variable, keeping the first three terms. We then apply equation \eqref{eq:main_eqt} and the initial conditions \eqref{eq:initial_conds} for the functions ${u}\left( x,0 \right)$, ${u}^{{\prime}}\left( x,0 \right)$ and ${u}^{{\prime}{\prime}}\left( x,0 \right)$, which gives us
    \begin{align*}
        {u}\left( x,{t}_{1} \right) &= {u}\left( x,0 \right) + {{t}_{1}}{u}^{{\prime}}\left( x,0 \right) + \frac{{t}_{1}^{2}}{2}{u}^{{\prime}{\prime}}\left( x,0 \right) + {\widetilde R}_{2}\left( x,{t}_{1} \right) \\
        &= {\psi}_{0}\left( x \right) + {\tau}{\psi}_{1}\left( x \right) + \frac{{\tau}^{2}}{2}{\psi}_{2}\left( x \right) + {\widetilde R}_{2}\left( x,{\tau} \right)\,,
    \end{align*}
    where
    \begin{equation*}
        {\psi}_{2}\left( x \right) = {f}_{0}\left( x \right) + {q}_{0}\frac{\d ^{2}{\psi}_{0}\left( x \right)}{\d {x}^{2}}\,,\quad {\widetilde R}_{2}\left( x,{\tau} \right) = \frac{1}{2}\int\limits_{0}^{{\tau}}{{\left( {\tau} - t \right)}^{2}{u}^{{\prime}{\prime}{\prime}}\left( x,t \right)}{\d t}\,,
    \end{equation*}
    on the other hand, recall that
    \begin{equation*}
        {u}_{1}\left( x \right) = {\psi}_{0}\left( x \right) + {\tau}{\psi}_{1}\left( x \right) + \frac{{\tau}^{2}}{2}{\psi}_{2}\left( x \right)\,.
    \end{equation*}
    We obtain the following estimation for ${z}_{1}\left( x \right)$
    \begin{align*}
        \max\limits_{0 \leq x \leq \ell}\left| {z}_{1}\left( x \right) \right| &= \max\limits_{0 \leq x \leq \ell}\left| {u}\left( x,{t}_{1} \right) - {u}_{1}\left( x \right) \right| = \max\limits_{0 \leq x \leq \ell}\left| {\widetilde R}_{2}\left( x,{\tau} \right) \right| \\
        &= \frac{1}{2}\max\limits_{0 \leq x \leq \ell}\left| \int\limits_{0}^{{\tau}}{{\left( {\tau} - t \right)}^{2}{u}^{{\prime}{\prime}{\prime}}\left( x,t \right)}{\d t} \right| \leq \frac{1}{2}\max\limits_{\left( x,t \right)}\left| {u}^{{\prime}{\prime}{\prime}}\left( x,t \right) \right|\int\limits_{0}^{{\tau}}{{\left( {\tau} - t \right)}^{2}}{\d t} \\
        &= \frac{1}{6}\max\limits_{\left( x,t \right)}\left| {u}^{{\prime}{\prime}{\prime}}\left( x,t \right) \right|{\tau}^{3} = {\c}_{34}{\tau}^{3}\,,
    \end{align*}
    from here, we find that
    \begin{equation}\label{eq:theorem_norm_z1}
        {\left\| {z}_{1} \right\|}^{2} = \int\limits_{0}^{{\ell}}{{\left( {u}\left( x,{t}_{1} \right) - {u}_{1}\left( x \right) \right)}^{2}}{\d x} \leq {\ell}{\c}_{34}^{2}{\tau}^{6}\,.
    \end{equation}
    In order to obtain an estimation for ${\overline{\vartheta}}_{1}^{2}$, we consider the Taylor series expansion of the function ${u}\left( x,{t}_{1} \right)$ about the point ${t} = {0}$ with respect to the temporal variable, but this time we keep only the first two terms. Similar to before, we apply the initial conditions \eqref{eq:initial_conds} for the functions ${u}\left( x,0 \right)$, ${u}^{{\prime}}\left( x,0 \right)$ to obtain the following expression, {\ie}
    \begin{align*}
        {u}\left( x,{t}_{1} \right) &= {u}\left( x,0 \right) + {{t}_{1}}{u}^{{\prime}}\left( x,0 \right) + {\widetilde R}_{1}\left( x,{t}_{1} \right) \nonumber \\
        &= {\psi}_{0}\left( x \right) + {\tau}{\psi}_{1}\left( x \right) + {\widetilde R}_{1}\left( x,{\tau} \right)\,,
    \end{align*}
    here
    \begin{equation*}
        {\widetilde R}_{1}\left( x,{\tau} \right) = \int\limits_{0}^{{\tau}}{{\left( {\tau} - t \right)}{u}^{{\prime}{\prime}}\left( x,t \right)}{\d t}\,.
    \end{equation*}
    It should be noted that ${u}_{1}\left( x \right)$ is an approximation of ${u}\left( x,{t}_{1} \right)$, {\ie} ${u}\left( x,{t}_{1} \right) \approx {u}_{1}\left( x \right) = {\psi}_{0}\left( x \right) + {\tau}{\psi}_{1}\left( x \right)$.
    \begin{align*}
        \max\limits_{0 \leq x \leq \ell}\left| \frac{{\d}{z}_{1}\left( x \right)}{{\d}{x}} \right| &= \max\limits_{0 \leq x \leq \ell}\left| \frac{{\d}}{{\d}{x}}{\left( {u}\left( x,{t}_{1} \right) - {u}_{1}\left( x \right) \right)} \right| = \max\limits_{0 \leq x \leq \ell}\left| \frac{{\d}}{{\d}{x}}{{\widetilde R}_{1}\left( x,{\tau} \right)} \right| \\
        &= \max\limits_{0 \leq x \leq \ell}\left| \frac{{\d}}{{\d}{x}}{\int\limits_{0}^{{\tau}}{{\left( {\tau} - t \right)}{u}^{{\prime}{\prime}}\left( x,t \right)}{\d t}} \right| \leq \max\limits_{\left( x,t \right)}{\left| {u}_{{x}{t}{t}}\left( x,t \right) \right|}{\int\limits_{0}^{{\tau}}{{\left( {\tau} - t \right)}}{\d t}} \\
        &\leq \frac{1}{2}\max\limits_{\left( x,t \right)}{\left| {u}_{{x}{t}{t}}\left( x,t \right) \right|}{\tau}^{2} = {\c}_{35}{\tau}^{2}\,,
    \end{align*}
    from here, we obtain
    \begin{equation}\label{eq:theorem_norm_of_second_deriv_of_z1}
        {\left\| \frac{{\d}{z}_{1}}{{\d}{x}} \right\|}^{2} = \int\limits_{0}^{{\ell}}{{\left[ \frac{{\d}}{{\d}{x}}\left( {u}\left( x,{t}_{1} \right) - {u}_{1}\left( x \right) \right) \right]}^{2}}{\d x} \leq {\ell}{{\c}_{35}^{2}{\tau}^{4}}\,.
    \end{equation}
    Based on \eqref{eq:theorem_norm_z1} and \eqref{eq:theorem_norm_of_second_deriv_of_z1}, we can derive the following estimations
    \begin{align}\label{eq:theorem_norm_of_overline_lambda_1}
        {\overline{\lambda}}_{1} &= {\overline{\mu}}_{1}^{2} + \frac{1}{2}\left( {\alpha}_{0} + {\beta}_{0}{\vartheta}_{0}^{2} \right){\overline{\vartheta}}_{1}^{2}\nonumber \\
        &= {\left\| \frac{{z}_{1} - {z}_{0}}{\tau} \right\|}^{2} + \frac{1}{2}\left( {\alpha}_{0} + {\beta}_{0}{\left\| \frac{{\d}{u}_{0}}{{\d}{x}} \right\|}^{2} \right){\left\| \frac{{\d}{z}_{1}}{{\d}{x}} \right\|}^{2}\nonumber \\
        &= \frac{1}{{\tau}^{2}}{\left\| {z}_{1} \right\|}^{2} + \frac{1}{2}\left( {\alpha}_{0} + {\beta}_{0}{\left\| \frac{{\d}{\psi}_{0}}{{\d}{x}} \right\|}^{2} \right){\left\| \frac{{\d}{z}_{1}}{{\d}{x}} \right\|}^{2}\nonumber \\
        &\leq {\ell}\left( {\c}_{34}^{2} + \frac{1}{2}\left( {\alpha}_{0} + {\beta}_{0}{{\c}_{36}^{2}} \right){\c}_{35}^{2} \right){\tau}^{4} = {\c}_{37}{\tau}^{4}\,,\quad {z}_{0}\left( x \right) = {u}\left( x,0 \right) - {u}_{0}\left( x \right) \equiv 0\,,
    \end{align}
    furthermore, it can be concluded from equation \eqref{eq:theorem_norm_of_overline_lambda_1} that
    \begin{align}\label{eq:theorem_norm_of_widetilde_alpha}
        {\widetilde \alpha} &= \frac{\overline{\alpha}}{1 - {\c}_{33}{\tau}} = \frac{1}{1 - {\c}_{33}{\tau}}\left( {\overline{\lambda}}_{1} + \frac{1}{2}\left( {\alpha}_{1} + {\beta}_{1}{\vartheta}_{1}^{2} \right){\overline{\vartheta}}_{0}^{2} + {\c}_{28}{\overline{T}}{\tau}^{4} \right)\nonumber \\
        &= \frac{1}{1 - {\c}_{33}{\tau}}\left( {\overline{\lambda}}_{1} + {\c}_{28}{\overline{T}}{\tau}^{4} \right) \leq \frac{{\c}_{37} + {\c}_{28}{\overline{T}}}{1 - {\c}_{33}{\tau}}{\tau}^{4} = {\c}_{38}{\tau}^{4}\,.
    \end{align}
    From \eqref{eq:theorem_gronwall_lemma}, taking the inequalities \eqref{eq:theorem_norm_of_overline_lambda_1} and \eqref{eq:theorem_norm_of_widetilde_alpha} into consideration, the estimates for \hyperref[theorem:theorem1]{\bf Theorem \ref*{theorem:theorem1}} are obtained.
\end{proof}
%
%According to the \hyperref[theorem:theorem1]{\bf Theorem \ref*{theorem:theorem1}}, the following statement is valid.
\begin{remark}\label{rmk:remark2}
    The error ${z}_{k}\left( x \right) = {u}\left( x,{t}_{k} \right) - {u}_{k}\left( x \right)$ of an approximate solution to the problem \eqref{eq:main_eqt}-\eqref{eq:boundary_conds} is bounded by the following inequality:
    \begin{equation}\label{eq:remark_estimate}
        \max\limits_{1 \leq {k} \leq m}{\left\| {z}_{k} \right\|} \leq {\M}_{7}{\tau}^{2}\,,
    \end{equation}
    where $\displaystyle {m} = \left[ \frac{\overline{T}}{{\tau}} \right]$, and $\overline{T}$ satisfies $0 < \overline{T} \leq {T}$.

    The validity of inequality \eqref{eq:remark_estimate} follows immediately from the identity
    \begin{equation*}
        {z}_{k}\left( x \right) = {z}_{0}\left( x \right) + {\tau}\sum\limits_{i = 0}^{k - 1}{\frac{{\Delta}{z}_{i}\left( x \right)}{\tau}}\,,\quad {z}_{0}\left( x \right) = {u}\left( x,0 \right) - {u}_{0}\left( x \right) \equiv 0\,,
    \end{equation*}
    by virtue of the second bound of \hyperref[theorem:theorem1]{\bf Theorem \ref*{theorem:theorem1}}.
\end{remark}
\section{Results of Numerical Computations}\label{sec:numresult}
\subsection{Designing a Spatial Fourth-Order Accuracy Three-Term Difference Scheme for Solving the Obtained Linear Ordinary Differential Equations}\label{subsec:design}
This subsection outlines the development of a three-term difference scheme that attains fourth-order accuracy over an equidistant space grid of length $h$. The main motivation for constructing this scheme lies in the self-contained nature of our paper, wherein we seek to provide a comprehensive treatment of numerical methods for the problem at hand. Notably, the technique used in this scheme is widely known and has been previously discussed in various literature ({\eg} see \cite{bahvalov1976}). Although the primary focus of our research pertains to the development of the symmetric three-layer semi-discrete scheme regarding the temporal variable, we also employ a combination of this scheme with a spatial fourth-order accuracy three-term difference scheme to obtain numerical results. It should be emphasized that the article \cite{RogavaTsiklauri_LocConvg2012} explores the same approach.

Observe that the symmetric three-layer semi-discrete scheme \eqref{eq:semidiscrete_scheme}, used in numerical solutions of the problem \eqref{eq:main_eqt}-\eqref{eq:boundary_conds}, yields second-order linear ordinary differential equations per temporal layer. These differential equations are given by \eqref{eq:discrt_operator_eqt}.

By introducing the notation
\begin{equation*}
    {\A} = {\I} - \frac{1}{2}{\tau}^{2}{q}_{k}\frac{\d ^{2}}{\d {x}^{2}}\,,\ \text{with}\ D\left( {\A} \right) = \left\{ u\left( x \right) \in {C}^{2}\left( \left[ 0,\ell \right] \right) \mid u\left( 0 \right) = u\left( \ell \right) = 0 \right\}\,,
\end{equation*}
one can express equation \eqref{eq:discrt_operator_eqt} in the following manner
\begin{equation*}
    {\A}{u}_{k + 1}\left( x \right) = {\tau}^{2}{f}_{k}\left( x \right) + {2}{u}_{k}\left( x \right) - {\A}{u}_{k - 1}\left( x \right)\,.
\end{equation*}
If the inverse of the operator ${\A}$ is applied to both sides of the preceding equation, the resulting expression can be obtained as follows
\begin{equation*}
    {u}_{k + 1}\left( x \right) = {\A}^{-1}{v}_{k}\left( x \right) - {u}_{k - 1}\left( x \right)\,,\quad \text{where}\quad {v}_{k}\left( x \right) = {\tau}^{2}{f}_{k}\left( x \right) + {2}{u}_{k}\left( x \right)\,.
\end{equation*}
Denoting ${w}_{k}\left( x \right) = {\A}^{-1}{v}_{k}\left( x \right)$, we have
\begin{equation*}
    {\A}{w}_{k}\left( x \right) = {v}_{k}\left( x \right)\,.
\end{equation*}
By setting the operator $\displaystyle {\A} = {\I} - \frac{1}{2}{\tau}^{2}{q}_{k}\frac{\d ^{2}}{\d {x}^{2}}$ in the above-mentioned equation, we can deduce that the resulting equation can be written in expanded form as follows
\begin{equation}\label{eq:numres_eqt_wk}
    {w}_{k}^{\prime\prime}\left( x \right) - {p}_{k}{w}_{k}\left( x \right) = {\tilde{v}}_{k}\left( x \right)\,,
\end{equation}
where $\displaystyle {p}_{k} = \frac{2}{{\tau}^{2}{q}_{k}}$ and ${\tilde{v}}_{k}\left( x \right) = -{p}_{k}{v}_{k}\left( x \right)$.

To obtain a fourth-order accuracy three-term difference scheme for the equations \eqref{eq:numres_eqt_wk}, we initially omit the temporal layer index $k$ to simplify the derivation process. Upon deriving the scheme, we subsequently reintroduce the temporal layer index $k$ into the equations. It should be noted that the resulting difference equations are dependent on two indices, namely $k$ and $i$, where $i$ represents the $i$-th node along the spatial grid.

Let $h$ denote the uniform spacing of a spatial grid defined over the domain $\left[ 0,\ell \right]$ with $m + 1$ equally spaced points. We define the partition of the domain as:
\begin{center}
    $0 = {x}_{0} < {x}_{1} < \cdots < {x}_{m} = \ell$, with ${x}_{i} = {i}{h}$ for $i = 0,1,\ldots,m$; here, $\displaystyle h = \frac{\ell}{m}$.
\end{center}

Let us consider the Taylor expansion of a function $w\left( x \right)$ about a point $x = {x}_{i}$, while maintaining the first six terms, {\ie}
\begin{equation*}
    w\left( x \right) = \sum\limits_{j = 0}^{5}{\frac{{w}^{\left( j \right)}\left( {x}_{i} \right)}{j!}{\left( x - {x}_{i} \right)}^{j}} + \bigO\left( {h}^{6} \right)\,.
\end{equation*}
If we substitute $x = {x}_{i-1}$ and $x = {x}_{i+1}$ into the Taylor expansion of the function $w\left( x \right)$ centred at ${x}_{i}$, then sum and rearrange the resulting expressions, we obtain
\begin{equation*}
    {w}^{\prime\prime}\left( {x}_{i} \right) = \frac{{\Delta}^{2} w\left( {x}_{i-1} \right)}{{h}^{2}} - \frac{{h}^{2}}{12}{w}^{\left( 4 \right)}\left( {x}_{i} \right) + \bigO\left( {h}^{4} \right)\,.
\end{equation*}
Consider the equation \eqref{eq:numres_eqt_wk} in the context of a set of points where $x = {x}_{i}$. By substituting the last equality for ${w}^{\prime\prime}\left( {x}_{i} \right)$, the equation \eqref{eq:numres_eqt_wk} can be reformulated as follows:
\begin{equation*}
    \frac{{\Delta}^{2} w\left( {x}_{i-1} \right)}{{h}^{2}} - \frac{{h}^{2}}{12}{w}^{\left( 4 \right)}\left( {x}_{i} \right) - {p}{w}\left( {x}_{i} \right) = \tilde{v}\left( {x}_{i} \right) + \bigO\left( {h}^{4} \right)\,.
\end{equation*}
By evaluating the second-order derivative of equation \eqref{eq:numres_eqt_wk} at the points $x = {x}_{i}$ and determining
\begin{equation*}
    {w}^{\left( 4 \right)}\left( {x}_{i} \right) = {p}{w}^{\prime\prime}\left( {x}_{i} \right) + \tilde{v}^{\prime\prime}\left( {x}_{i} \right) = {p}^{2}{w}\left( {x}_{i} \right) + {p}\tilde{v}\left( {x}_{i} \right) + \tilde{v}^{\prime\prime}\left( {x}_{i} \right)\,,
\end{equation*}
one can substitute the resulting value of ${w}^{\left( 4 \right)}\left( {x}_{i} \right)$ into the previous equation to draw the following conclusion
\begin{align*}
    \frac{{\Delta}^{2} w\left( {x}_{i-1} \right)}{{h}^{2}} - {p}\left( 1 + \frac{{h}^{2}}{12} {p} \right)w\left( {x}_{i} \right) = \left( 1 + \frac{{h}^{2}}{12} {p} \right)\tilde{v}\left( {x}_{i} \right) + \frac{{h}^{2}}{12}\tilde{v}^{\prime\prime}\left( {x}_{i} \right) + \bigO\left( {h}^{4} \right)\,.
\end{align*}
Through the use of the central difference approximation with a second-order accuracy to approximate the second-order derivative of the function $\tilde{v}\left( x \right)$ at the discrete spatial grid points ${x}_{i}$, an alternative formulation of the prior equation can be obtained with equivalent accuracy
\begin{equation}\label{eq:numres_exact_three-term_scheme}
    \frac{{\Delta}^{2} w\left( {x}_{i-1} \right)}{{h}^{2}} - {p}\left( 1 + \frac{{h}^{2}}{12} {p} \right)w\left( {x}_{i} \right) = \frac{1}{12}\Bigl( \tilde{v}\left( {x}_{i + 1} \right) + \left( 10 + {h}^{2}{p} \right)\tilde{v}\left( {x}_{i} \right) + \tilde{v}\left( {x}_{i - 1} \right) \Bigr) + \bigO\left( {h}^{4} \right)\,.
\end{equation}

For every temporal layer $k$, we define ${u}_{k,i}$ and ${w}_{k,i}$ to be the approximations of the unknown functions ${u}_{k}\left( x \right)$ and ${w}_{k}\left( x \right)$ at the discrete point ${x}_{i}$, respectively. Specifically, we have ${u}_{k}\left( {x}_{i} \right) \approx {u}_{k,i}$ and ${w}_{k}\left( {x}_{i} \right) \approx {w}_{k,i}$. Furthermore, we use ${f}_{k,i}$, ${v}_{k,i}$ and ${\tilde{v}}_{k,i}$ to denote the values of the given functions ${f}_{k}\left( x \right)$, ${v}_{k}\left( x \right)$ and ${\tilde{v}}_{k}\left( x \right)$ at each point ${x}_{i}$, respectively. It is noteworthy that ${v}_{k,i} \approx {\tau}^{2}{f}_{k,i} + {2}{u}_{k,i}$. Finally, the equation \eqref{eq:numres_exact_three-term_scheme} can be rewritten as follows
\begin{equation}\label{eq:tridiag_system_w}
    \begin{aligned}
        {w}_{k,i - 1} + {b}_{k}{w}_{k,i} + {w}_{k,i + 1} &= {\varphi}_{k,i}\,, \\
        {w}_{k,0} = {w}_{k,m} &= 0\,,
    \end{aligned}
\end{equation}
where the values of ${b}_{k}$ and ${\varphi}_{k,i}$ are given by:
\begin{equation*}
    {b}_{k} = -2\left[ 1 + \frac{{\delta}^{2}}{{q}_{k}}\left( 1 + \frac{{\delta}^{2}}{{6}{q}_{k}} \right) \right]\,,\quad {\varphi}_{k,i} = - \frac{{\delta}^{2}}{{6}{q}_{k}}\left( {v}_{k,i - 1} + {2}\left( 5 + \frac{{\delta}^{2}}{{q}_{k}} \right){v}_{k,i} + {v}_{k,i + 1} \right)\,,\quad \delta = \frac{h}{\tau}\,.
\end{equation*}
The problem defined by equations \eqref{eq:discrt_operator_eqt}-\eqref{eq:operator_diff_eq_bound_cond} can be approximated by computing solutions ${u}_{k + 1,i}$ for $k = 1,2,\ldots,{n - 1}$ and $i = 1,2,\ldots,{m - 1}$. These solutions can be obtained by substituting the solutions for ${w}_{k,i}$ and ${u}_{k - 1,i}$ into the equation ${u}_{k + 1,i} = {w}_{k,i} - {u}_{k - 1,i}$. To accomplish the desired objective, it is necessary to effectively solve the tridiagonal systems of linear equations in each discrete time step defined by \eqref{eq:tridiag_system_w}, in which the coefficient matrices are characterized by strict diagonal dominance. Moreover, for each fixed layer, the coefficient matrices denoted as $\mathcal{B}_{k} \in {\mathbb{R}^{\left( m - 1 \right) \times \left( m - 1 \right)}}$ in system \eqref{eq:tridiag_system_w} are symmetric tridiagonal Toeplitz matrices. It is worth noting that there exist straightforward closed-form solutions for their eigenvalues ({\cf} \cite{smith1985}):
\begin{equation*}
    {\lambda}_{k,i} = -2\left[ 1 + \frac{{\delta}^{2}}{{q}_{k}}\left( 1 + \frac{{\delta}^{2}}{{6}{q}_{k}} \right) - \cos\left( \frac{{i}{\pi}}{m} \right) \right]\,,\quad i = 1,2,\ldots,m - 1\,.
\end{equation*}
Our focus is to determine the lower bound of the parameter $\delta$. To achieve this, we proceed by assessing the condition number of the matrix $\mathcal{B}_{k}$, employing the matrix norm induced by the (vector) Euclidean norm, denoted as ${\left\| \cdot \right\|}_{2}$, {\viz}
\begin{align}\label{eq:cond_number_estimate}
    {\kappa}_{2}\left( \mathcal{B}_{k} \right) = {\left\| \mathcal{B}_{k}^{-1} \right\|}_{2} {\left\| \mathcal{B}_{k} \right\|}_{2} &= \frac{\max\limits_{1 \leq i \leq m - 1}{\left| {\lambda}_{k,i} \right|}}{\min\limits_{1 \leq i \leq m - 1}{\left| {\lambda}_{k,i} \right|}} = 1 + \frac{2\cos\left( \dfrac{{\pi}}{m} \right)}{1 + \dfrac{{\delta}^{2}}{{q}_{k}}\left( 1 + \dfrac{{\delta}^{2}}{{6}{q}_{k}} \right) - \cos\left( \dfrac{{\pi}}{m} \right)} \nonumber\\
    &\leq 1 + \frac{2}{\dfrac{{\delta}^{2}}{{q}_{k}}\left( 1 + \dfrac{{\delta}^{2}}{{6}{q}_{k}} \right)} \leq 1 + \varepsilon\,,\quad \varepsilon > 0\,,\quad m \geq 2\,.
\end{align}
It should be noted that an identical estimate can be readily obtained for the condition numbers of the matrix $\mathcal{B}_{k}$ by involving matrix norms induced by the vector $1$-norm and the vector $\infty$-norm, correspondingly.

From the given inequality \eqref{eq:cond_number_estimate}, we can derive the following result
\begin{equation*}
    \delta \geq \sqrt{{q}_{k}\left( \sqrt{9 + \frac{12}{\varepsilon}} - 3 \right)}\,.
\end{equation*}
In conclusion, by further strengthening the above-stated inequality and considering the condition ${q}_{k} > {\c}_{0} > 0$, one can infer that
\begin{equation*}
    \delta = \frac{h}{\tau} \geq \max\limits_{1 \leq {k} \leq {n - 1}}{\sqrt{\frac{{2}{q}_{k}}{\varepsilon}}} > \sqrt{\frac{{2}{\c}_{0}}{\varepsilon}}\,.
\end{equation*}

Since the value of ${q}_{k}$ depends on an unknown function, it becomes necessary to specify $h$ and $\tau$ in accordance with the condition $\delta > \sqrt{\frac{{2}{\c}_{0}}{\varepsilon}}$. This strategic choice of spacing of spatial and temporal grids guarantees in a certain way that the condition numbers of the coefficient matrices $\mathcal{B}_{k}$ for the tridiagonal systems \eqref{eq:tridiag_system_w} fall within a favourable range for efficient computational purposes.

In order to determine the values of the initial conditions \eqref{eq:numres_zerothlayr} and \eqref{eq:semidscrete_scheme_first_layer} at each spatial node ${x}_{i}$, $i = 1,2,\ldots,{m - 1}$, we evaluate the values of the given function ${\psi}_{0}\left( x \right)$ at each point ${x}_{i}$ to obtain the initial condition ${u}_{0,i}$, where ${u}_{0,i} = {\psi}_{0}\left( {x}_{i} \right)$. For the first temporal layer, we compute the values of all given functions in \eqref{eq:semidscrete_scheme_first_layer} at the points ${x}_{i}$, denoting them with lower indices $k$ and $i$ regarding the original functions. To approximate the second-order derivative of the function ${\psi}_{0}\left( x \right)$ with respect to the spatial variable at the points ${x}_{i}$, we use the central finite difference scheme with fourth-order accuracy, where the uniform grid spacing $h$ is halved to ensure that we remain within the spatial interval $\left[ 0,\ell \right]$. Specifically,
\begin{equation}\label{eq:numres_second_derr_4thord}
    \tilde{\psi}_{0,i} = {\psi}^{\prime\prime}_{0}\left( {x}_{i} \right) \approx \frac{-{\psi}_{0}\left( {x}_{i - 1} \right) + {16}{\psi}_{0}\left( {x}_{i - 1 / 2} \right) - {30}{\psi}_{0}\left( {x}_{i} \right) + {16}{\psi}_{0}\left( {x}_{i + 1 / 2} \right) - {\psi}_{0}\left( {x}_{i + 1} \right)}{{3}{h}^{2}}\,.
\end{equation}
Suppose we consider the central finite difference scheme given by \eqref{eq:numres_second_derr_4thord} in the initial condition described by \eqref{eq:semidscrete_scheme_first_layer} for the first temporal layer. In that case, we can represent it discretely at spatial points ${x}_{i}$ as follows:
\begin{equation*}
    {u}_{1,i} = {\psi}_{0,i} + {\tau}{\psi}_{1,i} + \frac{{\tau}^{2}}{2}\left( {f}_{0,i} + {q}_{0}\tilde{\psi}_{0,i} \right)\,, \ \text{with} \ {\psi}_{0,i} = {\psi}_{0}\left( {x}_{i} \right) \ \text{and} \ {\psi}_{1,i} = {\psi}_{1}\left( {x}_{i} \right)\,.
\end{equation*}

To implement the spatial discretization algorithm, the final step is to compute the integral term in the formula ${q}_{k}$, as defined in \eqref{eq:numres_qintegr}. This is accomplished using composite Simpson's rule by providing that the spatial interval $\left[ 0,\ell \right]$ is divided into $m$ subintervals, where $m$ is an even number. The resulting expression can be obtained as follows:
\begin{align*}
    {S}_{k} = \int\limits_{0}^{\ell}{\left( {u}_{k}^{\prime}\left( x \right) \right)}^{2}\d x \approx \frac{h}{3}\left[ {\left( {u}_{k}^{\prime}\left( {x}_{0} \right) \right)}^{2} + {\left( {u}_{k}^{\prime}\left( {x}_{m} \right) \right)}^{2} + {4}\sum\limits_{i = 1}^{m / 2}{{\left( {u}_{k}^{\prime}\left( {x}_{{2}{i} - 1} \right) \right)}^{2}} + {2}\sum\limits_{i = 1}^{m / 2 - 1}{{\left( {u}_{k}^{\prime}\left( {x}_{{2}{i}} \right) \right)}^{2}} \right]\,.
\end{align*}
Note that to evaluate the values of ${u}_{k}^{\prime}\left( {x}_{0} \right)$ and ${u}_{k}^{\prime}\left( {x}_{1} \right)$, we employ the forward finite difference approximation with fourth-order accuracy, which can be expressed as follows:
\begin{align*}
    {u}_{k}^{\prime}\left( {x}_{j} \right) \approx \frac{-25{u}_{k,j} + 48{u}_{k,j + 1} - 36{u}_{k,j + 2} + 16{u}_{k,j + 3} - 3{u}_{k,j + 4}}{12{h}}\,,\quad {j} = 0,1\,.
\end{align*}
In order to compute the values of the first-order derivative of the function ${u}_{k}\left( x \right)$ at the grid points ${x}_{m-1}$ and ${x}_{m}$, we use a technique similar to those employed in previous cases. However, in this scenario, we apply a fourth-order accurate backward finite difference scheme, which is defined as follows:
\begin{align*}
    {u}_{k}^{\prime}\left( {x}_{j} \right) \approx \frac{3{u}_{k,j - 4} - 16{u}_{k,j - 3} + 36{u}_{k,j - 2} - 48{u}_{k,j - 1} + 25{u}_{k,j}}{12{h}}\,,\quad {j} = {m - 1},{m}\,.
\end{align*}
For the nodes ${x}_{j}$ with $j = 2,3,\ldots,{m - 2}$, the values of ${u}_{k}^{\prime}\left( {x}_{j} \right)$ are determined by means of the central finite difference approximation with fourth-order accuracy, defined as:
\begin{align*}
    {u}_{k}^{\prime}\left( {x}_{j} \right) \approx \frac{{u}_{k,j - 2} - 8{u}_{k,j - 1} + 8{u}_{k,j + 1} - {u}_{k,j + 2}}{12{h}}\,.
\end{align*}
Finally, we denote the fourth-order accurate approximations of the values of ${S}_{k}$ and ${q}_{k}$ as $\tilde{S}_{k}$ and $\tilde{q}_{k}$, respectively. Thus, we have $\tilde{q}_{k} = {\alpha}_{k} + {\beta}_{k}\tilde{S}_{k}$.

\subsection{Stability of the Three-Point System Obtained by the Spatial Discretization Algorithm}\label{subsec:intermediate system}
Let us study the stability of the tridiagonal system of linear equations \eqref{eq:tridiag_system_w}. To do so, we rewrite this system in the matrix-vector form using the following notations:
\begin{equation*}
    {\boldsymbol{w}}_{k} = {\left( {w}_{k,1},{w}_{k,2},\ldots,{w}_{k,m - 1} \right)}^{\top}\,,
\end{equation*}
and
\begin{equation*}
    {\boldsymbol{v}}_{k}^{j} = {\left( {v}_{k,j},{v}_{k,j + 1},\ldots,{v}_{k,j + m - 2} \right)}^{\top}\,,\quad j = 0,1,2\,.
\end{equation*}
Consequently, the system \eqref{eq:tridiag_system_w} is reformulated as follows:
\begin{equation}\label{eq:tridiag_system_matrix_form}
    \mathcal{B}_{k} {\boldsymbol{w}}_{k} = {\boldsymbol{\varphi}}_{k}\,,\quad k = 1,2,\ldots,{n - 1}\,,
\end{equation}
where
\begin{equation*}
    {\boldsymbol{\varphi}}_{k} = - \frac{{\delta}^{2}}{{6}{q}_{k}}\left( {\boldsymbol{v}}_{k}^{0} + {2}\left( 5 + \frac{{\delta}^{2}}{{q}_{k}} \right){\boldsymbol{v}}_{k}^{1} + {\boldsymbol{v}}_{k}^{2} \right)\,,\quad {\boldsymbol{v}}_{k}^{j} = {\tau}^{2}{\boldsymbol{f}}_{k}^{j} + {2}{\boldsymbol{u}}_{k}^{j}\,.
\end{equation*}

The solution associated with the given data $\left( {\boldsymbol{f}}_{k},{\boldsymbol{u}}_{k},{q}_{k} \right)$ is denoted by ${\boldsymbol{w}}_{k}$, where the system is referred to as \eqref{eq:tridiag_system_matrix_form}. Herein, considering a perturbation of this data, denoted by $\left( {\boldsymbol{\tilde{f}}}_{k},{\boldsymbol{\tilde{u}}}_{k},{\tilde{q}}_{k} \right)$, leads to the solution denoted by ${\boldsymbol{\tilde{w}}}_{k}$. The objective is to estimate ${\boldsymbol{w}}_{k} - {\boldsymbol{\tilde{w}}}_{k}$ using the following differences: ${\boldsymbol{f}}_{k} - {\boldsymbol{\tilde{f}}}_{k}$, ${\boldsymbol{u}}_{k} - {\boldsymbol{\tilde{u}}}_{k}$, and ${q}_{k} - {\tilde{q}}_{k}$. It is evident that ${\boldsymbol{\tilde{w}}}_{k}$ satisfies the following system
\begin{equation}\label{eq:tridiag_system_pertubed}
    \tilde{\mathcal{B}}_{k} {\boldsymbol{\tilde{w}}}_{k} = {\boldsymbol{\tilde{\varphi}}}_{k}\,,
\end{equation}
where $\tilde{\mathcal{B}}_{k}$ is derived from the matrix $\mathcal{B}_{k}$ by substituting ${q}_{k}$ with ${\tilde{q}}_{k}$. Furthermore, ${\tilde{q}}_{k}$ is obtained by replacing ${\boldsymbol{u}}_{k}$ with ${\boldsymbol{\tilde{u}}}_{k}$ in ${q}_{k}$.

\begin{remark}\label{rmk:stability}
    Let us consider the system of linear equations \eqref{eq:tridiag_system_w} for which the following {\apriori} estimate is valid:
    \begin{equation}\label{eq:stability_estimate_remark}
        \begin{aligned}
            {\left\| {\boldsymbol{w}}_{k} - {\boldsymbol{\tilde{w}}}_{k} \right\|}_{2} &\leq 2 {\left\| {\boldsymbol{u}}_{k}^{1} - {\boldsymbol{\tilde{u}}}_{k}^{1} \right\|}_{2} + {\tau}^{2} \max\limits_{0 \leq j \leq 2}{{\left\| {\boldsymbol{f}}_{k}^{j} - {\boldsymbol{\tilde{f}}}_{k}^{j} \right\|}_{2}} \\
            &+ \frac{2 \left| {q}_{k} - {\tilde{q}}_{k} \right|}{{q}_{k}}\left( 1 + \frac{{\tilde{q}}_{k}}{{q}_{k}} \right)\left( 2 {\left\| {\boldsymbol{u}}_{k}^{1} \right\|}_{2} + {\tau}^{2} \max\limits_{0 \leq j \leq 2}{{\left\| {\boldsymbol{f}}_{k}^{j} \right\|}_{2}} \right)\,,
        \end{aligned}
    \end{equation}
    where ${\boldsymbol{w}}_{k}$ denotes the solution associated with the data $\left( {\boldsymbol{f}}_{k},{\boldsymbol{u}}_{k},{q}_{k} \right)$, referring to the system of linear equations \eqref{eq:tridiag_system_matrix_form}. Moreover, ${\boldsymbol{\tilde{w}}}_{k}$ represents the solution obtained from the perturbed data $\left( {\boldsymbol{\tilde{f}}}_{k},{\boldsymbol{\tilde{u}}}_{k},{\tilde{q}}_{k} \right)$, which corresponds to the perturbed system \eqref{eq:tridiag_system_pertubed}.
\end{remark}
\begin{proof}
    It is evident that the subsequent identity can be derived from equations \eqref{eq:tridiag_system_matrix_form} and \eqref{eq:tridiag_system_pertubed}
    \begin{align*}
        {\boldsymbol{w}}_{k} - {\boldsymbol{\tilde{w}}}_{k} = \mathcal{B}_{k}^{-1} {\boldsymbol{\varphi}}_{k} - \tilde{\mathcal{B}}_{k}^{-1} {\boldsymbol{\tilde{\varphi}}}_{k} &= \left( \mathcal{B}_{k}^{-1} - \tilde{\mathcal{B}}_{k}^{-1} \right) {\boldsymbol{\varphi}}_{k} + \tilde{\mathcal{B}}_{k}^{-1} \left( {\boldsymbol{\varphi}}_{k} - {\boldsymbol{\tilde{\varphi}}}_{k} \right) \\
        &= \tilde{\mathcal{B}}_{k}^{-1} \left( \tilde{\mathcal{B}}_{k} - \mathcal{B}_{k} \right) \mathcal{B}_{k}^{-1} {\boldsymbol{\varphi}}_{k} + \tilde{\mathcal{B}}_{k}^{-1} \left( {\boldsymbol{\varphi}}_{k} - {\boldsymbol{\tilde{\varphi}}}_{k} \right)\,.
    \end{align*}
    Hence, we obtain that
    \begin{equation}\label{eq:stability_inequality_difference_pert}
        {\left\| {\boldsymbol{w}}_{k} - {\boldsymbol{\tilde{w}}}_{k} \right\|}_{2} \leq {\left\|  \tilde{\mathcal{B}}_{k}^{-1} \left( \tilde{\mathcal{B}}_{k} - \mathcal{B}_{k} \right) \mathcal{B}_{k}^{-1} {\boldsymbol{\varphi}}_{k} \right\|}_{2} + {\left\| \tilde{\mathcal{B}}_{k}^{-1} \left( {\boldsymbol{\varphi}}_{k} - {\boldsymbol{\tilde{\varphi}}}_{k} \right) \right\|}_{2}\,.
    \end{equation}
    
    The vector ${\boldsymbol{\varphi}}_{k}$ fulfills the estimate
    \begin{equation}\label{eq:stability_varphi_norm}
        {\left\| {\boldsymbol{\varphi}}_{k} \right\|}_{2} \leq \frac{{\delta}^{2}}{{6}{q}_{k}}\left( {\left\| {\boldsymbol{v}}_{k}^{0} \right\|}_{2} + {2}\left( 5 + \frac{{\delta}^{2}}{{q}_{k}} \right){\left\| {\boldsymbol{v}}_{k}^{1} \right\|}_{2} + {\left\| {\boldsymbol{v}}_{k}^{2} \right\|}_{2} \right)\,.
    \end{equation}
    The following inequality holds for the vectors ${\boldsymbol{v}}_{k}^{j}$, where $j = 0,1,2$
    \begin{equation}\label{eq:stability_inqt_vj}
        {\left\| {\boldsymbol{v}}_{k}^{j} \right\|}_{2} \leq 2{\left\| {\boldsymbol{u}}_{k}^{1} \right\|}_{2} + {\tau}^{2} \max\limits_{0 \leq j \leq 2}{{\left\| {\boldsymbol{f}}_{k}^{j} \right\|}_{2}}\,.
    \end{equation}
    From the estimation \eqref{eq:stability_varphi_norm}, taking into account \eqref{eq:stability_inqt_vj} it follows that
    \begin{equation}\label{eq:stability_varphi_with_u}
        {\left\| {\boldsymbol{\varphi}}_{k} \right\|}_{2} \leq \frac{2{\delta}^{2}}{{q}_{k}} \left( 1 + \frac{{\delta}^{2}}{6{q}_{k}} \right)\left( 2{\left\| {\boldsymbol{u}}_{k}^{1} \right\|}_{2} + {\tau}^{2} \max\limits_{0 \leq j \leq 2}{{\left\| {\boldsymbol{f}}_{k}^{j} \right\|}_{2}} \right)\,.
    \end{equation}
    
    For the inverse matrix of $\mathcal{B}_{k}$, the following estimation is valid,
    \begin{equation}\label{eq:stability_inv_matrix_norm_inqt}
        {\left\| \mathcal{B}_{k}^{-1} \right\|}_{2} = \frac{1}{\min\limits_{1 \leq i \leq m - 1}{\left| {\lambda}_{k,i} \right|}} \leq {\left( \frac{2{\delta}^{2}}{{q}_{k}} \left( 1 + \frac{{\delta}^{2}}{6{q}_{k}} \right) \right)}^{-1}\,.
    \end{equation}
    Analogously to the expression \eqref{eq:stability_inv_matrix_norm_inqt}, we derive the following inequality
    \begin{equation}\label{eq:stability_inv_pert_matrix_norm_inqt}
        {\left\| \tilde{\mathcal{B}}_{k}^{-1} \right\|}_{2} \leq {\left( \frac{2{\delta}^{2}}{{\tilde{q}}_{k}} \left( 1 + \frac{{\delta}^{2}}{6{\tilde{q}}_{k}} \right) \right)}^{-1}\,.
    \end{equation}
    Let us now move forward with the evaluation of the matrix difference, denoted as $\tilde{\mathcal{B}}_{k} - \mathcal{B}_{k}$. Evidently, at each temporal layer, the matrix exhibits the characteristics of a diagonal-constant matrix, featuring nonzero elements only along the principal diagonal, with all other entries being zero. These components precisely correspond to the difference $\tilde{b}_{k} - {b}_{k}$. Through a straightforward calculation, we arrive at the result
    \begin{equation}\label{eq:stability_diagonal_matrix_elements}
        \tilde{b}_{k} - {b}_{k} = \frac{2{\delta}^{2}}{{\tilde{q}}_{k}{q}_{k}}\left( {\tilde{q}}_{k} - {q}_{k} \right)\left( 1 + \frac{{\delta}^{2}}{{6}{\tilde{q}}_{k}} + \frac{{\delta}^{2}}{{6}{q}_{k}} \right)\,.
    \end{equation}
    Upon considering the inequalities \eqref{eq:stability_varphi_with_u}, \eqref{eq:stability_inv_matrix_norm_inqt}, \eqref{eq:stability_inv_pert_matrix_norm_inqt} together with \eqref{eq:stability_diagonal_matrix_elements} for the first term of the right-hand side of inequality \eqref{eq:stability_inequality_difference_pert}, we derive the following conclusion
    \begin{equation}\label{eq:stability_first_term_inequality_diff}
        {\left\|  \tilde{\mathcal{B}}_{k}^{-1} \left( \tilde{\mathcal{B}}_{k} - \mathcal{B}_{k} \right) \mathcal{B}_{k}^{-1} {\boldsymbol{\varphi}}_{k} \right\|}_{2} \leq \frac{\left| {\tilde{q}}_{k} - {q}_{k} \right|}{{q}_{k}}\left( 1 + \frac{{\tilde{q}}_{k}}{{q}_{k}} \right)\left( 2{\left\| {\boldsymbol{u}}_{k}^{1} \right\|}_{2} + {\tau}^{2} \max\limits_{0 \leq j \leq 2}{{\left\| {\boldsymbol{f}}_{k}^{j} \right\|}_{2}} \right)\,.
    \end{equation}

    Employing a straightforward transformation, we immediately establish the following identity
    \begin{equation*}
        {\boldsymbol{\varphi}}_{k} - {\boldsymbol{\tilde{\varphi}}}_{k} = {\boldsymbol{\psi}}_{k,1} + {\boldsymbol{\psi}}_{k,2}\,,
    \end{equation*}
    where:
    \begin{align}
        {\boldsymbol{\psi}}_{k,1} &= \frac{{\delta}^{2}}{6{\tilde{q}}_{k}{q}_{k}} \left(  {q}_{k} - {\tilde{q}}_{k} \right) \left[ {\boldsymbol{v}}_{k}^{0} + {2}\left( 5 + \frac{{\delta}^{2}}{{\tilde{q}}_{k}} + \frac{{\delta}^{2}}{{q}_{k}} \right){\boldsymbol{v}}_{k}^{1} + {\boldsymbol{v}}_{k}^{2} \right]\,,\label{eq:stability_psi1} \\
        {\boldsymbol{\psi}}_{k,2} &= \frac{{\delta}^{2}}{{6}{\tilde{q}}_{k}}\left[ \left( {\boldsymbol{\tilde{v}}}_{k}^{0} - {\boldsymbol{v}}_{k}^{0} \right) + {2}\left( 5 + \frac{{\delta}^{2}}{{\tilde{q}}_{k}} \right)\left( {\boldsymbol{\tilde{v}}}_{k}^{1} - {\boldsymbol{v}}_{k}^{1} \right) + \left( {\boldsymbol{\tilde{v}}}_{k}^{2} - {\boldsymbol{v}}_{k}^{2} \right) \right]\,.\label{eq:stability_psi2}
    \end{align}
    By applying inequality \eqref{eq:stability_inqt_vj} in a way similar to \eqref{eq:stability_varphi_with_u}, we can derive the estimates for \eqref{eq:stability_psi1} and \eqref{eq:stability_psi2}, {\viz}
    \begin{align}
        {\left\| {\boldsymbol{\psi}}_{k,1} \right\|}_{2} &\leq \frac{2{\delta}^{2}}{{\tilde{q}}_{k}{q}_{k}} \left|  {q}_{k} - {\tilde{q}}_{k} \right| \left( 1 + \frac{{\delta}^{2}}{{6}{\tilde{q}}_{k}} + \frac{{\delta}^{2}}{{6}{q}_{k}} \right) \left( 2{\left\| {\boldsymbol{u}}_{k}^{1} \right\|}_{2} + {\tau}^{2} \max\limits_{0 \leq j \leq 2}{{\left\| {\boldsymbol{f}}_{k}^{j} \right\|}_{2}} \right)\,,\label{eq:stability_norm_psi1} \\
        {\left\| {\boldsymbol{\psi}}_{k,2} \right\|}_{2} &\leq \frac{2{\delta}^{2}}{{\tilde{q}}_{k}} \left( 1 + \frac{{\delta}^{2}}{6{\tilde{q}}_{k}} \right)\left( 2{\left\| {\boldsymbol{\tilde{u}}}_{k}^{1} - {\boldsymbol{u}}_{k}^{1} \right\|}_{2} + {\tau}^{2} \max\limits_{0 \leq j \leq 2}{{\left\| {\boldsymbol{\tilde{f}}}_{k}^{j} - {\boldsymbol{f}}_{k}^{j} \right\|}_{2}} \right)\,.\label{eq:stability_norm_psi2}
    \end{align}
    In light of inequalities \eqref{eq:stability_norm_psi1} and \eqref{eq:stability_norm_psi2} along with \eqref{eq:stability_inv_pert_matrix_norm_inqt} applied to the second term of the right-hand side of inequality \eqref{eq:stability_inequality_difference_pert}, we deduce the subsequent estimation
    \begin{align}\label{eq:stability_second_term_inequality_diff}
        {\left\| \tilde{\mathcal{B}}_{k}^{-1} \left( {\boldsymbol{\varphi}}_{k} - {\boldsymbol{\tilde{\varphi}}}_{k} \right) \right\|}_{2} &\leq \frac{\left|  {q}_{k} - {\tilde{q}}_{k} \right|}{{q}_{k}} \left( 1 + \frac{{\tilde{q}}_{k}}{{q}_{k}} \right) \left( 2{\left\| {\boldsymbol{u}}_{k}^{1} \right\|}_{2} + {\tau}^{2} \max\limits_{0 \leq j \leq 2}{{\left\| {\boldsymbol{f}}_{k}^{j} \right\|}_{2}} \right)\nonumber \\
        &+ 2{\left\| {\boldsymbol{\tilde{u}}}_{k}^{1} - {\boldsymbol{u}}_{k}^{1} \right\|}_{2} + {\tau}^{2} \max\limits_{0 \leq j \leq 2}{{\left\| {\boldsymbol{\tilde{f}}}_{k}^{j} - {\boldsymbol{f}}_{k}^{j} \right\|}_{2}}\,.
    \end{align}
    When we substitute the inequalities \eqref{eq:stability_first_term_inequality_diff} and \eqref{eq:stability_second_term_inequality_diff} into \eqref{eq:stability_inequality_difference_pert}, the result \eqref{eq:stability_estimate_remark} is obtained, thereby concluding the proof. Definitely, the stability of the system \eqref{eq:tridiag_system_w} follows from the inequality \eqref{eq:stability_estimate_remark}.
\end{proof}

\subsection{Numerical Illustrations}\label{subsec:numillust}
We provide a complete numerical implementation of the four benchmark examples (referred to as \hyperref[problem:test1]{\bf Test \ref*{problem:test1}}, \hyperref[problem:test2]{\bf Test \ref*{problem:test2}}, \hyperref[problem:test3]{\bf Test \ref*{problem:test3}}, and \hyperref[problem:test4]{\bf Test \ref*{problem:test4}}), including analysis of the obtained results, for our proposed algorithm that deals with discretizing time and spatial domains. All computations were performed using GNU Octave 8.4.0. You can access the source code, supplementary materials, and data obtained from these cases on \href{https://github.com/zv1991/Semi-Discrete-Scheme-for-Kirchhoff-Eqn}{GitHub}\footnote{GitHub repository: \url{https://github.com/zv1991/Semi-Discrete-Scheme-for-Kirchhoff-Eqn}} and \href{https://doi.org/10.5281/zenodo.10137813}{Zenodo} \cite{rogava_2023_10137922}.

The approximation error, denoted as ${E}_{m}\left( {t}_{k} \right)$, is defined by the following expression:
\begin{equation}\label{eq:approx_error}
    {E}_{m}\left( {t}_{k} \right) = \max\limits_{0 \leq i \leq m}\left| u\left( {x}_{i},{t}_{k} \right) - {u}_{k,i} \right|\,.
\end{equation}
Here, $u\left( {x}_{i},{t}_{k} \right)$ represents the values of the exact solution $u\left( x,t \right)$ for problem \eqref{eq:main_eqt}-\eqref{eq:boundary_conds} at the points ${x}_{i}$ and ${t}_{k}$. Furthermore, ${u}_{k,i}$ corresponds to the computed approximate solutions of problem \eqref{eq:main_eqt}-\eqref{eq:boundary_conds}. It should be noted that the order of accuracy for the semi-discrete scheme \eqref{eq:semidiscrete_scheme}, when combined with the spatial discretization algorithm, is $\bigO\left( {\tau}^{2} + {h}^{4} \right)$.

Consider the following pair of test problems ({\cf} \cite{RogavaTsiklauri_LocConvg2012,vashakidze2020,vashakidze2022}) with the given parameters: length $\ell = 1$, time interval $T = 1$, coefficient functions $\displaystyle \alpha\left( t \right) = 2 + \sin\left( \frac{10 \pi}{T} t \right)$, and $\beta\left( t \right) = 1 + {t}^{2}$.
\begin{test}\label{problem:test1}
    \begin{align*}
        {\psi}_{0}\left( x \right) &= 0\,,\quad {\psi}_{1}\left( x \right) = \frac{\lambda \pi}{T}\sin\left( \frac{\pi}{\ell} x \right)\,,\\
        f\left( x,t \right) &= {\pi}^{2}\left[ -{\left( \frac{\lambda}{T} \right)}^{2} + \frac{1}{{\ell}^{2}}\left( \alpha\left( t \right) + \frac{{\pi}^{2}}{2\ell}\beta\left( t \right){\sin}^{2}\left( \frac{\lambda \pi}{T} t \right) \right) \right]\sin\bigg( \frac{\pi}{\ell} x \bigg)\sin\left( \frac{\lambda \pi}{T} t \right)\,.
    \end{align*}
\end{test}
\begin{test}\label{problem:test2}
    \begin{align*}
        {\psi}_{0}\left( x \right) &= \sin\left( \frac{\lambda \pi}{\ell} x \right)\,,\quad {\psi}_{1}\left( x \right) = {\pi}\sin\left( \frac{\lambda \pi}{\ell} x \right)\,,\\
        f\left( x,t \right) &= {\pi}^{2}\left[ 1 + \frac{{\lambda}^{2}}{{\ell}^{2}}\left( \alpha\left( t \right) + \frac{{\lambda}^{2} {\pi}^{2}}{2 \ell} {\e}^{2 \pi t}\beta\left( t \right) \right) \right]\sin\left( \frac{\lambda \pi}{\ell} x \right){\e}^{\pi t}\,.
    \end{align*}
\end{test}
\noindent Let $\displaystyle u\left( x,t \right) = \sin\bigg( \frac{\pi}{\ell} x \bigg)\sin\left( \frac{\lambda \pi}{T} t \right)$ and $\displaystyle u\left( x,t \right) = \sin\left( \frac{\lambda \pi}{\ell} x \right){\e}^{\pi t}$ represent the solutions to the test problems in \hyperref[problem:test1]{\bf Test \ref*{problem:test1}} and \hyperref[problem:test2]{\bf Test \ref*{problem:test2}}, respectively.

\hyperref[fig:logarithm-error-test1]{\bf Figure \ref*{fig:logarithm-error-test1}} presents log-log scale graphs demonstrating the dependence of maximum absolute errors between the exact and approximate solutions of \hyperref[problem:test1]{\bf Test \ref*{problem:test1}} under four different scenarios: $\lambda = 1$, $3$, $7$, $11$, with respect to temporal steps. The logarithmic scales are used on both the abscissa and ordinate axes. The horizontal axis corresponds to the decimal logarithms of the temporal steps denoted as ${\tau}_{j} = {10}^{-1}{\left( 100 - j \right)}^{-1}$, while the vertical axis represents the maximum absolute errors $g\left( {\tau}_{j} \right) = \max_{1 \leq k \leq {n}_{j}} {E}_{1000}\left( {k}{\tau}_{j} \right)$ for ${j}=0,1,\ldots,99$. It is assumed that when $h$ is sufficiently small, the behavior of $g\left( {\tau}_{j} \right)$ can be expressed by $g\left( {\tau}_{j} \right) \approx \mathrm{c}{\tau}_{j}^{2}$ ($\mathrm{c} = \mathrm{const} > 0$). Consequently, ${\log}_{10}\left( g\left( {\tau}_{j} \right) \right) \approx {2}\,{\log}_{10}\left( {\tau}_{j} \right) + {\log}_{10}\left( \mathrm{c} \right)$. In other words, on a log-log scale graph, the error exhibits a linear relationship with a slope of $2$, which signifies a second-order accuracy regarding the temporal variable. To perform curve fitting on the given data points in all four two-dimensional graphs, Simple Linear Regression (SLR) is employed. The resulting curves correspond to straight lines with slopes approximately equal to $2$, providing evidence that the proposed scheme \eqref{eq:semidiscrete_scheme} achieves second-order accuracy with respect to the temporal variable. In each of the four cases, the division number $m$ of the spatial domain remains fixed at $m = 1000$.

\begin{figure}[H]
    \centering
    \begin{subfigure}{.49\textwidth}
        \centering
        \includegraphics[width=1\linewidth]{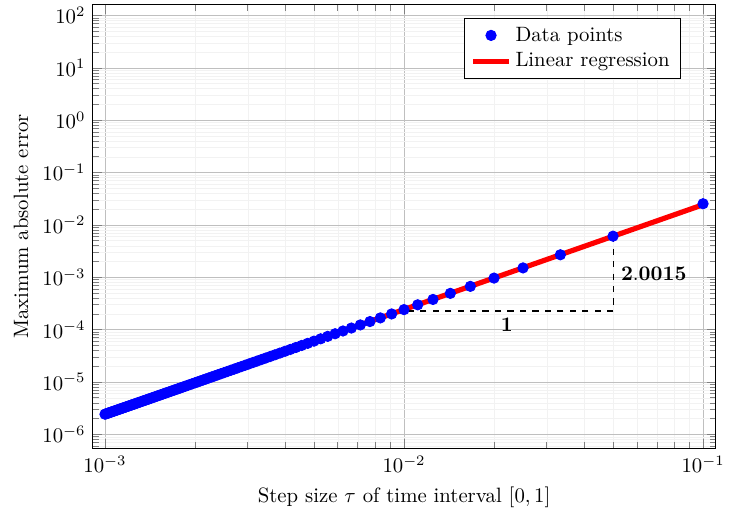}
        \caption{$\lambda = 1$.}
        \label{fig:log10osc1m1000}
    \end{subfigure}
    \hfill
    \begin{subfigure}{.49\textwidth}
        \centering
        \includegraphics[width=1\linewidth]{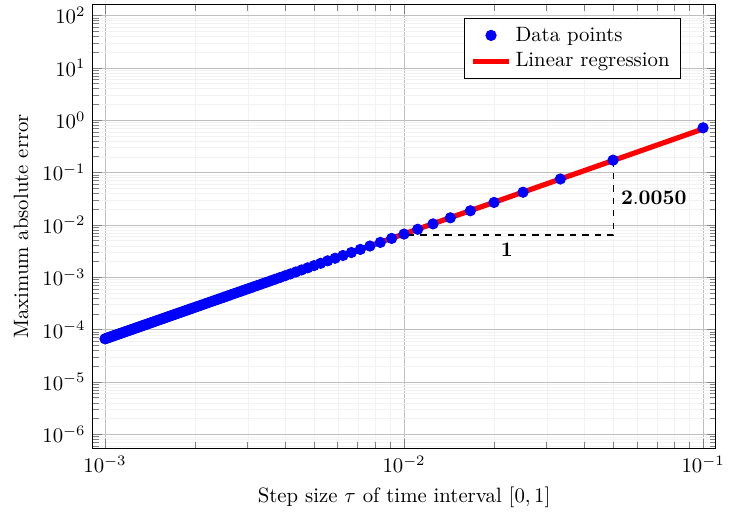}
        \caption{$\lambda = 3$.}
        \label{fig:log10osc3m1000}
    \end{subfigure}
    \hfill
    \begin{subfigure}{.49\textwidth}
        \centering
        \includegraphics[width=1\linewidth]{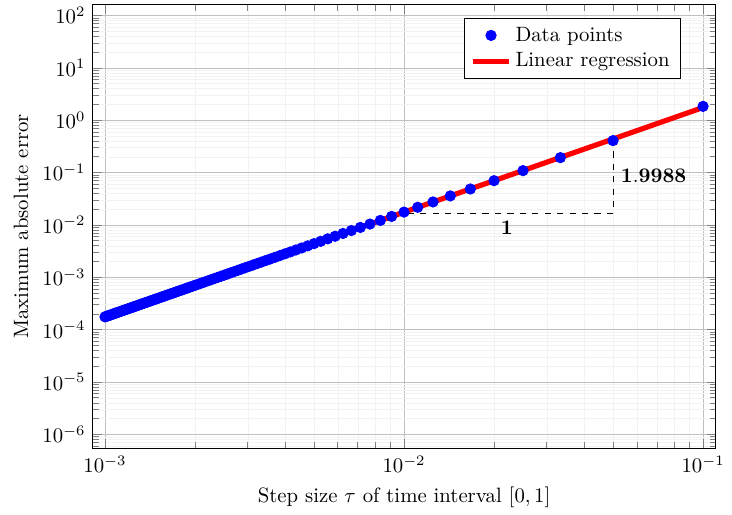}
        \caption{$\lambda = 7$.}
        \label{fig:log10osc7m1000}
    \end{subfigure}
    \hfill
    \begin{subfigure}{.49\textwidth}
        \centering
        \includegraphics[width=1\linewidth]{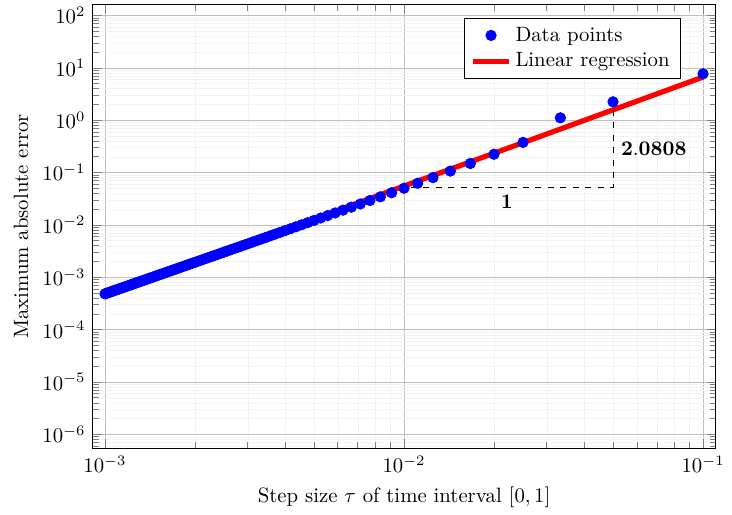}
        \caption{$\lambda = 11$.}
        \label{fig:log10osc11m1000}
    \end{subfigure}
    \caption{A scatter plot on a log-log scale graph representing the dependence of the maximum absolute errors between the exact and approximate solutions of \hyperref[problem:test1]{\bf Test \ref*{problem:test1}} on the temporal steps, along with a regression line.}
    \label{fig:logarithm-error-test1}
\end{figure}

Consider the set of sample points denoted by the data pairs $\left( X,Y \right) = \left\{ \left( {X}_{j},{Y}_{j} \right),\;j = 0,1,\ldots,99 \right\}$, where ${X}_{j}$ and ${Y}_{j}$ represent the base-$10$ logarithm of ${\tau}_{j}$ and the base-$10$ logarithm of $g\left( {\tau}_{j} \right)$, respectively.

Let us define the function
\begin{equation*}
    {Y}_{j} = \widehat{Y}_{j} + \widehat{\varepsilon}_{j}\,,\quad j = 0,1,\ldots,99\,,
\end{equation*}
where $\widehat{Y}_{j}$ (fitted values) are obtained by substituting ${X}_{j}$ into the equation of regression line: $\widehat{Y}_{j} = \widehat{b}_{1} {X}_{j} + \widehat{b}_{0}$. The residuals $\widehat{\varepsilon}_{j} = {Y}_{j} - \widehat{Y}_{j}$ represent the differences between the given data points ${Y}_{j}$ and the fitted values $\widehat{Y}_{j}$. In statistics, it is well-known that the sum of squared residuals (${\mathrm{SS}}_{\mathrm{res}}$) can be expressed as follows:
\begin{equation*}
    {\mathrm{SS}}_{\mathrm{res}} = \sum\limits_{j}{{\left( \widehat{\varepsilon}_{j} \right)}^{2}} = {\mathrm{S}}_{YY}\left( 1 - \frac{{\mathrm{S}}_{XY}^{2}}{{\mathrm{S}}_{XX}{\mathrm{S}}_{YY}} \right) = {\mathrm{S}}_{YY}\left( 1 - {r}^{2} \right)\,,
\end{equation*}
where
\begin{equation*}
    {\mathrm{S}}_{XX} = \sum\limits_{j}{{\left( \bar{X} - {X}_{j} \right)}^{2}}\,,\quad{\mathrm{S}}_{XY} = \sum\limits_{j}{\left( \bar{X} - {X}_{j} \right)\left( \bar{Y} - {Y}_{j} \right)}\,,\quad{\mathrm{S}}_{YY} = \sum\limits_{j}{{\left( \bar{Y} - {Y}_{j} \right)}^{2}}\,.
\end{equation*}
Here, $\bar{X}$ and $\bar{Y}$  represent the arithmetic means of the values of $X$ and $Y$, respectively.

In \hyperref[tab:linear-regression-table]{\bf Table \ref*{tab:linear-regression-table}}, for all four cases under consideration, the coefficient of determination, denoted as ${r}^{2}$, consistently matches a value of $1$ with a high level of precision. This indicates that a linear regression equation perfectly describes the relationship between ${X}_{j}$ and ${Y}_{j}$, with all data points exhibiting close proximity to the corresponding regression line $\widehat{Y}_{j}$.
\begin{table}[H]
    \centering
    % Required packages
% \usepackage{pgfplots}
% \usepackage{pgf}
% \usepackage{pgfplotstable}
% \pgfplotsset{compat = newest}

%%% Generated with pgfplotstable %%%

%%%%%%%%%%%%%%%%%%%%%%%%%%%%%%%
\begin{tabular}{ccccc}
    \toprule
    $\lambda$ & ${ \widehat{Y}_{j} }$ & $\min \limits _{j}{\left| \widehat{\varepsilon}_{j} \right|}$ & $\max \limits _{j}{\left| \widehat{\varepsilon}_{j} \right|}$ & ${r}^{2}$ \\
    \midrule
    $\pgfmathprintnumber[std,fixed zerofill,sci zerofill,precision=0,1000 sep={}]{1}$ & $\pgfmathprintnumber[std,std=-1:0,fixed zerofill,sci zerofill,precision=4,1000 sep={}]{2.001504221665009}\cdot {X}_{j} \pgfmathprintnumber[std,std=-1:0,fixed zerofill,sci zerofill,print sign,precision=4,1000 sep={}]{0.385903054232940}$ & $\pgfmathprintnumber[std,std=-1:0,fixed zerofill,sci zerofill,sci e,precision=4,1000 sep={}]{0.000020779907439}$ & $\pgfmathprintnumber[std,std=-1:0,fixed zerofill,sci zerofill,sci e,precision=4,1000 sep={}]{0.018305898492725}$ & $\pgfmathprintnumber[std,std=1,fixed zerofill,sci zerofill,sci e,precision=4,1000 sep={}]{0.999992789521704}$ \\
    $\pgfmathprintnumber[std,fixed zerofill,sci zerofill,precision=0,1000 sep={}]{3}$ & $\pgfmathprintnumber[std,std=-1:0,fixed zerofill,sci zerofill,precision=4,1000 sep={}]{2.005029621056219}\cdot {X}_{j} \pgfmathprintnumber[std,std=-1:0,fixed zerofill,sci zerofill,print sign,precision=4,1000 sep={}]{1.836734865245398}$ & $\pgfmathprintnumber[std,std=-1:0,fixed zerofill,sci zerofill,sci e,precision=4,1000 sep={}]{0.000005775399193}$ & $\pgfmathprintnumber[std,std=-1:0,fixed zerofill,sci zerofill,sci e,precision=4,1000 sep={}]{0.020736820463751}$ & $\pgfmathprintnumber[std,std=1,fixed zerofill,sci zerofill,sci e,precision=4,1000 sep={}]{0.999989320082047}$ \\
    $\pgfmathprintnumber[std,fixed zerofill,sci zerofill,precision=0,1000 sep={}]{7}$ & $\pgfmathprintnumber[std,std=-1:0,fixed zerofill,sci zerofill,precision=4,1000 sep={}]{1.998830092297937}\cdot {X}_{j} \pgfmathprintnumber[std,std=-1:0,fixed zerofill,sci zerofill,print sign,precision=4,1000 sep={}]{2.239433824041239}$ & $\pgfmathprintnumber[std,std=-1:0,fixed zerofill,sci zerofill,sci e,precision=4,1000 sep={}]{0.000009719898283}$ & $\pgfmathprintnumber[std,std=-1:0,fixed zerofill,sci zerofill,sci e,precision=4,1000 sep={}]{0.030555308208797}$ & $\pgfmathprintnumber[std,std=1,fixed zerofill,sci zerofill,sci e,precision=4,1000 sep={}]{0.999977423328479}$ \\
    $\pgfmathprintnumber[std,fixed zerofill,sci zerofill,precision=0,1000 sep={}]{11}$ & $\pgfmathprintnumber[std,std=-1:0,fixed zerofill,sci zerofill,precision=4,1000 sep={}]{2.080833284629712}\cdot {X}_{j} \pgfmathprintnumber[std,std=-1:0,fixed zerofill,sci zerofill,print sign,precision=4,1000 sep={}]{2.902834186027562}$ & $\pgfmathprintnumber[std,std=-1:0,fixed zerofill,sci zerofill,sci e,precision=4,1000 sep={}]{0.000094191658662}$ & $\pgfmathprintnumber[std,std=0,fixed zerofill,sci zerofill,sci e,precision=4,1000 sep={}]{0.214206899942410}$ & $\pgfmathprintnumber[std,std=1,fixed zerofill,sci zerofill,sci e,precision=4,1000 sep={}]{0.998400554457967}$ \\
    \bottomrule
\end{tabular}
    \caption{The regression line and residuals represent the relationship between the given data points and their curve fitting.}
    \label{tab:linear-regression-table}
\end{table}

In \hyperref[fig:at-the-one-point]{\bf Figure \ref*{fig:at-the-one-point}}, we present the exact and approximate solutions of \hyperref[problem:test2]{\bf Test \ref*{problem:test2}} at the specific time point $t = 1$. The solid and dashed curves correspond to the exact and approximate solutions. In \hyperref[fig:at-the-one-point]{\bf Figure \ref*{fig:at-the-one-point}} and \hyperref[tab:at-the-one-point]{\bf Table \ref*{tab:at-the-one-point}}, {\bf (a)} and {\bf (b)} represent the case where $\lambda = 5$, while {\bf (c)} and {\bf (d)} correspond to $\lambda = 17$. In the original configuration, the spatial grid has a spacing of $h = 1/32$, which is subsequently halved. Additionally, the value of the temporal step $\tau$ is equal to ${h}^{2}$ for each pair of considered cases. The test results demonstrate that solely increasing the number of temporal layers is insufficient to achieve high-order accuracy; it is also necessary to increase the number of spatial domain divisions. This phenomenon arises as a consequence of the significant oscillation number, denoted as $\lambda$, appearing in the exact solution of the considered cases. It is essential to observe that the constant in the error estimate of the proposed algorithm contains a maximum absolute value of high-order partial derivatives of the solution with respect to both temporal and spatial variables. Furthermore, it is also dependent on the maximum absolute values of $q\left( t \right)$, which in turn involve the time-varying coefficients $\alpha\left( t \right)$, $\beta\left( t \right)$, as well as the integral of ${\left[ u_{x}^{\prime}\left( x,t \right) \right]}^{2}$ with respect to $x$. From a numerical implementation perspective, it is also recommended to carefully choose values for $h$ and $\tau$ in such a manner that the condition numbers of the coefficient matrices of the tridiagonal systems obtained per temporal layer remain within a practically acceptable range. This guarantees the existence of a well-defined solution for each time step.

\begin{figure}[H]
    \centering
    \begin{subfigure}{.49\textwidth}
        \centering
        \includegraphics[width=1\linewidth]{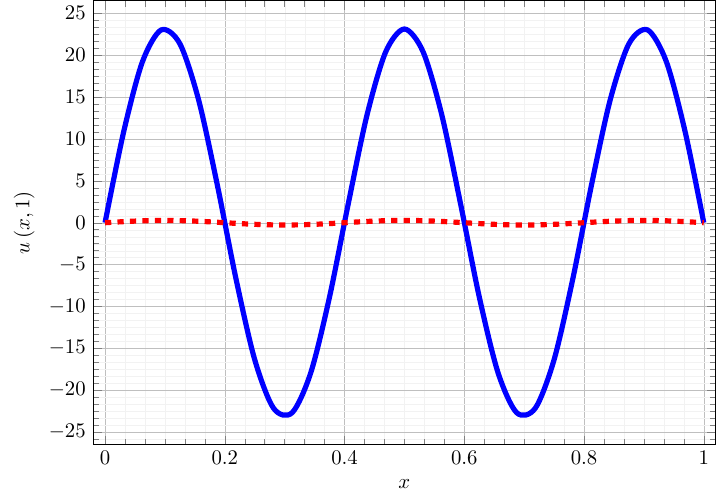}
        \caption{$\lambda = 5$, $m = 32$.}
        \label{fig:osc5n1024m32}
    \end{subfigure}
    \hfill
    \begin{subfigure}{.49\textwidth}
        \centering
        \includegraphics[width=1\linewidth]{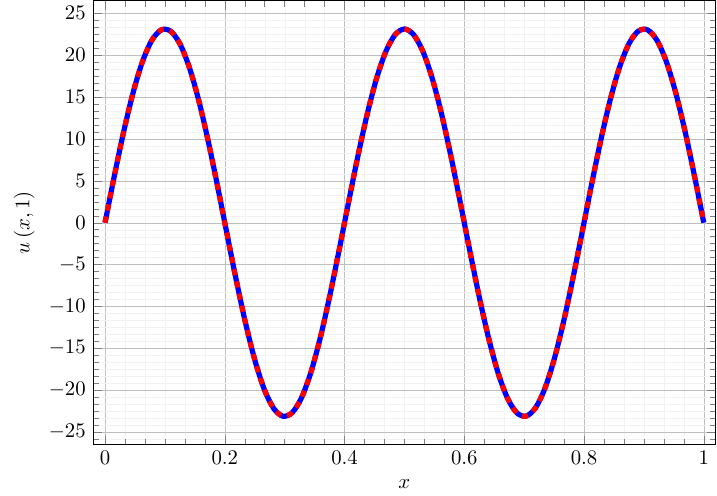}
        \caption{$\lambda = 5$, $m = 64$.}
        \label{fig:osc5n4096m64}
    \end{subfigure}
    \hfill
    \begin{subfigure}{.49\textwidth}
        \centering
        \includegraphics[width=1\linewidth]{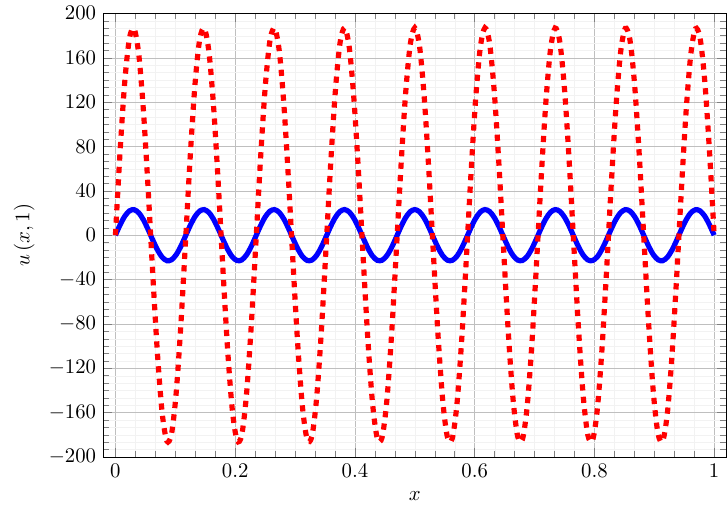}
        \caption{$\lambda = 17$, $m = 128$.}
        \label{fig:osc17n16384m128}
    \end{subfigure}
    \hfill
    \begin{subfigure}{.49\textwidth}
        \centering
        \includegraphics[width=1\linewidth]{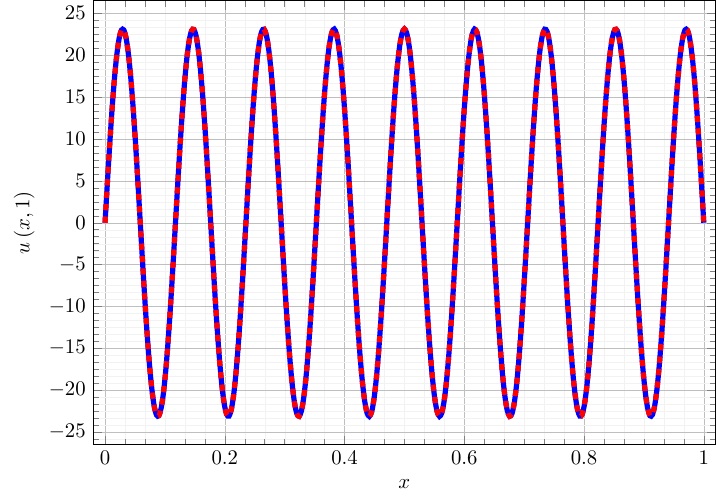}
        \caption{$\lambda = 17$, $m = 256$.}
        \label{fig:osc17n65536m256}
    \end{subfigure}
    \caption{The exact and approximate solutions of \hyperref[problem:test2]{\bf Test \ref*{problem:test2}} at the point $t = 1$ are represented by solid and dashed lines, respectively.}
    \label{fig:at-the-one-point}
\end{figure}

The resulting maximum absolute errors for each distinct case are presented in \hyperref[tab:at-the-one-point]{\bf Table \ref*{tab:at-the-one-point}}. Moreover, we evaluate the maximum values of condition numbers of the coefficient matrices associated with the tridiagonal systems \eqref{eq:tridiag_system_w} at the specified temporal layers. These condition numbers are indicated in \hyperref[tab:at-the-one-point]{\bf Table \ref*{tab:at-the-one-point}}.

It should be noted that in \hyperref[tab:at-the-one-point]{\bf Table \ref*{tab:at-the-one-point}}, our main focus centres around tracking the evolution of errors in relation to the variable $t$. In this instance, we have set the values of $\tau$ and $h$ in such a way that $\tau = {h}^{2}$, resulting in an overall approximation order of $\bigO\left( {h}^{4} \right)$ for the combined scheme. However, this specific choice does not generally ensure the stability of the condition numbers of the coefficient matrices $\mathcal{B}_{k}$ associated with the tridiagonal systems \eqref{eq:tridiag_system_w}, which can be further affected by the oscillatory behaviour of the values of ${q}_{k}$. In addition, we compute the maximum values of condition numbers of the matrices $\mathcal{B}_{j}$ for $j = 1,2,\ldots,{k - 1}$. As we pursue numerical solutions in the $\left( j + 1 \right)$-layer, it is necessary to solve the systems \eqref{eq:tridiag_system_w} for the previous $j$-th layer.

\begin{table}[H]
    \centering
    \begin{subtable}{.49\textwidth}
        \centering
        % Required packages
% \usepackage{pgfplots}
% \usepackage{pgf}
% \usepackage{pgfplotstable}
% \pgfplotsset{compat = newest}

%%% Generated with pgfplotstable %%%

%%%%%%%%%%%%%%%%%%%%%%%%%%%%%%%
\begin{tabular}{cccc}
    \toprule
    $k$ & ${t}_{k}$ & $\max\limits_{1 \leq j \leq k} {E}_{32}\left ( {t}_{j} \right )$ & $\max\limits_{1 \leq j \leq k - 1} {\kappa}_{2}\left( \mathcal{B}_{j} \right)$ \\
    \midrule
    $\pgfmathprintnumber[std,fixed zerofill,sci zerofill,precision=0,1000 sep={}]{2.560000000000000e+02}$ & $\pgfmathprintnumber[std,std=-1:0,fixed zerofill,sci zerofill,precision=2,1000 sep={}]{2.500000000000000e-01}$ & $\pgfmathprintnumber[std,std=-1:0,fixed zerofill,sci zerofill,sci e,precision=4,1000 sep={}]{1.097990120188075e-02}$ & $\pgfmathprintnumber[std,std=1,fixed zerofill,sci zerofill,sci e,precision=4,1000 sep={}]{1.960787905113772e+00}$ \\
    $\pgfmathprintnumber[std,fixed zerofill,sci zerofill,precision=0,1000 sep={}]{5.120000000000000e+02}$ & $\pgfmathprintnumber[std,std=-1:0,fixed zerofill,sci zerofill,precision=2,1000 sep={}]{5.000000000000000e-01}$ & $\pgfmathprintnumber[std,std=-1:0,fixed zerofill,sci zerofill,sci e,precision=4,1000 sep={}]{1.201590345761350e-02}$ & $\pgfmathprintnumber[std,std=1,fixed zerofill,sci zerofill,sci e,precision=4,1000 sep={}]{7.462647814624078e+00}$ \\
    $\pgfmathprintnumber[std,fixed zerofill,sci zerofill,precision=0,1000 sep={}]{7.680000000000000e+02}$ & $\pgfmathprintnumber[std,std=-1:0,fixed zerofill,sci zerofill,precision=2,1000 sep={}]{7.500000000000000e-01}$ & $\pgfmathprintnumber[std,std=-1:0,fixed zerofill,sci zerofill,sci e,precision=4,1000 sep={}]{1.613082539672880e+01}$ & $\pgfmathprintnumber[std,std=-1:0,fixed zerofill,sci zerofill,sci e,precision=4,1000 sep={}]{1.608954154553009e+02}$ \\
    $\pgfmathprintnumber[std,fixed zerofill,sci zerofill,precision=0,1000 sep={}]{1.024000000000000e+03}$ & $\pgfmathprintnumber[std,std=-1:0,fixed zerofill,sci zerofill,precision=2,1000 sep={}]{1.000000000000000e+00}$ & $\pgfmathprintnumber[std,std=-1:0,fixed zerofill,sci zerofill,sci e,precision=4,1000 sep={}]{3.547819516295949e+02}$ & $\pgfmathprintnumber[std,std=-1:0,fixed zerofill,sci zerofill,sci e,precision=4,1000 sep={}]{4.118509895840504e+02}$ \\
    \bottomrule
\end{tabular}
        \caption{$\lambda = 5$, $n = 1024$.}
        \label{tab:osc5n1024m32}
    \end{subtable}%
    \hfill
    \begin{subtable}{.49\textwidth}
        \centering
        % Required packages
% \usepackage{pgfplots}
% \usepackage{pgf}
% \usepackage{pgfplotstable}
% \pgfplotsset{compat = newest}

%%% Generated with pgfplotstable %%%

%%%%%%%%%%%%%%%%%%%%%%%%%%%%%%%
\begin{tabular}{cccc}
    \toprule
    $k$ & ${t}_{k}$ & $\max\limits_{1 \leq j \leq k} {E}_{64}\left ( {t}_{j} \right )$ & $\max\limits_{1 \leq j \leq k - 1} {\kappa}_{2}\left( \mathcal{B}_{j} \right)$ \\
    \midrule
    $\pgfmathprintnumber[std,fixed zerofill,sci zerofill,precision=0,1000 sep={}]{1.024000000000000e+03}$ & $\pgfmathprintnumber[std,std=-1:0,fixed zerofill,sci zerofill,precision=2,1000 sep={}]{2.500000000000000e-01}$ & $\pgfmathprintnumber[std,std=-1:0,fixed zerofill,sci zerofill,sci e,precision=4,1000 sep={}]{3.227321422570206e-03}$ & $\pgfmathprintnumber[std,std=1,fixed zerofill,sci zerofill,sci e,precision=4,1000 sep={}]{1.148918551166952e+00}$ \\
    $\pgfmathprintnumber[std,fixed zerofill,sci zerofill,precision=0,1000 sep={}]{2.048000000000000e+03}$ & $\pgfmathprintnumber[std,std=-1:0,fixed zerofill,sci zerofill,precision=2,1000 sep={}]{5.000000000000000e-01}$ & $\pgfmathprintnumber[std,std=-1:0,fixed zerofill,sci zerofill,sci e,precision=4,1000 sep={}]{3.227321422570206e-03}$ & $\pgfmathprintnumber[std,std=1,fixed zerofill,sci zerofill,sci e,precision=4,1000 sep={}]{2.458770420576346e+00}$ \\
    $\pgfmathprintnumber[std,fixed zerofill,sci zerofill,precision=0,1000 sep={}]{3.072000000000000e+03}$ & $\pgfmathprintnumber[std,std=-1:0,fixed zerofill,sci zerofill,precision=2,1000 sep={}]{7.500000000000000e-01}$ & $\pgfmathprintnumber[std,std=-1:0,fixed zerofill,sci zerofill,sci e,precision=4,1000 sep={}]{3.227321422570206e-03}$ & $\pgfmathprintnumber[std,std=-1:0,fixed zerofill,sci zerofill,sci e,precision=4,1000 sep={}]{1.106256960537748e+01}$ \\
    $\pgfmathprintnumber[std,fixed zerofill,sci zerofill,precision=0,1000 sep={}]{4.096000000000000e+03}$ & $\pgfmathprintnumber[std,std=-1:0,fixed zerofill,sci zerofill,precision=2,1000 sep={}]{1.000000000000000e+00}$ & $\pgfmathprintnumber[std,std=-1:0,fixed zerofill,sci zerofill,sci e,precision=4,1000 sep={}]{3.349417418220924e-03}$ & $\pgfmathprintnumber[std,std=-1:0,fixed zerofill,sci zerofill,sci e,precision=4,1000 sep={}]{6.260904373762481e+01}$ \\
    \bottomrule
\end{tabular}
        \caption{$\lambda = 5$, $n = 4096$.}
        \label{tab:osc5n4096m64}
    \end{subtable}%
    \hfill
    \begin{subtable}{.49\textwidth}
        \centering
        % Required packages
% \usepackage{pgfplots}
% \usepackage{pgf}
% \usepackage{pgfplotstable}
% \pgfplotsset{compat = newest}

%%% Generated with pgfplotstable %%%

%%%%%%%%%%%%%%%%%%%%%%%%%%%%%%%
\begin{tabular}{cccc}
    \toprule
    $k$ & ${t}_{k}$ & $\max\limits_{1 \leq j \leq k} {E}_{128}\left ( {t}_{j} \right )$ & $\max\limits_{1 \leq j \leq k - 1} {\kappa}_{2}\left( \mathcal{B}_{j} \right)$ \\
    \midrule
    $\pgfmathprintnumber[std,fixed zerofill,sci zerofill,precision=0,1000 sep={}]{4.096000000000000e+03}$ & $\pgfmathprintnumber[std,std=-1:0,fixed zerofill,sci zerofill,precision=2,1000 sep={}]{2.500000000000000e-01}$ & $\pgfmathprintnumber[std,std=-1:0,fixed zerofill,sci zerofill,sci e,precision=4,1000 sep={}]{7.870652318171523e-03}$ & $\pgfmathprintnumber[std,std=1,fixed zerofill,sci zerofill,sci e,precision=4,1000 sep={}]{1.650537229840322e+00}$ \\
    $\pgfmathprintnumber[std,fixed zerofill,sci zerofill,precision=0,1000 sep={}]{8.192000000000000e+03}$ & $\pgfmathprintnumber[std,std=-1:0,fixed zerofill,sci zerofill,precision=2,1000 sep={}]{5.000000000000000e-01}$ & $\pgfmathprintnumber[std,std=-1:0,fixed zerofill,sci zerofill,sci e,precision=4,1000 sep={}]{7.870652318171523e-03}$ & $\pgfmathprintnumber[std,std=1,fixed zerofill,sci zerofill,sci e,precision=4,1000 sep={}]{5.720181662401454e+00}$ \\
    $\pgfmathprintnumber[std,fixed zerofill,sci zerofill,precision=0,1000 sep={}]{1.228800000000000e+04}$ & $\pgfmathprintnumber[std,std=-1:0,fixed zerofill,sci zerofill,precision=2,1000 sep={}]{7.500000000000000e-01}$ & $\pgfmathprintnumber[std,std=-1:0,fixed zerofill,sci zerofill,sci e,precision=4,1000 sep={}]{1.075988851052223e-02}$ & $\pgfmathprintnumber[std,std=-1:0,fixed zerofill,sci zerofill,sci e,precision=4,1000 sep={}]{3.078399685784205e+01}$ \\
    $\pgfmathprintnumber[std,fixed zerofill,sci zerofill,precision=0,1000 sep={}]{1.638400000000000e+04}$ & $\pgfmathprintnumber[std,std=-1:0,fixed zerofill,sci zerofill,precision=2,1000 sep={}]{1.000000000000000e+00}$ & $\pgfmathprintnumber[std,std=-1:0,fixed zerofill,sci zerofill,sci e,precision=4,1000 sep={}]{1.665762002670701e+02}$ & $\pgfmathprintnumber[std,std=-1:0,fixed zerofill,sci zerofill,sci e,precision=4,1000 sep={}]{4.334956306264361e+03}$ \\
    \bottomrule
\end{tabular}
        \caption{$\lambda = 17$, $n = 16384$.}
        \label{tab:osc17n16384m128}
    \end{subtable}%
    \hfill
    \begin{subtable}{.49\textwidth}
        \centering
        % Required packages
% \usepackage{pgfplots}
% \usepackage{pgf}
% \usepackage{pgfplotstable}
% \pgfplotsset{compat = newest}

%%% Generated with pgfplotstable %%%

%%%%%%%%%%%%%%%%%%%%%%%%%%%%%%%
\begin{tabular}{cccc}
    \toprule
    $k$ & ${t}_{k}$ & $\max\limits_{1 \leq j \leq k} {E}_{256}\left ( {t}_{j} \right )$ & $\max\limits_{1 \leq j \leq k - 1} {\kappa}_{2}\left( \mathcal{B}_{j} \right)$ \\
    \midrule
    $\pgfmathprintnumber[std,fixed zerofill,sci zerofill,precision=0,1000 sep={}]{1.638400000000000e+04}$ & $\pgfmathprintnumber[std,std=-1:0,fixed zerofill,sci zerofill,precision=2,1000 sep={}]{2.500000000000000e-01}$ & $\pgfmathprintnumber[std,std=-1:0,fixed zerofill,sci zerofill,sci e,precision=4,1000 sep={}]{2.441744870245177e-03}$ & $\pgfmathprintnumber[std,std=1,fixed zerofill,sci zerofill,sci e,precision=4,1000 sep={}]{1.089318876035231e+00}$ \\
    $\pgfmathprintnumber[std,fixed zerofill,sci zerofill,precision=0,1000 sep={}]{3.276800000000000e+04}$ & $\pgfmathprintnumber[std,std=-1:0,fixed zerofill,sci zerofill,precision=2,1000 sep={}]{5.000000000000000e-01}$ & $\pgfmathprintnumber[std,std=-1:0,fixed zerofill,sci zerofill,sci e,precision=4,1000 sep={}]{2.441744870245177e-03}$ & $\pgfmathprintnumber[std,std=1,fixed zerofill,sci zerofill,sci e,precision=4,1000 sep={}]{1.995546329546939e+00}$ \\
    $\pgfmathprintnumber[std,fixed zerofill,sci zerofill,precision=0,1000 sep={}]{4.915200000000000e+04}$ & $\pgfmathprintnumber[std,std=-1:0,fixed zerofill,sci zerofill,precision=2,1000 sep={}]{7.500000000000000e-01}$ & $\pgfmathprintnumber[std,std=-1:0,fixed zerofill,sci zerofill,sci e,precision=4,1000 sep={}]{2.441744870245177e-03}$ & $\pgfmathprintnumber[std,std=1,fixed zerofill,sci zerofill,sci e,precision=4,1000 sep={}]{8.247848042227183e+00}$ \\
    $\pgfmathprintnumber[std,fixed zerofill,sci zerofill,precision=0,1000 sep={}]{6.553600000000000e+04}$ & $\pgfmathprintnumber[std,std=-1:0,fixed zerofill,sci zerofill,precision=2,1000 sep={}]{1.000000000000000e+00}$ & $\pgfmathprintnumber[std,std=-1:0,fixed zerofill,sci zerofill,sci e,precision=4,1000 sep={}]{2.441744870245177e-03}$ & $\pgfmathprintnumber[std,std=-1:0,fixed zerofill,sci zerofill,sci e,precision=4,1000 sep={}]{4.719079110249369e+01}$ \\
    \bottomrule
\end{tabular}
        \caption{$\lambda = 17$, $n = 65536$.}
        \label{tab:osc17n65536m256}
    \end{subtable}
    \caption{The maximum absolute errors between the exact and approximate solutions of \hyperref[problem:test2]{\bf Test \ref*{problem:test2}} with respect to the temporal variable, as well as the maximum values of condition numbers of the matrices associated with the tridiagonal systems at the different temporal layers.}
    \label{tab:at-the-one-point}
\end{table}

Let us consider the test problem, which represents a homogeneous formulation of the problem \eqref{eq:main_eqt}-\eqref{eq:boundary_conds} under consideration. In this context, we have constructed an exact solution for this particular problem.
\begin{test}\label{problem:test3}
    \begin{alignat*}{2}
        {\psi}_{0}\left( x \right) &= {a}\sin\left( \frac{\lambda \pi}{\ell} x \right)\,,\quad {\psi}_{1}\left( x \right) &=& \frac{a}{2}\sin\left( \frac{\lambda \pi}{\ell} x \right)\,,\quad f\left( x,t \right) = 0\,,\\
        \alpha\left( t \right) &= \frac{{\ell}^{3} - b}{4 \ell {\lambda}^{2} {\pi}^{2} {\left( 1 + t \right)}^{2}}\,,\quad \beta\left( t \right) &=& \frac{b}{2 {a}^{2} {\lambda}^{4} {\pi}^{4} {\left( 1 + t \right)}^{3}}\,,\quad 0 < b < {\ell}^{3}\,.
    \end{alignat*}
\end{test}
$\displaystyle u\left( x,t \right) = {a}\sqrt{1 + t}\sin\left( \frac{\lambda \pi}{\ell} x \right)$ is an exact solution for the specific problem instance in \hyperref[problem:test3]{\bf Test \ref*{problem:test3}}. To numerically realize the problem, we adopt the following parameter values: the length of the string, $\ell = 1$; time interval, $T = 1$; $a = 1$; and $b = {\ell}^{3} / 2$.

In \hyperref[tab:Test3-errorTable]{\bf Table \ref*{tab:Test3-errorTable}}, we present the maximum absolute errors between the exact and approximate solutions of \hyperref[problem:test3]{\bf Test \ref*{problem:test3}}, computed using formula \eqref{eq:approx_error}. Similar to the previous problem in \hyperref[problem:test2]{\bf Test \ref*{problem:test2}}, we consider four cases with the same oscillation numbers $\lambda$, assuming $\tau$ to be equal to $h$. For the case when $\lambda = 5$, we use spatial grid sizes of $m = 16$ and $m = 32$. In the case of $\lambda = 17$, we double the spatial grid sizes, resulting in $m = 64$ and $m = 128$. As demonstrated in the \hyperref[tab:Test3-errorTable]{\bf Table \ref*{tab:Test3-errorTable}}, the maximum absolute errors, corresponding to the specified temporal steps, remain sufficiently small as long as the temporal and spatial grids are not excessively small. In the generated plots, the approximate solution closely approximates the exact solution, resulting in an overlapping representation. Hence, these plots have been omitted from the article.

One of the contributing factors to these small errors is the relatively low maximum absolute value of $q\left( t \right)$, which can be attributed to the time-varying coefficients $\alpha\left( t \right)$ and $\beta\left( t \right)$. Furthermore, since $\delta = 1$ and the maximum value of $q\left( t \right)$ is small, they ensure that the condition numbers of the coefficient matrices of the tridiagonal systems obtained for each temporal layer are excellent, being close to unity (see \eqref{eq:cond_number_estimate}). This, in turn, further enhances the accuracy of the solutions.

\begin{table}[H]
    \centering
    \begin{subtable}{.49\textwidth}
        \centering
        % Required packages
% \usepackage{pgfplots}
% \usepackage{pgf}
% \usepackage{pgfplotstable}
% \pgfplotsset{compat = newest}

%%% Generated with pgfplotstable %%%

%%%%%%%%%%%%%%%%%%%%%%%%%%%%%%%
\begin{tabular}{cccc}
    \toprule
    $k$ & ${t}_{k}$ & $\max\limits_{1 \leq j \leq k} {E}_{16}\left ( {t}_{j} \right )$ \\
    \midrule
    $\pgfmathprintnumber[std,fixed zerofill,sci zerofill,precision=0,1000 sep={}]{4.000000000000000e+00}$ & $\pgfmathprintnumber[std,std=-1:0,fixed zerofill,sci zerofill,precision=2,1000 sep={}]{2.500000000000000e-01}$ & $\pgfmathprintnumber[std,std=-1:0,fixed zerofill,sci zerofill,sci e,precision=4,1000 sep={}]{2.934656207118636e-04}$ \\
    $\pgfmathprintnumber[std,fixed zerofill,sci zerofill,precision=0,1000 sep={}]{8.000000000000000e+00}$ & $\pgfmathprintnumber[std,std=-1:0,fixed zerofill,sci zerofill,precision=2,1000 sep={}]{5.000000000000000e-01}$ & $\pgfmathprintnumber[std,std=-1:0,fixed zerofill,sci zerofill,sci e,precision=4,1000 sep={}]{1.394512595313646e-03}$ \\
    $\pgfmathprintnumber[std,fixed zerofill,sci zerofill,precision=0,1000 sep={}]{1.200000000000000e+01}$ & $\pgfmathprintnumber[std,std=-1:0,fixed zerofill,sci zerofill,precision=2,1000 sep={}]{7.500000000000000e-01}$ & $\pgfmathprintnumber[std,std=-1:0,fixed zerofill,sci zerofill,sci e,precision=4,1000 sep={}]{3.092005093163319e-03}$ \\
    $\pgfmathprintnumber[std,fixed zerofill,sci zerofill,precision=0,1000 sep={}]{1.600000000000000e+01}$ & $\pgfmathprintnumber[std,std=-1:0,fixed zerofill,sci zerofill,precision=2,1000 sep={}]{1.000000000000000e+00}$ & $\pgfmathprintnumber[std,std=-1:0,fixed zerofill,sci zerofill,sci e,precision=4,1000 sep={}]{5.245705034576220e-03}$ \\
    \bottomrule
\end{tabular}
        \caption{$\lambda = 5$, $n = 16$.}
        \label{tab:Test3_osc5_n16_m16}
    \end{subtable}%
    \hfill
    \begin{subtable}{.49\textwidth}
        \centering
        % Required packages
% \usepackage{pgfplots}
% \usepackage{pgf}
% \usepackage{pgfplotstable}
% \pgfplotsset{compat = newest}

%%% Generated with pgfplotstable %%%

%%%%%%%%%%%%%%%%%%%%%%%%%%%%%%%
\begin{tabular}{cccc}
    \toprule
    $k$ & ${t}_{k}$ & $\max\limits_{1 \leq j \leq k} {E}_{32}\left ( {t}_{j} \right )$ \\
    \midrule
    $\pgfmathprintnumber[std,fixed zerofill,sci zerofill,precision=0,1000 sep={}]{8.000000000000000e+00}$ & $\pgfmathprintnumber[std,std=-1:0,fixed zerofill,sci zerofill,precision=2,1000 sep={}]{2.500000000000000e-01}$ & $\pgfmathprintnumber[std,std=-1:0,fixed zerofill,sci zerofill,sci e,precision=4,1000 sep={}]{1.218007797700871e-04}$ \\
    $\pgfmathprintnumber[std,fixed zerofill,sci zerofill,precision=0,1000 sep={}]{1.600000000000000e+01}$ & $\pgfmathprintnumber[std,std=-1:0,fixed zerofill,sci zerofill,precision=2,1000 sep={}]{5.000000000000000e-01}$ & $\pgfmathprintnumber[std,std=-1:0,fixed zerofill,sci zerofill,sci e,precision=4,1000 sep={}]{4.944182004891218e-04}$ \\
    $\pgfmathprintnumber[std,fixed zerofill,sci zerofill,precision=0,1000 sep={}]{2.400000000000000e+01}$ & $\pgfmathprintnumber[std,std=-1:0,fixed zerofill,sci zerofill,precision=2,1000 sep={}]{7.500000000000000e-01}$ & $\pgfmathprintnumber[std,std=-1:0,fixed zerofill,sci zerofill,sci e,precision=4,1000 sep={}]{1.051797882890781e-03}$ \\
    $\pgfmathprintnumber[std,fixed zerofill,sci zerofill,precision=0,1000 sep={}]{3.200000000000000e+01}$ & $\pgfmathprintnumber[std,std=-1:0,fixed zerofill,sci zerofill,precision=2,1000 sep={}]{1.000000000000000e+00}$ & $\pgfmathprintnumber[std,std=-1:0,fixed zerofill,sci zerofill,sci e,precision=4,1000 sep={}]{1.750059916048707e-03}$ \\
    \bottomrule
\end{tabular}
        \caption{$\lambda = 5$, $n = 32$.}
        \label{tab:Test3_osc5_n32_m32}
    \end{subtable}%
    \hfill
    \begin{subtable}{.49\textwidth}
        \centering
        % Required packages
% \usepackage{pgfplots}
% \usepackage{pgf}
% \usepackage{pgfplotstable}
% \pgfplotsset{compat = newest}

%%% Generated with pgfplotstable %%%

%%%%%%%%%%%%%%%%%%%%%%%%%%%%%%%
\begin{tabular}{cccc}
    \toprule
    $k$ & ${t}_{k}$ & $\max\limits_{1 \leq j \leq k} {E}_{64}\left ( {t}_{j} \right )$ \\
    \midrule
    $\pgfmathprintnumber[std,fixed zerofill,sci zerofill,precision=0,1000 sep={}]{1.600000000000000e+01}$ & $\pgfmathprintnumber[std,std=-1:0,fixed zerofill,sci zerofill,precision=2,1000 sep={}]{2.500000000000000e-01}$ & $\pgfmathprintnumber[std,std=-1:0,fixed zerofill,sci zerofill,sci e,precision=4,1000 sep={}]{4.372569443242824e-04}$ \\
    $\pgfmathprintnumber[std,fixed zerofill,sci zerofill,precision=0,1000 sep={}]{3.200000000000000e+01}$ & $\pgfmathprintnumber[std,std=-1:0,fixed zerofill,sci zerofill,precision=2,1000 sep={}]{5.000000000000000e-01}$ & $\pgfmathprintnumber[std,std=-1:0,fixed zerofill,sci zerofill,sci e,precision=4,1000 sep={}]{1.627974093104445e-03}$ \\
    $\pgfmathprintnumber[std,fixed zerofill,sci zerofill,precision=0,1000 sep={}]{4.800000000000000e+01}$ & $\pgfmathprintnumber[std,std=-1:0,fixed zerofill,sci zerofill,precision=2,1000 sep={}]{7.500000000000000e-01}$ & $\pgfmathprintnumber[std,std=-1:0,fixed zerofill,sci zerofill,sci e,precision=4,1000 sep={}]{3.375062264494133e-03}$ \\
    $\pgfmathprintnumber[std,fixed zerofill,sci zerofill,precision=0,1000 sep={}]{6.400000000000000e+01}$ & $\pgfmathprintnumber[std,std=-1:0,fixed zerofill,sci zerofill,precision=2,1000 sep={}]{1.000000000000000e+00}$ & $\pgfmathprintnumber[std,std=-1:0,fixed zerofill,sci zerofill,sci e,precision=4,1000 sep={}]{5.547277533604733e-03}$ \\
    \bottomrule
\end{tabular}
        \caption{$\lambda = 17$, $n = 64$.}
        \label{tab:Test3_osc17_n64_m64}
    \end{subtable}%
    \hfill
    \begin{subtable}{.49\textwidth}
        \centering
        % Required packages
% \usepackage{pgfplots}
% \usepackage{pgf}
% \usepackage{pgfplotstable}
% \pgfplotsset{compat = newest}

%%% Generated with pgfplotstable %%%

%%%%%%%%%%%%%%%%%%%%%%%%%%%%%%%
\begin{tabular}{cccc}
    \toprule
    $k$ & ${t}_{k}$ & $\max\limits_{1 \leq j \leq k} {E}_{128}\left ( {t}_{j} \right )$ \\
    \midrule
    $\pgfmathprintnumber[std,fixed zerofill,sci zerofill,precision=0,1000 sep={}]{3.200000000000000e+01}$ & $\pgfmathprintnumber[std,std=-1:0,fixed zerofill,sci zerofill,precision=2,1000 sep={}]{2.500000000000000e-01}$ & $\pgfmathprintnumber[std,std=-1:0,fixed zerofill,sci zerofill,sci e,precision=4,1000 sep={}]{1.034233894383618e-04}$ \\
    $\pgfmathprintnumber[std,fixed zerofill,sci zerofill,precision=0,1000 sep={}]{6.400000000000000e+01}$ & $\pgfmathprintnumber[std,std=-1:0,fixed zerofill,sci zerofill,precision=2,1000 sep={}]{5.000000000000000e-01}$ & $\pgfmathprintnumber[std,std=-1:0,fixed zerofill,sci zerofill,sci e,precision=4,1000 sep={}]{3.803725249600376e-04}$ \\
    $\pgfmathprintnumber[std,fixed zerofill,sci zerofill,precision=0,1000 sep={}]{9.600000000000000e+01}$ & $\pgfmathprintnumber[std,std=-1:0,fixed zerofill,sci zerofill,precision=2,1000 sep={}]{7.500000000000000e-01}$ & $\pgfmathprintnumber[std,std=-1:0,fixed zerofill,sci zerofill,sci e,precision=4,1000 sep={}]{7.852195217923352e-04}$ \\
    $\pgfmathprintnumber[std,fixed zerofill,sci zerofill,precision=0,1000 sep={}]{1.280000000000000e+02}$ & $\pgfmathprintnumber[std,std=-1:0,fixed zerofill,sci zerofill,precision=2,1000 sep={}]{1.000000000000000e+00}$ & $\pgfmathprintnumber[std,std=-1:0,fixed zerofill,sci zerofill,sci e,precision=4,1000 sep={}]{1.287558497556462e-03}$ \\
    \bottomrule
\end{tabular}
        \caption{$\lambda = 17$, $n = 128$.}
        \label{tab:Test3_osc17_n128_m128}
    \end{subtable}
    \caption{The maximum absolute errors between the exact and approximate solutions of \hyperref[problem:test3]{\bf Test \ref*{problem:test3}}, specifically in relation to the temporal variable.}
    \label{tab:Test3-errorTable}
\end{table}

Consider the homogeneous form of the Kirchhoff string equation \eqref{eq:main_eqt} for which the exact solution is currently unknown. In this situation, we opt for specific parameter values: the length of the string, $\ell = 4$; the time interval, $T = 1$; and the time-dependent coefficients are fixed, with $\alpha\left( t \right) = \beta\left( t \right) = 1$.
\begin{test}\label{problem:test4}
    \begin{equation*}
        {\psi}_{0}\left( x \right) = \exp\left( -{\left( {2}{x} - {\ell} \right)}^{2} \right)\left( {\widetilde{P}}_{33}\left( x \right) - {\widetilde{P}}_{31}\left( x \right) \right)\,,\quad {\psi}_{1}\left( x \right) = 0\,,\quad f\left( x,t \right) = 0\,.
    \end{equation*}
    Here, ${\widetilde{P}}_{N}\left( x \right)$ denotes the shifted Legendre polynomials of degree $N$, defined as $\displaystyle {\widetilde{P}}_{N}\left( x \right) = {P}_{N}\left( \frac{2}{\ell} x - 1 \right)$, where ${P}_{N}\left( x \right)$ represents the Legendre polynomials of degree $N$. The difference of the shifted Legendre polynomials ensures the fulfilment of the compatibility conditions for ${\psi}_{0}\left( x \right)$, namely ${\psi}_{0}\left( 0 \right) = {\psi}_{0}\left( \ell \right) = 0$.
\end{test}
\begin{figure}[H]
    \centering
    \includegraphics[width=0.49\linewidth]{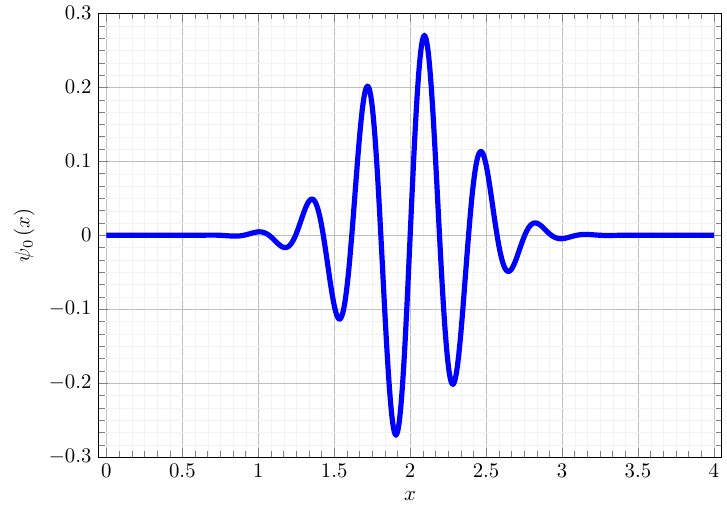}
    \caption{Graph of the initial condition ${\psi}_{0}\left( x \right)$.}
    \label{fig:plot_graph_of_psi0}
\end{figure}
Before we begin explaining the approach used to derive the solution for this particular scenario, let us introduce the essential notations required for better understanding.

Let
\begin{equation*}
    u\left( {x}_{i},{t}_{k}^{\left( j \right)} \right) \approx {u}_{k,i}^{\left( j \right)}\,,\quad i = 0,1,\ldots,m\,,\quad k = 0,1,\ldots,{n}_{j}\,,
\end{equation*}
where
\begin{equation*}
    {n}_{j} = {2}^{j} {n}_{0}\,,\quad {t}_{k}^{\left( j \right)} = {k}{\tau}_{j}\,,\quad {\tau}_{j} = \frac{T}{{n}_{j}}\,.
\end{equation*}
We propose the following criteria to assess the errors between every two successive approximate solutions for $j = 1,2,\ldots\,$, namely:
\begin{equation*}
    {\widetilde{E}}_{m}^{\left( j \right)} = \max\limits_{0 \leq k \leq {n}_{j}} \max\limits_{0 \leq i \leq m}\left|  {u}_{k,i}^{\left( j - 1 \right)} - {u}_{2k,i}^{\left( j \right)} \right|\,.
\end{equation*}
Suppose that approximate solutions ${u}_{k,i}^{\left( 0 \right)}$ and ${u}_{2k,i}^{\left( 1 \right)}$ are found, and an error tolerance $tol$ is fixed. If ${\widetilde{E}}_{m}^{\left( j \right)} \leq tol$, then terminate the process and return the solution ${u}_{2k,i}^{\left( j \right)}$.

We have implemented six different cases, where the error tolerance is defined as $tol = 10^{-r}$ with $r = 1,2,\ldots,6$. The spatial domain is discretized into $m = 1024$ segments, and the initial division of the temporal domain is prescribed as ${n}_{0} = 128$. The numerical results are compiled in \hyperref[tab:ErrorTolerance]{\bf Table \ref*{tab:ErrorTolerance}}. Notably, the obtained numerical solutions for all six cases exhibit remarkable similarity through their respective plots, providing convincing evidence of effectively approximating the exact solution to the given problem.
\begin{table}[H]
    \centering
    % Required packages
% \usepackage{pgfplots}
% \usepackage{pgf}
% \usepackage{pgfplotstable}
% \pgfplotsset{compat = newest}

%%% Generated with pgfplotstable %%%

%%%%%%%%%%%%%%%%%%%%%%%%%%%%%%%
\begin{tabular}{cccc}
    \toprule
    $tol$ & $\max\left( {n}_{j} \right)$ & $\max\left( j \right)$ & ${\widetilde{E}}_{m}^{\left( j \right)}$ \\
    \midrule
    $\pgfmathprintnumber[std,fixed zerofill,sci zerofill,sci e,precision=0,1000 sep={}]{1e-01}$ & ${512}$   & ${2}$ & $\pgfmathprintnumber[std,std=-1:0,fixed zerofill,sci zerofill,sci e,precision=4,1000 sep={}]{3.679055412531656e-02}$ \\
    $\pgfmathprintnumber[std,fixed zerofill,sci zerofill,sci e,precision=0,1000 sep={}]{1e-02}$ & ${1024}$  & ${3}$ & $\pgfmathprintnumber[std,std=-1:0,fixed zerofill,sci zerofill,sci e,precision=4,1000 sep={}]{8.992767516970554e-03}$ \\
    $\pgfmathprintnumber[std,fixed zerofill,sci zerofill,sci e,precision=0,1000 sep={}]{1e-03}$ & ${4096}$  & ${5}$ & $\pgfmathprintnumber[std,std=-1:0,fixed zerofill,sci zerofill,sci e,precision=4,1000 sep={}]{2.186109009167927e-04}$ \\
    $\pgfmathprintnumber[std,fixed zerofill,sci zerofill,sci e,precision=0,1000 sep={}]{1e-04}$ & ${8192}$  & ${6}$ & $\pgfmathprintnumber[std,std=-1:0,fixed zerofill,sci zerofill,sci e,precision=4,1000 sep={}]{1.097694028723284e-05}$ \\
    $\pgfmathprintnumber[std,fixed zerofill,sci zerofill,sci e,precision=0,1000 sep={}]{1e-05}$ & ${16384}$ & ${7}$ & $\pgfmathprintnumber[std,std=-1:0,fixed zerofill,sci zerofill,sci e,precision=4,1000 sep={}]{1.815425744539079e-06}$ \\
    $\pgfmathprintnumber[std,fixed zerofill,sci zerofill,sci e,precision=0,1000 sep={}]{1e-06}$ & ${32768}$ & ${8}$ & $\pgfmathprintnumber[std,std=-1:0,fixed zerofill,sci zerofill,sci e,precision=4,1000 sep={}]{6.995203757007018e-07}$ \\
    \bottomrule
\end{tabular}
    \caption{Error tolerance and computed maximum temporal layer with corresponding loop counter.}
    \label{tab:ErrorTolerance}
\end{table}

For illustrative purposes, we present \hyperref[fig:UnkSol-FourSpecific]{\bf Figure \ref*{fig:UnkSol-FourSpecific}} under the condition that the error tolerance is $tol = {10}^{-6}$. The numerical approximation of the solution is assessed at four specific temporal instances, namely: $t = 0.25$, $t = 0.5$, $t = 0.75$, and $t = 1$.

In \hyperref[tab:UnkSol-CondNumbs]{\bf Table \ref*{tab:UnkSol-CondNumbs}}, we present data corresponding to the error tolerance $tol = {10}^{-6}$, along with the respective number of divisions of the temporal domain ${n}_{j}$, at which the process halts with ${\widetilde{E}}_{m}^{\left( j \right)} \leq {10}^{-6}$. Additionally, \hyperref[tab:UnkSol-CondNumbs]{\bf Table \ref*{tab:UnkSol-CondNumbs}} includes the maximum condition numbers denoted as $\max_{1 \leq k \leq {n}_{j} - 1} {\kappa}_{2}\left( \mathcal{B}_{k} \right)$ for each maximum temporal layer. Let us remind you that our method involves the determination of an approximate solution for each $\left( k + 1 \right)$-th layer, wherein a system of linear equations \eqref{eq:tridiag_system_w} is solved based on the preceding $k$-th layer, with $k$ taking values from 1 to ${n}_{j} - 1$. Furthermore, the resulting condition numbers conform to the estimate provided in \eqref{eq:cond_number_estimate}.
\begin{table}[H]
    \centering
    % Required packages
% \usepackage{pgfplots}
% \usepackage{pgf}
% \usepackage{pgfplotstable}
% \pgfplotsset{compat = newest}

%%% Generated with pgfplotstable %%%

%%%%%%%%%%%%%%%%%%%%%%%%%%%%%%%
\begin{tabular}{cccc}
    \toprule
    $j$ & ${\widetilde{E}}_{m}^{\left( j \right)}$ & ${n}_{j}$ & $\max\limits_{1 \leq k \leq {n}_{j} - 1} {\kappa}_{2}\left( \mathcal{B}_{k} \right)$ \\
    \midrule
    $0$ & $-$ & $\pgfmathprintnumber[std,fixed zerofill,sci zerofill,precision=0,1000 sep={}]{1.280000000000000e+02}$ & $\pgfmathprintnumber[std,std=-1:0,fixed zerofill,sci zerofill,sci e,precision=6,1000 sep={}]{5.754131612820781e+01}$ \\
    $1$ & $\pgfmathprintnumber[std,std=1,fixed zerofill,sci zerofill,sci e,precision=6,1000 sep={}]{1.391002508173658e-01}$ & $\pgfmathprintnumber[std,fixed zerofill,sci zerofill,precision=0,1000 sep={}]{2.560000000000000e+02}$ & $\pgfmathprintnumber[std,std=-1:0,fixed zerofill,sci zerofill,sci e,precision=6,1000 sep={}]{1.549921405644663e+01}$ \\
    $2$ & $\pgfmathprintnumber[std,std=-1:0,fixed zerofill,sci zerofill,sci e,precision=6,1000 sep={}]{3.679055412531656e-02}$ & $\pgfmathprintnumber[std,fixed zerofill,sci zerofill,precision=0,1000 sep={}]{5.120000000000000e+02}$ & $\pgfmathprintnumber[std,std=1,fixed zerofill,sci zerofill,sci e,precision=6,1000 sep={}]{4.482612604611548e+00}$ \\
    $3$ & $\pgfmathprintnumber[std,std=-1:0,fixed zerofill,sci zerofill,sci e,precision=6,1000 sep={}]{8.992767516970554e-03}$ & $\pgfmathprintnumber[std,fixed zerofill,sci zerofill,precision=0,1000 sep={}]{1.024000000000000e+03}$ & $\pgfmathprintnumber[std,std=1,fixed zerofill,sci zerofill,sci e,precision=6,1000 sep={}]{1.705540971213571e+00}$ \\
    $4$ & $\pgfmathprintnumber[std,std=-1:0,fixed zerofill,sci zerofill,sci e,precision=6,1000 sep={}]{1.800673629236365e-03}$ & $\pgfmathprintnumber[std,fixed zerofill,sci zerofill,precision=0,1000 sep={}]{2.048000000000000e+03}$ & $\pgfmathprintnumber[std,std=1,fixed zerofill,sci zerofill,sci e,precision=6,1000 sep={}]{1.099434809122676e+00}$ \\
    $5$ & $\pgfmathprintnumber[std,std=-1:0,fixed zerofill,sci zerofill,sci e,precision=6,1000 sep={}]{2.186109009167927e-04}$ & $\pgfmathprintnumber[std,fixed zerofill,sci zerofill,precision=0,1000 sep={}]{4.096000000000000e+03}$ & $\pgfmathprintnumber[std,std=1,fixed zerofill,sci zerofill,sci e,precision=6,1000 sep={}]{1.009048738473097e+00}$ \\
    $6$ & $\pgfmathprintnumber[std,std=-1:0,fixed zerofill,sci zerofill,sci e,precision=6,1000 sep={}]{1.097694028723284e-05}$ & $\pgfmathprintnumber[std,fixed zerofill,sci zerofill,precision=0,1000 sep={}]{8.192000000000000e+03}$ & $\pgfmathprintnumber[std,std=1,fixed zerofill,sci zerofill,sci e,precision=6,1000 sep={}]{1.000638283749081e+00}$ \\
    $7$ & $\pgfmathprintnumber[std,std=-1:0,fixed zerofill,sci zerofill,sci e,precision=6,1000 sep={}]{1.815425744539079e-06}$ & $\pgfmathprintnumber[std,fixed zerofill,sci zerofill,precision=0,1000 sep={}]{1.638400000000000e+04}$ & $\pgfmathprintnumber[std,std=1,fixed zerofill,sci zerofill,sci e,precision=6,1000 sep={}]{1.000041217900318e+00}$ \\
    $8$ & $\pgfmathprintnumber[std,std=-1:0,fixed zerofill,sci zerofill,sci e,precision=6,1000 sep={}]{6.995203757007018e-07}$ & $\pgfmathprintnumber[std,fixed zerofill,sci zerofill,precision=0,1000 sep={}]{3.276800000000000e+04}$ & $\pgfmathprintnumber[std,std=1,fixed zerofill,sci zerofill,sci e,precision=6,1000 sep={}]{1.000002597690866e+00}$ \\
    \bottomrule
\end{tabular}
    \caption{The results for the error tolerance $tol = {10}^{-6}$, including its corresponding number of divisions of the temporal domain and the maximum values of condition numbers.}
    \label{tab:UnkSol-CondNumbs}
\end{table}

\begin{figure}[H]
    \centering
    \begin{subfigure}{.49\textwidth}
        \centering
        \includegraphics[width=1\linewidth]{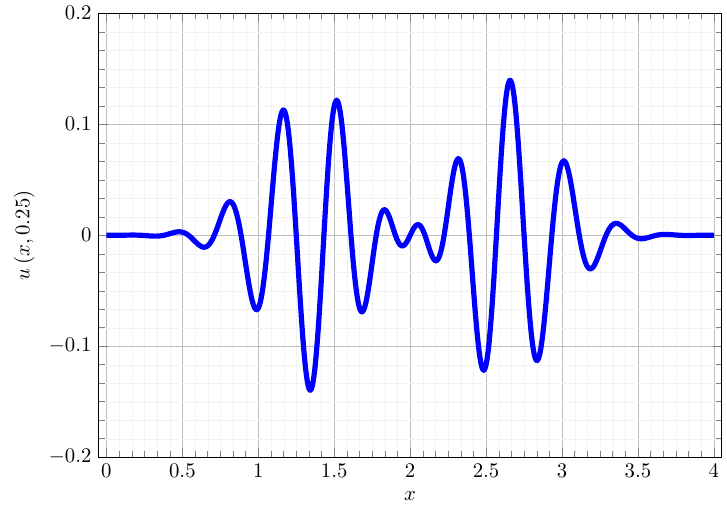}
        \caption{$t = 0.25$, $m = 1024$, and $n = 32768$.}
        \label{fig:UnkSol-Time0.25}
    \end{subfigure}
    \hfill
    \begin{subfigure}{.49\textwidth}
        \centering
        \includegraphics[width=1\linewidth]{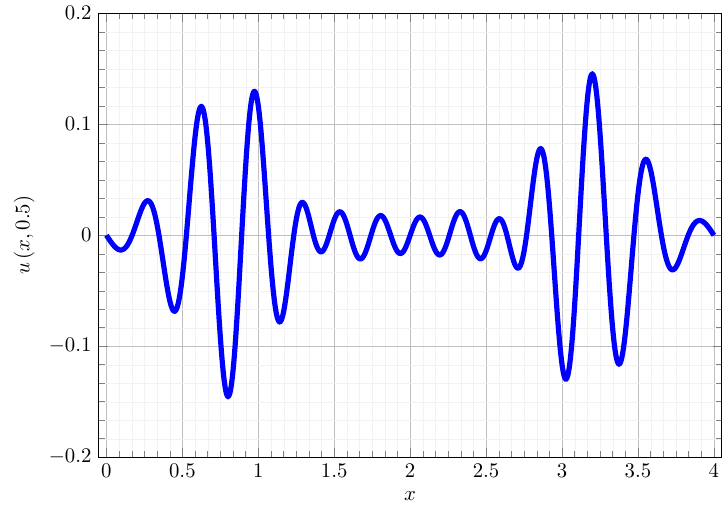}
        \caption{$t = 0.5$, $m = 1024$, and $n = 32768$.}
        \label{fig:UnkSol-Time0.50}
    \end{subfigure}
    \hfill
    \begin{subfigure}{.49\textwidth}
        \centering
        \includegraphics[width=1\linewidth]{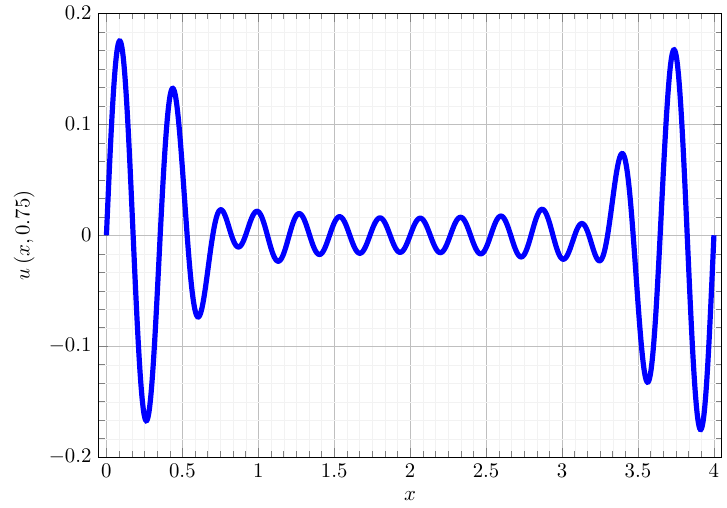}
        \caption{$t = 0.75$, $m = 1024$, and $n = 32768$.}
        \label{fig:UnkSol-Time0.75}
    \end{subfigure}
    \hfill
    \begin{subfigure}{.49\textwidth}
        \centering
        \includegraphics[width=1\linewidth]{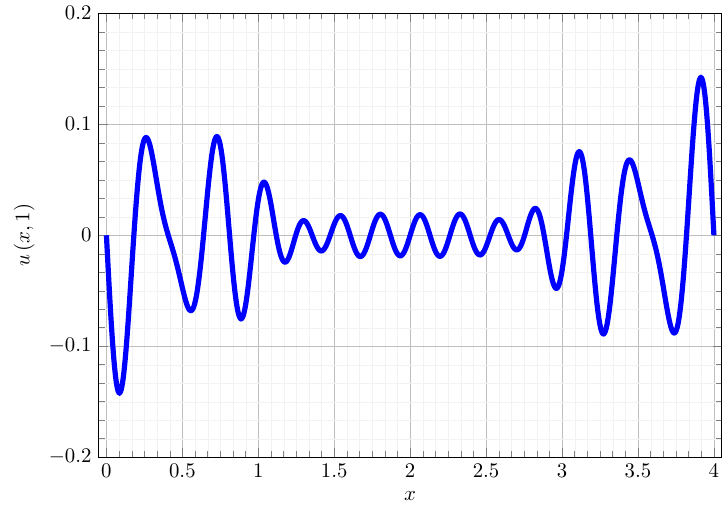}
        \caption{$t = 1$, $m = 1024$, and $n = 32768$.}
        \label{fig:UnkSol-Time1.00}
    \end{subfigure}
    \caption{Graph plots illustrate the numerical solutions to the given problem at four specific temporal instances, where the error tolerance is set to $tol = {10}^{-6}$.}
    \label{fig:UnkSol-FourSpecific}
\end{figure}

\section*{Funding}\label{sec:funding}
\noindent The second author of this article was supported by the Shota Rustaveli National Science Foundation of Georgia (SRNSFG) [grant number: FR-21-301, project title: ``Metamaterials with Cracks and Wave Diffraction Problems''].

\section*{Acknowledgement}\label{sec:acknowledgement}
\noindent The authors of this article would like to express their gratitude to the referees for their detailed consideration of our work and for providing helpful remarks. Taking these remarks into consideration played an important role in refining this paper. We would also like to extend our special appreciation to Dr. Andreas A. Buchheit for his invaluable contributions to our research. His expert insights into the practical implications and physical interpretations of the discussed problem, along with his extensive discussions and crucial recommendations, notably enhanced the overall quality and depth of our article. We also sincerely thank Prof. Dr. Menno Poot for sparing his time to engage in a thoughtful discussion regarding the problem under consideration.

\def\printchapternonum{}
\bibliographystyle{plainnat}
\bibliography{bibsource}

\end{document}